\documentclass{amsart}

\usepackage{hyperref}
\usepackage{amsmath,amssymb,amsthm}
\usepackage[mathscr]{eucal}

\title[Maxwell on Kerr]{Uniform energy bound and asymptotics for the Maxwell field on a slowly rotating Kerr black hole exterior}
\author[L. Andersson]{Lars Andersson}%${}^\dagger$}
\email{laan@aei.mpg.de}
\address{Albert Einstein Institute, Am M\"uhlenberg 1, D-14476 Potsdam,
Germany}

\author[P. Blue]{Pieter Blue}%${}^\ddagger$}
\email{P.Blue@ed.ac.uk}
\address{The School of Mathematics and the Maxwell Institute, University
of Edinburgh,James Clerk Maxwell Building,
The King's Buildings,
Mayfield Road,
Edinburgh,
Scotland EH9 3JZ, UK}

\date{February 2014}

\newcommand{\Reals}{\mathbb{R}}
\newcommand{\Complex}{\mathbb{C}}
\newcommand{\Naturals}{\mathbb{N}}
\renewcommand{\Re}{\mathrm{Re}}
\renewcommand{\Im}{\mathrm{Im}}
\newcommand{\supp}{\mathrm{supp}}

\newcommand{\cf}{cf.\ }
\newcommand{\eg}{e.g.\ }
\newcommand{\ie}{i.e.\ }
\newcommand{\apriori}{{a priori}}

\theoremstyle{plain}
\newtheorem{theorem}{Theorem}[section]
\newtheorem{corollary}[theorem]{Corollary}
\newtheorem{lemma}[theorem]{Lemma}

\newtheorem{definition}[theorem]{Definition}
\newtheorem{prop}[theorem]{Proposition}
\newtheorem{remark}[theorem]{Remark}

\newcommand{\vecFont}[1]{{\mathbf{#1}}}
\newcommand{\vecX}{\vecFont{X}}
\newcommand{\vecY}{\vecFont{Y}}
\newcommand{\Tnormalised}{\hat{\vecFont{T}}_{\text{PNV}}}
\newcommand{\Rnormalised}{\hat{\vecFont{R}}}
\newcommand{\Hnormalised}{\hat{\vecFont{\Theta}}}
\newcommand{\Pnormalised}{\hat{\vecFont{\Phi}}_{\text{PNV}}}
\newcommand{\vecLPNV}{\vecFont{l}_{\text{PNV}}}
\newcommand{\vecNPNV}{\vecFont{n}_{\text{PNV}}}
\newcommand{\vecMPNV}{\vecFont{m}_{\text{PNV}}}
\newcommand{\vecMbPNV}{\bar{\vecFont{m}}_{\text{PNV}}}
\newcommand{\vecH}{\vecFont{\Theta}}
\newcommand{\vecPPNV}{\vecFont{\Phi}_{\text{PNV}}}
\newcommand{\vecTperp}{\vecFont{T}_{\perp}}
\newcommand{\pt}{\partial_t}
\newcommand{\pr}{\partial_r}
\newcommand{\ph}{\partial_\theta}
\newcommand{\pp}{\partial_\phi}
\newcommand{\pAng}{\not\!\nabla}

\newcommand{\vecTBlend}{\vecFont{T}_\chi}
\newcommand{\vecTPerp}{\vecFont{T}_\perp}
\newcommand{\fnBlend}{\chi}
\newcommand{\vecTPNV}{\vecFont{T}_\text{PNV}}
\newcommand{\omegaBlend}{\omega_\chi}
\newcommand{\omegaPerp}{\omega_\perp}
\newcommand{\omegaPNV}{\omega_\text{PNV}}

\newcommand{\rp}{r_+}
\newcommand{\KDelta}{\Delta}
\newcommand{\KSigma}{\Sigma}
\newcommand{\KPi}{\Pi}
\newcommand{\Kp}{p}
\newcommand{\potlL}{V_{L}}
\newcommand{\potlFI}{V_{FI}}
\newcommand{\CarterQ}{\hat{Q}}
\newcommand{\OpQ}{\mathcal{Q}}

\newcommand{\MaxF}{\textrm{F}}
\newcommand{\MaxFStatic}{\textrm{F}_{\text{stationary}}}
\newcommand{\MaxFTotal}{\textrm{F}_{\text{total}}}
\newcommand{\phip}{\phi_1}
\newcommand{\phiz}{\phi_0}
\newcommand{\phim}{\phi_{-1}}
\newcommand{\phipm}{\phi_{\pm1}}
\newcommand{\phii}{\phi_i}
\newcommand{\solu}{\Upsilon}
\newcommand{\solub}{\bar{\Upsilon}}
\newcommand{\soluStatic}{\Upsilon_{\text{stationary}}}
\newcommand{\soluTotal}{\Upsilon_{\text{total}}}
\newcommand{\Electric}{\vecFont{E}}
\newcommand{\Magnetic}{\vecFont{B}}
\newcommand{\ElectricStatic}{\vecFont{E}_{\text{stationary}}}
\newcommand{\MagneticStatic}{\vecFont{B}_{\text{stationary}}}
\newcommand{\ChargeConstant}{q}
\newcommand{\ChargeE}{q_E}
\newcommand{\ChargeB}{q_B}

\newcommand{\phiUglyp}{\phi_{NP,0}}
\newcommand{\phiUglyz}{\phi_{NP,1}}
\newcommand{\phiUglym}{\phi_{NP,2}}

\newcommand{\CKalpha}{\alpha}
\newcommand{\CKrho}{\rho}
\newcommand{\CKsigma}{\sigma}
\newcommand{\CKalphab}{\underline{\alpha}}
\newcommand{\CKBundlepm}{V_{\pm1}}

\newcommand{\EnergyCrudeF}{E_{}[\MaxF]}
\newcommand{\EnergyCrudeFI}{E_{\text{FI}}[\solu]}
\newcommand{\EnergyCrudeFArg}[1]{E_{}[#1]}
\newcommand{\EnergyCrudeFIArg}[1]{E_{\text{FI}}[#1]}
\newcommand{\BulkMorawetzF}{B_{\pm}}
\newcommand{\BulkFILowerOrder}{B_{0}}
\newcommand{\BulkFIOne}{B_1}
\newcommand{\BulkFITwo}{B_{2,0}}

\newcommand{\inthst}[1]{\int_{\hst{#1}}}
\newcommand{\intBulk}[1]{\int_0^{#1}\inthst{t}}
\newcommand{\intSSurface}{\int_{\TwoSfc}}
\newcommand{\di}{\mathrm{d}}
\newcommand{\diNormal}[2]{\vecFont{n}^{#1}\di^3\mu_{#2}}
\newcommand{\diThreeNaive}{r^2\di r\di\omega}
\newcommand{\diThreeNaiveFI}{\di r\di\omega}
\newcommand{\diomega}{\sin\theta\di\theta\di\phi}
\newcommand{\diFourNatural}{\di^4\mu}

\newcommand{\diFourFI}{\di^4\mu_{\mathrm{FI}}}
\newcommand{\diFourSpectral}{\di\tilde{\mu}_{\mathrm{FI}}}

\newcommand{\Hodge}{*}
\newcommand{\insertion}[1]{i_{#1}}
\newcommand{\LeviCivita}{\epsilon}
\newcommand{\Lie}{\mathcal{L}}

\newcommand{\epsilonMain}{\frac{|a|}{M}}
\newcommand{\epsilonSlowRotationIntro}{\bar{\epsilon}_a}
\newcommand{\ConstantIntro}{C}

\newcommand{\Exterior}{\mathcal{M}}
\newcommand{\BulkZone}[2]{\Omega[#1,#2]}
\newcommand{\hst}[1]{\Sigma_{#1}}
\newcommand{\hsr}[1]{\tilde{\Sigma}_{#1}}
\newcommand{\bif}{\hsr{\rp}}
\newcommand{\spacelikeinfty}{\hsr{\infty}}

\newcommand{\TwoSfc}{\mathcal{S}^2}
\newcommand{\Stwo}{\mathbb{S}^2}

\newcommand{\KgMetric}{g}
\newcommand{\gMetric}{g}

\newcommand{\rfar}{r_1}  %To be the border between near and far.
\newcommand{\chirfar}{\chi_{|r-3M|\geq\rfar}}

\newcommand{\MaxH}{H}
\newcommand{\psii}{\psi_i}
\newcommand{\psiz}{\psi_0}

\newcommand{\Lagrangian}{\mathrm{L}}
\newcommand{\EMS}{\mathrm{T}}
\newcommand{\GenMomentum}[1]{\mathrm{P}_{(#1)}}
\newcommand{\GenEnergy}[1]{\mathrm{E}_{#1}}
\newcommand{\GenBulk}[1]{\mathrm{Bulk}_#1}
\newcommand{\vecMorawetzF}{\vecFont{A}}
\newcommand{\fnMorawetzF}{f}
\newcommand{\PDmat}{\mathrm{P}}
\newcommand{\geodesice}{e}
\newcommand{\geodesiclz}{\ell_z}
\newcommand{\geodesicQ}{Q}
\newcommand{\DCurlyRT}{\tilde{\CurlyR}'}
\newcommand{\DDCurlyRTT}{\tilde{\tilde{\CurlyR}}''}
\newcommand{\fnMorawetzS}{\mathcal{F}}
\newcommand{\fnMorawetzSA}{f_{1}}
\newcommand{\fnMorawetzSB}{f_{2}}
\newcommand{\fnMorawetzSC}{\mathcal{F}_{3}} %Why is \mathcal{f} ``}''?
\newcommand{\fnMorawetzSAA}{f_{1,1}}
\newcommand{\fnMorawetzSAB}{f_{1,2}}
\newcommand{\fnMorawetzSBA}{f_{2,1}}
\newcommand{\fnMorawetzSBB}{f_{2,2}}
\newcommand{\chimid}{\chi_{\text{mid}}}
\newcommand{\fnMorawetzSCA}{\mathcal{F}_{3,1}}
\newcommand{\fnMorawetzSCB}{f_{3,2}}
\newcommand{\vecMorawetzBasic}{\vecFont{A}}
\newcommand{\fnqS}{q}
\newcommand{\spectralt}{e}
\newcommand{\spectralp}{\ell_z}
\newcommand{\spectralQ}{Q}
\newcommand{\epsilondtsquared}{\epsilon_{\pt^2}}
\newcommand{\CurlyR}{\mathcal{R}}
\newcommand{\rorbit}{r_{\text{root}}}
\newcommand{\FinalTime}{T}
\newcommand{\ChiT}{\tilde{\chi}_{[0,\FinalTime]}}
\newcommand{\soluCut}{\solu_{\chi}}
\newcommand{\soluCutb}{{\bar\solu}_{\chi}}
\newcommand{\soluPartialFourier}{\hat{\solu}}  %why is this used?
\newcommand{\spectralcollection}{k}
\newcommand{\soluS}{\tilde{\solu}}
\newcommand{\soluSb}{\bar{\tilde{\solu}}}
\newcommand{\spectralspace}{\mathcal{K}}
\newcommand{\fullspectralspace}{(\rp,\infty)\times\spectralspace}
\newcommand{\dispectral}{\di\spectralcollection}  %put with other measures?
\newcommand{\potlFIz}{V_{FI,0}}
\newcommand{\ErrorSCut}{\tilde{J}_{\chi}}
\newcommand{\ErrorSIm}{\tilde{J}_{\Im}}
\newcommand{\EnergyS}{\mathrm{E}_{(\fnMorawetzS,\fnqS)}}
\newcommand{\BulkSMain}{\mathrm{Bulk}_{\mathrm{main}}}
\newcommand{\BulkSIm}{\mathrm{Bulk}_{\mathrm{Im}}}
\newcommand{\EMSSpec}{\mathrm{T}_{rr}}
\newcommand{\GenMomentumS}{\mathrm{P}_r}
\newcommand{\epsilondtsquaredMax}{\bar{\epsilon}_{\partial_t^2}}
\newcommand{\modspectralcollection}{|\spectralcollection|_{\epsilondtsquared}}
\newcounter{MyStep}
\newcommand{\StartStep}{\setcounter{MyStep}{0}}
\newcommand{\Step}[1]{\noindent\stepcounter{MyStep}\textbf{Step \theMyStep: #1:}}
\newcommand{\smallparameter}{\max\left(\frac{|a|}{\epsilondtsquared},\frac{\epsilondtsquared}{M}\right)}  %a/M follows from the others
\newcommand{\Hardya}{\mathtt{a}}
\newcommand{\Hardyb}{\mathtt{b}}
\newcommand{\Hardyc}{\mathtt{c}}
\newcommand{\Hardyd}{\mathtt{d}}
\newcommand{\Hardyalpha}{\alpha}
\newcommand{\Hardybeta}{\beta}
\newcommand{\HardyX}{X}
\newcommand{\HardyY}{Y}
\newcommand{\HardyZ}{Z}
\newcommand{\HardyprCoeff}{A}
\newcommand{\HardyPotl}{V_\mathrm{Hardy}}
\newcommand{\HardyPotlRed}{W}
\newcommand{\HardysoluRed}{\psi}
\newcommand{\HardyODEsolu}{v}
\newcommand{\HardyvarRed}{x}
\newcommand{\partialHardy}{\partial_\HardyvarRed}
\newcommand{\smallparameterwithrfar}{\max\left(\frac{a}{\epsilondtsquared},\frac{a}{M},\frac{\epsilondtsquared}{M},\frac{\rfar}{M}\right)}
\newcommand{\fnMorawetzSApprox}{\mathcal{F}_{\text{approx}}}
\newcommand{\fnqSApprox}{q_{\text{approx}}}

\newcommand{\ExponentRefined}{1/2}
\newcommand{\ExponentRefinedk}{k/2}
\newcommand{\ExponentRefinedThree}{3/2}
\newcommand{\chiTworfar}{\chi_{|r-3M|<2\rfar}}

\newcommand{\rescaledr}{\modspectralcollection^{\ExponentRefined}M^{-1}(r-\rorbit)}
\newcommand{\modrescaledr}{\modspectralcollection^{\ExponentRefined}M^{-1}|r-\rorbit|}
\newcommand{\chiThreerfar}{\chi_{|r-3M|<3\rfar}}

\newcommand{\rfarratio}{\bar{\epsilon}_1}

\newcommand{\hypothesisA}{For $|a|/M$ sufficiently small as in definition \ref{def:asmall}}
\newcommand{\hypothesisB}{\hypothesisA}
\newcommand{\hypothesisC}{For any projection away from the orbiting null geodesics and $|a|/M$ sufficiently small as in definition \ref{def:anyprojectionaway}}
\newcommand{\hypothesisD}{For $\max(|a|/\epsilondtsquared,\epsilondtsquared/M)$ sufficiently small as in definition \ref{def:epsilondtsquaredalsosmall}}
\newcommand{\hypothesisE}{There is a projection away from the orbiting null geodesic such that for $\max(|a|/\epsilondtsquared,\epsilondtsquared/M)$ sufficiently small as in definition \ref{def:chooseprojection}}
\newcommand{\hypothesisF}{For any projection away from the orbiting null geodesics and $\max(|a|/\epsilondtsquared,\epsilondtsquared/M)$ sufficiently small as in definition \ref{def:anyprojectionawayalso}}

\newcommand{\hypothesisFMax}{if $\MaxF$ is a regular solution of the Maxwell equation \eqref{eq:Maxwell}}
\newcommand{\hypothesisFsolu}{if $\MaxF$ is a regular solution of the Maxwell equation \eqref{eq:Maxwell} and $\solu=\solu[\MaxF]$}

\newcommand{\hypothesisFsoluS}{if $\MaxF$ is a regular, charge-free solution of the Maxwell equation \eqref{eq:Maxwell}, $\solu=\solu[\MaxF]$,  $\FinalTime>0$, and $\soluS$ is the $\FinalTime$-spectral transform of $\solu$}
\newcommand{\hypothesisFsoluSNoCharge}{if $\MaxF$ is a regular, charge-free solution of the Maxwell equation \eqref{eq:Maxwell}, $\solu=\solu[\MaxF]$, $\FinalTime>0$, and $\soluS$ is the $\FinalTime$-spectral transform of $\solu$}

\begin{document}
\begin{abstract}
We consider the Maxwell equation in the exterior of a very slowly rotating Kerr black hole. For this system, we prove the boundedness of a positive definite energy on each hypersurface of constant $t$. We also prove the convergence of each solution to a stationary Coulomb solution. We separate a general solution into the charged, Coulomb part and the uncharged part. Convergence to the Coulomb solutions follows from the fact that the uncharged part satisfies a Morawetz estimate, \ie that a spatially localised energy density is integrable in time. For the unchanged part, we study both the full Maxwell equation and the Fackerell-Ipser equation for one component. To treat the Fackerell-Ipser equation, we use a Fourier transform in $t$. For the Fackerell-Ipser equation, we prove a refined Morawetz estimate that controls $3/2$ derivatives with no loss near the orbiting null geodesics. 
%We consider the Maxwell equation in the exterior of a very slowly rotating Kerr black hole. For this system, we prove the boundedness of a positive definite energy on each hypersurface of constant $t$. We also prove a Morawetz estimate, \ie{} that a spatially localised energy density is integrable in time. We separate the charged and uncharged solutions For the unchanged part, we study both the full Maxwell equation and the Fackerell-Ipser equation for one component. To treat the Fackerell-Ipser equation, we use a Fourier transform in $t$. For the Fackerell-Ipser equation, we prove a refined Morawetz estimate that controls $3/2$ derivatives with no loss near the orbiting null geodesics. 
\end{abstract}

\maketitle

\tableofcontents

%\begin{quote}
%{\footnotesize \textit{``I used to believe that anything was better
%    than nothing. Now I know that sometimes nothing is better.'' }}
%\ \\ \hfill{\footnotesize\textit{ -- Glenda Jackson}}
%\end{quote} 

\section{Introduction}
For the Maxwell equation in the exterior of a slowly rotating Kerr black hole, this paper provides both a bound for a positive definite energy and a Morawetz estimate. The Kerr family of solutions to Einstein's equation is parameterised by a mass parameter $M$ and a rotation parameter $a$. For $|a|\leq M$, these solutions describe black hole space-times \cite{HawkingEllis}. We consider the very slowly rotating case $|a|\ll M$, so that we can use $|a|/M$ as a small parameter in a boot-strap argument. 

The exterior region of the Kerr space-time in Boyer-Lindquist coordinates $(t,r,\theta,\phi)=(t,r,\omega)$ is the region $\Exterior=\Reals\times(\rp,\infty)\times\Stwo$ with the Lorentzian metric
\begin{align}
\gMetric&=-\left(1-\frac{2Mr}{\KSigma}\right)\di t^2 -\frac{4aMr}{\KSigma}\di\phi\di t +\frac{\KPi}{\KSigma}\sin^2\theta\di\phi^2 +\KSigma\di\theta^2+\frac{\KSigma}{\KDelta}\di r^2 ,
\label{eq:KMetric}
\end{align}
where
\begin{align*}
\KDelta&=r^2-2Mr+a^2, &
\KSigma&=r^2+a^2\cos^2\theta =|\Kp|^2,\\
\KPi&=(r^2+a^2)^2-a^2\KDelta\sin^2\theta , \\
\potlL&=\frac{\KDelta}{(r^2+a^2)^2},&
\Kp&=r-ia\cos\theta  .
\end{align*}
The Maxwell equation for a two-form (antisymmetric tensor) $\MaxF\in\wedge^2(\Exterior)$ can be expressed in terms of differential forms or tensors as
\begin{subequations}
\label{eq:Maxwell}
\begin{align}
\di\Hodge\MaxF&=0,&\text{or}&&\nabla^\alpha\MaxF_{\alpha\beta}&= 0,\\
\di\MaxF&=0,&&&\nabla_{[\gamma}\MaxF_{\alpha\beta]}&= 0.
\end{align}
\end{subequations}
We consider real Maxwell fields. Although the exterior region of the
Kerr black hole can be analytically extended, it is globally
hyperbolic and foliated by Cauchy surfaces of constant $t$, which we
denote by $\hst{t}$. The Maxwell equation is hyperbolic. Thus, the
Maxwell equation with initial data given at $t=0$ forms a well posed
initial value problem.

The electric and magnetic components of the Maxwell field strength tensor are defined with respect to a unit, time-like vector. In section \ref{s:FurtherGeometryOfKerr}, we will consider many different vectors, but at this point, it is sufficient to consider
\begin{align*}
\Tnormalised&=\frac{r^2+a^2}{\sqrt{\KSigma\KDelta}}\pt +\frac{a}{\sqrt{\KSigma\KDelta}}\pp ,&
\Rnormalised&=\sqrt{\frac{\KDelta}{\KSigma}}\pr .
\end{align*}
The vector field $\Tnormalised$ is not orthogonal to the surfaces of constant $t$. Instead, it is chosen so that the vectors $\vecLPNV$ and $\vecNPNV$ in equation \eqref{eq:PNV} below are principle null vectors. 
From $\Tnormalised$, the electric and magnetic vectors are defined
by
\begin{align*}
\Electric&= -\insertion{\Tnormalised}\MaxF  &\Electric_\alpha&=\MaxF_{\alpha\beta}\Tnormalised^\beta ,\\
\Magnetic&= -\insertion{\Tnormalised}\Hodge\MaxF  &\Magnetic_\alpha&=\frac12 \LeviCivita_{\alpha\beta}{}^{\gamma\delta}\MaxF_{\gamma\delta}\Tnormalised^\beta .
\end{align*}
The electric and magnetic components are both orthogonal to $\Tnormalised$ and, hence, space-like. A complex combination of the radial components
\begin{align*}
\solu&=\solu[\MaxF]=\Kp\sqrt{2}\insertion{\Rnormalised}(\Electric+i\Magnetic)
=(r-ia\cos\theta)\sqrt{2}(\Electric_\alpha+i\Magnetic_\alpha)\Rnormalised^\alpha 
\end{align*}
satisfies the Fackerell-Ipser equation \cite{FackerellIpser}
\begin{align*}
(\nabla^\alpha\nabla_\alpha-\potlFI)\solu&=0 ,
\end{align*}
where $\potlFI=-2M\Kp^{-3}=-2M/(r-ia\cos\theta)^3$. The potential $\potlFI$ is complex, which makes it significantly more difficult to treat than a real-valued potential.

To state the first of our main results in a simple way, we introduce a non-negative energy for each of $\MaxF$, treated as a solution of the Maxwell equation, and $\solu$, treated as a solution of a wave-like equation, 
\begin{subequations}
\label{eq:EnergyCrudeDefn}
\begin{align}
\EnergyCrudeF(t)&=\inthst{t} \left(|\Electric|^2+|\Magnetic|^2\right) \diThreeNaive,  \\
\EnergyCrudeFI(t)&= \inthst{t} \left(\frac{r^2+a^2}{\KDelta}|\pt\solu|^2 +\frac{\KDelta}{r^2+a^2}|\pr\solu|^2 +\frac{|\pAng\solu|^2}{r^2} +\frac{1}{r^2}|\solu|^2\right) \diThreeNaive ,
\end{align}
\end{subequations}
where $\pAng$ is the differential operator mapping $C^\infty$ to $C^\infty\times C^\infty$ defined by $\pAng\solu=(\ph\solu,(\sin\theta)^{-1}\pp\solu)$, and $\di\omega=\sin\theta\di\theta\di\phi$. 

\begin{theorem}[Uniform energy bound and decay to Coulomb solution]
\label{thm:EnergyBound}
There are positive constants $C$ and $\epsilonSlowRotationIntro$ and, for each choice of $M$ and $a$, a two parameter family of $t$ independent Coulomb solutions such that if $M>0$, $a\in[-\epsilonSlowRotationIntro,\epsilonSlowRotationIntro]$, and $\MaxF$ is a regular\footnote{Regular is defined in subsection \ref{ss:RegularityConditions}.} solution of the Maxwell equation \eqref{eq:Maxwell}, then $\forall t>0$
\begin{align*}
\EnergyCrudeF(t)+\EnergyCrudeFI(t)
&\leq \ConstantIntro\left(\EnergyCrudeF(0)+\EnergyCrudeFI(0)\right) . 
\end{align*}
Furthermore, for each $\MaxF$, there are parameters $\ChargeE$ and $\ChargeB$ determining a Coulomb solution, $\MaxFStatic$, with electric and magnetic vectors $\ElectricStatic$ and $\MagneticStatic$, such that
\begin{align*}
\int_{\Exterior} \frac{|\Electric-\ElectricStatic|^2+|\Magnetic-\MagneticStatic|^2}{r^2} \diThreeNaive\di t 
&\leq \ConstantIntro\left(\EnergyCrudeF(0)+\EnergyCrudeFI(0)\right) .
\end{align*}
More detailed space-time integral estimates are given in the Morawetz estimate \ref{thm:Morawetz}. 
\end{theorem}

One can reasonably argue that the energies in theorem \ref{thm:EnergyBound} suffer from two flaws. First, they are \textit{ad hoc}, rather than geometrically defined using the vector-field method. Second, the energies in theorem \ref{thm:EnergyBound} are degenerate with respect to the naturally induced energies on each $\hst{t}$. Regarding the first point, we have chosen to state the results in this form for simplicity, and, in fact, the proof of the theorem relies on showing that these energies are equivalent to geometrically defined energies. The geometric energies that we use are constructed from a vector field that vanishes as $r\rightarrow\rp$ on surfaces of constant $t$. It is for this reason that the energies are degenerate. The degeneracy of these energies is somewhat concealed in equation \eqref{eq:EnergyCrudeDefn}, since it might not be immediately obvious to the reader that the $1$-form $\di r$ degenerates as $r\rightarrow\rp$ (\cf remark \ref{rem:EnergyDegeneracy}), and hence so does the volume integrals $\diThreeNaive$. The coefficients of the terms in $\EnergyCrudeFI$ are exactly such that they compensate for the behaviour as $r\rightarrow\rp$ of the corresponding vector fields, and there is no further degeneracy, other than the one arising from the degeneracy of the volume form. Finally, in response to the second point, we refer to a powerful result in \cite{DafermosRodnianski:LectureNotes}, which explains how this type of degeneracy can be removed for a very general class of matter models in general relativity, including solutions of the Maxwell system. 

To prove a uniform bound on the energy, it is important to simultaneously prove some form of decay estimate. For many hyperbolic PDEs, there is a positive, conserved energy arising from a time translation symmetry and Noether's theorem. Although the Kerr space-time admits a time translation symmetry generated by $\pt$, because the coefficient of $\di t^2$ in the Kerr metric, $-1+2Mr/\KSigma$, changes sign, the vector field $\pt$ is not globally time-like in the exterior region, and the energy it generates is not necessarily positive. As in previous work for the wave equation on the Kerr space-time \cite{AnderssonBlue,DafermosRodnianski:KerrBoundedness,TataruTohaneanu}, we prove a Morawetz estimate, or integrated local energy decay estimate, for the Maxwell field. However, there is a major obstacle to proving such estimates. 

There are charged solutions of the Maxwell equation in the exterior of the Kerr black hole that do not decay in time. The magnetic and electric charges of a Maxwell field are  $(4\pi)^{-1}\intSSurface \MaxF$ and $(4\pi)^{-1}\intSSurface\Hodge\MaxF$. From the Maxwell equations, these are the same on any pair of $2$-surfaces that can be deformed to each other. In Minkowski space, $\Reals^{3+1}$, any $2$-surface can be deformed to a point, so that in the absence of charged sources, all solutions of the Maxwell equation are uncharged. However, this is not the case in the exterior of a Kerr black hole. These solutions are independent of time, so they cannot decay. In treating charged solutions, we decompose solutions into a stationary, charged part and an uncharged, dynamical part. The charged part is also called the Coulomb part. 

\begin{prop}[Decomposition into charged and uncharged parts]
\label{prop:ChargeDecomposition}
There are positive constants $C$ and $\epsilonSlowRotationIntro$ such that if  $M>0$, $a\in[-\epsilonSlowRotationIntro,\epsilonSlowRotationIntro]$, and $\MaxFTotal$ is a regular solution of the Maxwell system \eqref{eq:Maxwell}, then there are $\MaxFStatic$ and $\MaxF$ such that
\begin{enumerate}
\item $\MaxFTotal=\MaxFStatic+\MaxF$, 
\item $\MaxFStatic$ and $\MaxF$ are solutions of the Maxwell system \eqref{eq:Maxwell}, 
\item the Lie derivative along $\pt$ of $\MaxFStatic$ vanishes: $\Lie_{\pt}\MaxFStatic=0$, 
\item for any closed $2$-surface, $\TwoSfc$, the surface integrals of $\MaxF$ and of $\Hodge\MaxF$ vanish: $\intSSurface \MaxF=0=\intSSurface\Hodge\MaxF$, and
\item there is a constant $\ConstantIntro$ such that
\begin{align*}
\ConstantIntro^{-1} \EnergyCrudeFArg{\MaxFTotal}
&\leq \left(\EnergyCrudeFArg{\MaxFStatic}+\EnergyCrudeFArg{\MaxF}\right)
\leq \ConstantIntro \EnergyCrudeFArg{\MaxFTotal} ,\\
\ConstantIntro^{-1} \EnergyCrudeFIArg{\soluTotal}
&\leq \left(\EnergyCrudeFIArg{\soluStatic}+\EnergyCrudeFIArg{\solu}\right)
\leq \ConstantIntro \EnergyCrudeFIArg{\soluTotal} ,
\end{align*}
where $\soluTotal$, $\soluStatic$, and $\solu$ denote the radial components of $\MaxFTotal$, $\MaxFStatic$, and $\MaxFTotal$ respectively. 
\end{enumerate}
\end{prop}

For the remainder of this paper, we will typically use $\MaxF$ to denote an uncharged solution or the uncharged part of a solution $\MaxFTotal$. 

%Since $\Lie_{\pt}\MaxFStatic=0$, we call $\MaxFStatic$ the static component and $\MaxF$ the dynamic component. Conversely, since $(4\pi)^{-1}\intSSurface \MaxF$ and $(4\pi)^{-1}\intSSurface\Hodge\MaxF$ are the magnetic and electric charges, we call $\MaxF$ the uncharged component and $\MaxFStatic$ the charged or Coulomb component. 

To investigate the decay properties of solutions of the Maxwell system, it is useful to perform a null or spinor decomposition. To do so, first we introduce the null vectors 
\begin{align}
\vecLPNV&=\frac{1}{\sqrt{2}}(\Tnormalised+\Rnormalised),&
\vecNPNV&=\frac{1}{\sqrt{2}}(\Tnormalised-\Rnormalised),
\label{eq:PNV}
\end{align}
and the unit vectors
\begin{align*}
\Hnormalised&=\frac{1}{\sqrt{\KSigma}}\ph,&
\Pnormalised&=\frac{1}{\sqrt{\KSigma}}\left(a\sin\theta\pt+\frac{1}{\sin\theta}\pp\right) .
\end{align*}
The vectors $\vecLPNV$ and $\vecNPNV$ are each double solutions of the
principal null vector equation \cite{Wald}.  This allows us to consider
the complex null tetrads
\begin{align}
\label{eq:TetradForCharge}
\vecLPNV,&&
\vecNPNV,&&
\vecMPNV&=\frac{1}{\sqrt{2}}\left(\Hnormalised+i\Pnormalised\right),&
\vecMbPNV .
\end{align}
This is known as the Carter tetrad \cite{Znajek}. The (complex) spinor components of
the Maxwell field are
\begin{align*}
\phip&=2\MaxF[\vecLPNV,\vecMPNV] ,\\
\phiz&=\sqrt{2}\left(\MaxF[\vecLPNV,\vecNPNV]+ \MaxF[\Hnormalised,\Pnormalised]\right),&\solu&=(r-ia\cos\theta)\phiz\\
\phim&=2\MaxF[\vecNPNV,\vecMbPNV] . 
\end{align*}
Here, the components are indexed by spin weight, following the convention of Price \cite{Price} and used in \cite{Blue:Maxwell}. This differs from the indexing convention used by almost the entire rest of the physical literature, in which, ignoring powers of $\sqrt{2}$, the quantities $\phim$, $\phiz$, and $\phip$ are denoted $-\phiUglym$, $\phiUglyz$, and $\phiUglyp$. This also differs from the notation used in the precursor to the proof of nonlinear stability of Minkowski space \cite{ChristodoulouKlainerman:DecayOfLinearFields}, where, after the appropriate vector bundles have been identified and powers of $\sqrt{2}$ have again be ignored, the quantities $\phim$, $\phiz$, and $\phip$ are denoted $\CKalphab$, $\CKrho+i\CKsigma$, and $\CKalpha$, and where spin weight is called signature \cite{ChristodoulouKlainerman:NonlinearStabilityOfMinkowski}. In subsection \ref{ss:RegularityOfTheNullDecomposition}, the null components $\phii$ are identitified with sections of vector bundles, so that if $\MaxF$ is smooth, then the corresponding sections will also be smooth. The unusual factors of $\sqrt{2}$ are chosen to obtain convenient factors in lemma \ref{lemma:MaxFEMSComponents}.

A Morawetz or integrated local energy decay estimate provides a bound
on weighted space-time integrals of the components of the Maxwell
field. One major obstacle to proving such estimates is the existence
of orbiting null geodesics. In the Schwarzschild $a=0$ case, these
occur at $r=3M$. In the slowly rotating case $|a|\ll M$, the orbiting
null geodesics remain near $r=3M$. To avoid these orbiting null
geodesics and the necessary degeneracy in the Morawetz estimate near
them, we introduce a distance $\rfar$ and a smooth cut-off $\chirfar$,
which is identically $1$ for $|r-3M|>2\rfar$, identically $0$ for
$|r-3M|<\rfar$, monotone in between, and is such that, for all
$k\in\Naturals:$ the derivative $\pr^k\chirfar$ is bounded by a
constant times $\rfar^{-1}$. We introduce a time $\FinalTime>0$ at
which we wish to estimate the energy. For a charge-free Maxwell field
with components $\phii$ and $\solu=\solu[\MaxF]=\Kp\phiz$, we define the following bulk space-time integrals
\begin{subequations}
\begin{align}
\BulkMorawetzF
&=\intBulk{\FinalTime}\frac{M\KDelta}{(r^2+a^2)^2} |\phipm|^2 \diFourNatural ,\\
\BulkFILowerOrder
&=\intBulk{\FinalTime} \frac{M|\phiz|^2}{r^2}\diFourNatural  
=\intBulk{\FinalTime} \frac{M}{r^4}|\solu|^2r^2\diFourFI ,\\
\BulkFITwo
&=\intBulk{\FinalTime} \left(\frac{M\KDelta^2}{(r^2+a^2)r^2}|\pr\solu|^2 +\chirfar\frac{M^2|\pt \solu|^2+|\pAng\solu|^2}{r}\right)\diFourFI, \\
\BulkFIOne
&=\int_{\rp}^{\infty} (1-\chirfar)
\left|\int_{0}^{\FinalTime}\int_{\Stwo} \Im(\solub\pt\solu) \diomega\di t\right| \di r ,
\end{align}
\end{subequations}
where $\diFourNatural$ is the geometrically defined volume form $\KSigma\diomega\di r\di t$, and $\diFourFI$ is the coordinate volume form $\diomega\di r\di t$. For $T<0$, we reverse the sign in these bulk terms, so that they remain nonnegative. The indexing is chosen so that $\BulkMorawetzF$ involves $\phipm$, $\BulkFILowerOrder$ involves $\phiz$ or, equivalently, $\solu$ with no derivatives, $\BulkFIOne$ involves $\solu$ with one derivative in the integral, and $\BulkFITwo$ involves $\solu$ with two derivatives but with a degeneracy near the orbiting null geodesics. 

\begin{theorem}[Space-time integrated local energy (Morawetz) estimate]
\label{thm:Morawetz}
There are positive constants $\epsilonSlowRotationIntro$, $\ConstantIntro$, such
that 
if 
$\epsilonMain\leq\epsilonSlowRotationIntro$, 
$\MaxFTotal$ is a regular solution of the Maxwell equation \eqref{eq:Maxwell} for which $\EnergyCrudeFArg{\MaxFTotal}(0)$ and $\EnergyCrudeFIArg{\soluTotal}(0)$ are finite, 
and the quantities $\BulkMorawetzF$, $\BulkFILowerOrder$, $\BulkFIOne$, and $\BulkFITwo$ are defined in terms of the uncharged part $\MaxF$, 
then $\forall T\in\Reals:$
\begin{align}
\label{eq:BulkTermsDef}
\BulkMorawetzF
+\BulkFILowerOrder
+\BulkFIOne
+\BulkFITwo
&\leq \ConstantIntro\left(\EnergyCrudeFArg{\MaxFTotal}(0) + \EnergyCrudeFIArg{\soluTotal}(0)\right)   .
\end{align}
\end{theorem}

The results of theorem \ref{thm:Morawetz} can be interpreted as a relatively weak form of decay, since they show that, for fixed intervals in $r$ the integral in $t$ and the angular variables is finite. Thus, the average of the electric and magnetic components must decay to zero over space-time regions with $r$ bounded and $t\rightarrow\pm\infty$. Although this type of decay estimate is relatively weak, it is sufficiently robust to have formed the foundation for many further decay results in the study of fields outside black holes. 

The major obstacles in proving energy and Morawetz estimates for the Maxwell field and their resolution in this paper are the following: 
\begin{enumerate}
\item There is no positive, conserved energy. Noether's theorem associates a conserved energy to symmetries of a PDE, and $\pt$ generates a symmetry of the Kerr space-time and hence the Maxwell equation on it. However, because $\pt$ fails to be time-like everywhere in the exterior, this energy fails to be nonnegative. On the other hand, any time-like vector generates a nonnegative energy, but these need not be conserved. To generate a positive energy, we use the the vector field $\vecTBlend=\pt+\fnBlend (a/(\rp^2+a^2))\pp$ which we introduced in \cite{AnderssonBlue}. Here, $\fnBlend$ is a function that goes from $1$ to $0$ in a range of $r$ values of our choice. We will take this region to be $[10M,11M]$ for convenience. This vector field is time-like in the exterior, so it generates a nonnegative energy. Furthermore, this vector field only fails to be a symmetry in the region $r\in[10M,11M]$. For $a$ small, this means that we can control the change in the energy by the bulk terms. 
\item There are orbiting null geodesics, and these null geodesics fill an open set. Orbiting null geodesics are a major obstacle to decay estimates such as the Morawetz estimate, since initial data can be chosen so that an arbitrarily large proportion of the energy remains near the orbiting null geodesics for an arbitrarily long period of time.
% \cite{Sbierksi}
In the Schwarzschild $a=0$ case, these occur only at $r=3M$, but in the $a\not=0$ Kerr case, these bifurcate from this value. However, for $a$ small, these orbiting null geodesics remain close to $r=3M$. Following our earlier work \cite{AnderssonBlue}, we use a phase space function $\DCurlyRT$ as a measure of distance from these orbiting null geodesics.  In our core estimates, we use $\chirfar$ to keep at least $\rfar/2$ away from these orbiting null geodesics when integrating two time or angular derivatives in expectation value (\ie{} one derivative in $L^2$). A more detailed description of the null geodesics and out treatment of them appears in subsection \ref{ss:FIMorawetzOutline}. Fortunately, since $3M$ and the interval $[10M,11M]$ are far apart, there is no danger that the degeneracy at the orbiting null geodesics will prevent us from estimating the terms arising from the failure of $\vecTBlend$ to be a generator of a symmetry. 
\item There are stationary solutions. The existence of such solutions prevents a Morawetz estimate from holding for general solutions. In fact, to the best of our knowledge, there is no way to prove a Morawetz estimate for the Maxwell field directly. Proving such an estimate would be an important advance in the field, even in the Schwarzschild case. Instead, we use the fact that $\solu$ satisfies the Fackerell-Ipser equation, which is a wave-like equation with a potential. Morawetz estimates are known for the wave equation. The Fackerell-Ipser equation also has stationary solutions, which arise from the stationary solutions to the Maxwell equation, because of the potential. In the Schwarzschild case, the stationary solutions are the only spherically symmetric solutions, so that, after projecting out the stationary solutions, one can use an additional positivity arising from the angular derivatives to counteract the influence of the potential in the Morawetz estimate \cite{Blue:Maxwell}. Although, the stationary solutions are not exactly spherically symmetric in the $a\not=0$ case, there is still a lower bound for the charge-free solutions in terms of their angular derivatives, which we state in lemma \ref{lemma:LowerBoundForAngularDerivatives}. This is sufficient to allow us to prove a Morawetz estimate for the Fackerell-Ipser equation. 
\item The Fackerell-Ipser equation has a complex potential. We consider energies for both the Maxwell and Fackerell-Ipser equations. The Fackerell-Ipser equation does not obviously arise as the Euler-Lagrange equation for any real-valued Lagrangian, which means that we cannot use Noether's theorem to generate a conserved energy from a symmetry. However, when $|a|/M$ is small, the imaginary part of the potential is also small. Unfortunately, unlike the terms arising from the failure of $\vecTBlend$ to be a generator of a symmetry, the imaginary part of the potential is not supported away from the orbiting null geodesics. This requires us to prove a bound on a term of the form $\BulkFIOne$. The problems arising from complex potentials were already discussed in \cite{AnderssonBlueNicolas}. The need for control of $\BulkFIOne$ is one of the reasons why we need to use pseudodifferential operators, instead of just differential operators. 
\end{enumerate}

\subsection{Previous results}
Fields outside of black holes have been studied for several decades, and Morawetz estimates have been found to be a particularly powerful tool in the last decade. Our earlier works \cite{AnderssonBlue,Blue:Maxwell} provide a more complete description of this. Here we briefly summarise some of the major results. 

The wave equation outside black holes has been more extensively
studied than higher spin fields. In the $a=0$, because $\pt$ is both
globally time-like and Killing, there is a conserved energy
\cite{Wald}. The idea of trying to prove a Morawetz estimate was first
introduced in terms of a coordinate dependent Schr\"odinger equation
\cite{LabaSoffer}. This was then carried over to the wave equation
\cite{BlueSoffer:LongPaper,BlueSterbenz,DafermosRodnianski:RedShift}. For
$a\not=0$, the construction of a bounded energy is significantly more
difficult. This was first proved in conjunction with a Morawetz
estimate for a range of frequencies
\cite{DafermosRodnianski:KerrBoundedness}. An energy bound and a Morawetz estimate for all frequencies was proved shortly afterward \cite{TataruTohaneanu}. Shortly after that, we proved a similar result at a higher level of regularity, using only classical differential operators instead of Fourier or spectral methods \cite{AnderssonBlue}. We hope to apply similar results to the Maxwell field problem, but it seems that there is not yet a theory of hidden symmetries for the Maxwell field that is as advanced as the theory for the wave equation. As explained in our previous work \cite{AnderssonBlue}, from Morawetz estimates, many pointwise decay estimates have been proved. 

For higher, integer spin fields outside a Schwarzschild black hole, energy and pointwise decay estimates were first proved in the far exterior, roughly $r>t+C$ \cite{IngeleseNicolo}. Energy and Morawetz estimates have also been used to prove pointwise decay in the full exterior region \cite{Blue:Maxwell}. Energy and Morawetz estimates have also been proved on arbitrary spherically symmetric black hole space-times, independently of whether they satisfy the Einstein equation \cite{SterbenzTataru}. For linearised gravity, the curvature also has a decomposition into components, and, in the $a=0$ Schwarzschild case, the middle component also satisfies a wave equation; energy and Morawetz estimates have also been proved for this case \cite{BlueSoffer:ReggeWheeler}. On a space-time that is a solution of the Einstein equation and for which the metric, connection coefficients, and curvature decay to the Schwarzschildean values at a certain level of regularity, one finds that higher derivatives of the connection coefficients and curvature satisfy certain equations. The equations for the derivatives of the curvature are wave-like equations. Under an assumption excluding the analogue of the stationary solutions to the Maxwell equation, from energy and Morawetz estimates, pointwise decay for the connection coefficients and the curvature have been proved \cite{Holzegel}. 

For both the Maxwell and linearised Einstein equation, the middle component of the null decomposition satisfies the Regge-Wheeler equation \cite{ReggeWheeler}, which is a wave-like equation with a real potential. In the $a\not=0$ case, the middle component of the Maxwell field, $\phiz$, satisfies the Fackerell-Ipser equation \cite{FackerellIpser}, and, the middle component of the linearised curvature satisfies a very similar equation \cite{AksteinerAndersson}. In addition, the extreme components for both the Maxwell and linearised Einstein equations also satisfy second-order PDEs, known as the Teukolsky equations \cite{Teukolsky}. One of the few results about solutions to the Maxwell and linearised Einstein equation outside rotating Kerr $a\not=0$ black holes is that there are no exponentially growing solutions to the Teukolsky equation \cite{Whiting}. 

For the Maxwell equation and linearised Einstein equation outside a Schwarzschild black hole, other techniques have also been applied. Scattering has been proved \cite{Bachelot}. Decay in $L_{\text{loc}}^\infty$ has been proved without a rate \cite{FinsterSmoller}. For solutions of the Regge-Wheeler equation, pointwise decay has been proved \cite{DonningerSchlagSoffer}. 

%%%%%%%%%%%%%%%%%%%%%%%%%%%%%%%%%%%%%%%%%%%%%%%%%%%%%%%%%%%%%%%%%%%%%%%%%%%%
\subsection{Outline of the proof and the core estimates}
The core estimates in this paper are that if $|a|/M$ is sufficiently small, then there is a choice of $\rfar/M>0$ and constant $C>0$ such that for any charge-free solution
\begin{align}
\EnergyCrudeF(T)&\leq C\left( \EnergyCrudeF(0) + \epsilonMain \left(\BulkMorawetzF+\BulkFILowerOrder\right) \right),
\tag{I}\label{eq:CoreI}\\
\BulkMorawetzF&\leq C\left( \EnergyCrudeF(T)+\EnergyCrudeF(0)+\BulkFILowerOrder \right),
\tag{II}\label{eq:CoreII}\\
\EnergyCrudeFI(T)&\leq C\left( \EnergyCrudeFI(0) +\EnergyCrudeF(0)+ \epsilonMain \left(\BulkMorawetzF+\BulkFILowerOrder+\BulkFIOne+\BulkFITwo\right)\right)
\tag{III}\label{eq:CoreIII}\\
\BulkFILowerOrder+\BulkFITwo&\leq C\left( \EnergyCrudeFI(T) +\EnergyCrudeFI(0) +\EnergyCrudeF(T)+\EnergyCrudeF(0)+\frac{|a|}{M}\BulkMorawetzF\right),
\tag{IV}\label{eq:CoreIV}\\
\BulkFIOne&\leq C\left( \EnergyCrudeFI(T) +\EnergyCrudeFI(0) +\BulkFILowerOrder+\BulkFITwo \right).
\tag{V}\label{eq:CoreV}
\end{align}
Recall these energies and bulk terms were defined in equations \eqref{eq:EnergyCrudeDefn} and \eqref{eq:BulkTermsDef}. In fact, for any sufficiently small choice of $\rfar/M>0$, there is a choice of $C$ such that these estimates hold, but this choice of $C$ is not uniform, so we choose a single value of $\rfar/M$. 
 
Theorems \ref{thm:EnergyBound} and \ref{thm:Morawetz} follow from proposition \ref{prop:ChargeDecomposition}, core estimates \eqref{eq:CoreI}-\eqref{eq:CoreV}, and the following simple bootstrap argument. By substituting core estimate \eqref{eq:CoreIV} into core estimates \eqref{eq:CoreII} and \eqref{eq:CoreV}, one finds that 
\begin{align*}
\BulkMorawetzF+\BulkFILowerOrder+\BulkFIOne+\BulkFITwo
\lesssim \EnergyCrudeF(\FinalTime)+\EnergyCrudeF(0)+\EnergyCrudeFI(\FinalTime)+\EnergyCrudeFI(0)+\frac{|a|}{M}\BulkMorawetzF .
\end{align*}
Thus, for $|a|/M$ sufficiently small, one finds
\begin{align*}
\BulkMorawetzF+\BulkFILowerOrder+\BulkFIOne+\BulkFITwo
\lesssim \EnergyCrudeF(\FinalTime)+\EnergyCrudeF(0)+\EnergyCrudeFI(\FinalTime)+\EnergyCrudeFI(0) .
\end{align*}
Substituting this into core estimates \eqref{eq:CoreI} and \eqref{eq:CoreIII} and taking $|a|/M$ sufficiently small proves theorems \ref{thm:EnergyBound} and \ref{thm:Morawetz}. 

The main method in this paper is to choose vector fields with which to generate energies by contracting with an energy-momentum tensor and then integrating over a hypersurface of constant $t$. Core estimates \eqref{eq:CoreI}-\eqref{eq:CoreII} are proved using the energy-momentum tensor for the Maxwell equation, using $\vecTBlend$ and a radial vector field $\vecMorawetzBasic$ respectively. Core estimate \eqref{eq:CoreIII} is proved using an approximate energy-momentum tensor for the Fackerell-Ipser equation and the vector field $\vecTBlend$. To prove core estimates \eqref{eq:CoreIV} and \eqref{eq:CoreV}, we introduce a Fourier-spectral transform of the Fackerell-Ipser equation. The variables of the Fourier-spectral transform correspond to time and angular derivatives. Solutions to this also have an analogue of an energy-momentum tensor. We use a radial vector field that points away from the orbiting null geodesics to prove core estimate \eqref{eq:CoreIV}. We then rescale this vector field by powers of the Fourier-spectral variables to obtain more control over the time and angular derivatives. This rescaled vector field allows us to prove core estimate \eqref{eq:CoreV}. 

Core estimate \eqref{eq:CoreV} actually follows from an estimate on 
\begin{align*}
\int \modspectralcollection^{\ExponentRefinedThree}|\soluS|^2 \diFourSpectral ,
\end{align*}
where $\spectralcollection$ refers to the spectral parameters corresponding to derivatives in the time and angular directions, and $\soluS$ is the transform of $\solu$. This is based on ideas in \cite{BlueSoffer:LongPaper} for the $a=0$ wave equation, where the exponent can be increased from $3/2$ to $2-\epsilon$ for any $\epsilon>0$. A similar result, with only logarithmic losses was proved in \cite{Tohaneanu}. 

One useful technique when treating these estimates is to note that there is an overall rescaling freedom in the Kerr metric $(M,a;t,r,\omega)\mapsto(\lambda M,\lambda a;\lambda t,\lambda r,\omega)$. It is for this reason that it is possible to use quantities like $|a|/M$ as a measure of slow rotation. To preserve this scale invariance, we also wish to set $\rfar/M$ when defining $\BulkFIOne$ and $\BulkFITwo$, instead of setting $M$ and then setting $\rfar$. 

In section \ref{s:FurtherGeometryOfKerr}, we introduce some more notation, discuss the geometry of the Kerr space, and consider the regularity of the components $\phipm$. In section \ref{s:Cohomology}, we discuss the decomposition of solutions into stationary and charge-free components, which allows us to prove proposition \ref{prop:ChargeDecomposition}. In section \ref{s:Maxwell}, we prove core estimates \eqref{eq:CoreI}-\eqref{eq:CoreII}, which are proved using the Maxwell equations without reference to the Fackerell-Ipser equation. In section \ref{s:FI}, we treat the Fackerell-Ipser equation and prove core estimate \eqref{eq:CoreIII}. In section \ref{s:FourierSpectral}, we treat the Fourier-spectral transform of the Fackerell-Ipser equation and prove core estimates \eqref{eq:CoreIV}-\eqref{eq:CoreV}. Recall that when we introduced core estimates \eqref{eq:CoreI}-\eqref{eq:CoreV}, we explained how they combine with the charge decomposition in proposition \ref{prop:ChargeDecomposition} to prove theorems \ref{thm:EnergyBound} and \ref{thm:Morawetz}.

%%%%%%%%%%%%%%%%%%%%%%%%%%%%%%%%%%%%%%%%%%%%%%%%%%%%%%%%%%%%%%%%%%%%%%%%%%%%%%%
\section{Some further notation and the geometry of the Kerr space-time}
\label{s:FurtherGeometryOfKerr}

\subsection{Some further notation}
We say that two quantities $A$ and $B$ are $a$ almost equal iff there is a constant $C$ such that $(1-C|a|/M)A\leq B\leq (1+C|a|/M)A$. 

We use $A\lesssim B$ to mean that there is a constant $C$, such that
$A\leq CB$, and that $C$ can be chosen independently of $M$, $a$, $r$,
$t$, $\theta$, $\phi$, $\MaxF$, $\solu$, $\soluS$ and chosen to be
uniformly with respect to some small parameter. Typically this means
that for sufficiently small $\epsilonSlowRotationIntro>0$, there is a
constant $C$ uniform in $\epsilonMain\leq\epsilonSlowRotationIntro$
such that $A\leq CB$. However, it certain cases, we will also
introduce other small parameters, such as $\epsilondtsquared/M$, where
$\epsilondtsquared$ is a parameter associated with coefficients of
$\pt$ in section \ref{s:FourierSpectral}. The relevant cases are defined explicitly
in subsection \ref{ss:Smallness}. We use $A\sim B$ to mean $A\lesssim B$ and $B\lesssim A$. If estimates of this form hold for one value of $\epsilonSlowRotationIntro$, they will also hold for any smaller (positive) value of $\epsilonSlowRotationIntro$. Thus, if there is a finite collection estimates of this form, then an $\epsilonSlowRotationIntro$ can be found for which all of the estimates hold. We are also free to choose smaller $\epsilonSlowRotationIntro$ to obtain stronger estimates. 

Frequently in this paper, we will deal with rational functions. When discussing polynomials or rational functions, we will implicitly take this to mean a polynomial or rational function in $r$, $M$, $\epsilondtsquared$, $a$, and, in certain cases, $a\cos\theta$. We define a homogeneous rational function to be a ratio of homogeneous polynomials. We define a homogeneous polynomial of degree $m$ to be of maximal degree in $r$ if the coefficient of $r^m$ is nonzero. We define a homogeneous rational function to be of maximal degree if it is the ratio of two homogeneous polynomials of maximal degree in $r$. A homogeneous rational function has both a degree and an asymptotic growth rate in $r$; for homogeneous rational functions of maximal degree, these are equal. The set of homogeneous rational functions is closed under the operations of taking products and ratios. The set of homogeneous rational functions of maximal degree is also closed under these operations. Similarly, the set of homogeneous polynomials which have a positive lower and upper bound on compact subsets of the closure of the exterior, $r\geq\rp$, is also closed under these operations. Importantly, if a homogeneous rational function of maximal degree has a positive lower and upper bound on each compact subset of the closure of the exterior when $a=0=\epsilondtsquared$, then it will still have such bounds for $a$ and $\epsilondtsquared$ sufficiently small. This is not true for polynomials or rational functions that fail to be of maximal degree. Furthermore, if a homogeneous polynomial, $p$, is of maximal degree and has growth rate $r^{m}$ for $a=0=\epsilondtsquared$, then the influence of $a$ and $\epsilondtsquared$ is lower order, in the sense that $|p(r,M,0,0)-p(r,M,\epsilondtsquared,a)|=O(r^{-m-1})$. 

There are four situations in which we will encounter homogeneous rational functions that fail to have positive lower and upper bounds in each compact subset of the closure of the exterior. First, for $a=0=\epsilondtsquared$, the function might be nonnegative but vanish at $r=2M$. Crucially, whenever this occurs in this paper, we have been able to factor the rational function as a power of $\KDelta$ times a rational function with a positive lower and upper bound on each compact subset of the closure of the exterior. Second, for $a=0$, the function might vanish at some $r\in(2M,\infty)$. This occurs only once in our argument, in the coefficient of the term denoted $\DCurlyRT$. Third, for a homogeneous rational function of maximal degree, the derivative need not have maximal degree if the original function is of degree zero. This situation occurs in our argument twice; it occurs with the coefficient of $\pt^2$ in $\DCurlyRT$ and with all the coefficients in the term denoted $\DDCurlyRTT$. These second and third problems are treated in lemma \ref{lemma:PropertiesOfDCurlyRT}. Fourthly, there can be a complicated sum of derivatives so that even for $a=0=\epsilondtsquared$, the homogeneous rational function is not positive in the exterior. This occurs in the term $\mathcal{V}$ in lemma \ref{lemma:BulkMainMorawetzFI}.

Recall that we defined
\begin{align*}
%\diomega&=\sin\theta\di\theta\di\phi,&
\diFourFI&=\di r\diomega\di t,&
\diFourNatural&=\KSigma\diFourFI .
\end{align*}

\subsection{Foliations}
\label{ss:Foliations}
Recall that we consider the exterior region of the Kerr space-time,
which is a manifold parameterised by
$(t,r,\omega)\in\Reals\times(\rp,\infty)\times\Stwo$ with the metric
given in equation \eqref{eq:KMetric}. As is well
known \cite{HawkingEllis}, this can be uniquely extended to a maximal
analytic extension, which, in turn, has a $C^0$ conformal
compactification. 

Of particular value in our analysis will be subregions of the form
\begin{align*}
\BulkZone{t_1}{t_2}
&= \{ (t',r',\omega') | t'\in(t_1,t_2),
r'\in(\rp,\infty),\omega'\in\Stwo \} ,
\end{align*}
for $t_1<t_2$. 
The entire exterior region and each subregion $\BulkZone{t_1}{t_2}$
are foliated by hypersurfaces 
\begin{align*}
\hst{t}&=\{ (t',r',\omega') | t'=t, 
r'\in(\rp,\infty),\omega'\in\Stwo \} .
\end{align*}
The exterior is also foliated by $\hsr{r}=\{ (t',r',\omega') |
t'\in\Reals, r'=r,\omega'\in\Stwo \}$. When dealing with
$\BulkZone{t_1}{t_2}$, we will also use $\hsr{r}$ to denote
$\hsr{r}\cap\BulkZone{t_1}{t_2}$. We orient $\hst{t}$ and $\hsr{r}$
with normal $1$-form pointing along $-\di t$ and $-\di r$
respectively. Along surfaces of fixed $t$, inside the conformal
compactification of the maximal analytic extension of the exterior
region, there is a well defined limit set as $r\rightarrow\rp$ and as
$r\rightarrow\infty$, and this limit is independent of $t$. We will
denote these as $\bif$ and $\spacelikeinfty$. These are known as the
bifurcation sphere and space-like infinity respectively. They are each
$2$-dimensional surfaces, although each $\hst{t}$ and $\hsr{r}$ (for
$r\in(\rp,\infty)$) is a $3$-hypersurface. Thus, for any regular
vector field, the flux through $\bif$ and $\spacelikeinfty$ is
zero. In the conformal compactification of the maximal analytic
extension, the boundary of $\BulkZone{t_1}{t_2}$ is
$\hst{t_2}\cup-\hst{t_1}\cup-\bif\cup\spacelikeinfty$.

In a flux integral, one needs the normal to a hypersurface and the
induced volume form. For $\hst{t}$, this is 
\begin{align*}
\diNormal{\alpha}{\hst{t}}
&=-\left(\frac{\gMetric^{\alpha t}}{(\gMetric^{tt})^{1/2}}\right) 
\sqrt{\gMetric_{rr}\gMetric_{\theta\theta}\gMetric_{\phi\phi}}
\di r\di\theta\di\phi \\
&=-\gMetric^{\alpha t} 
\left(\frac{\gMetric_{tt}\gMetric_{\phi\phi}-\gMetric_{t\phi}^2}{\gMetric_{\phi\phi}}\right)^{1/2}
\sqrt{\gMetric_{rr}\gMetric_{\theta\theta}\gMetric_{\phi\phi}}
\di r\di\theta\di\phi \\
&=-\gMetric^{t\alpha}\sqrt{-\det\gMetric}\sin\theta 
\di r\di\theta\di\phi ,\\
&=\left(\pt+\frac{2aMr}{\KPi}\pp\right)^\alpha\frac{\KPi}{\KDelta}\sin\theta
\di r\di\theta\di\phi .
\end{align*}
Similarly,
\begin{align*}
\diNormal{\alpha}{\hsr{r}}
&=-\pr^\alpha \KDelta\sin\theta \di t\di\theta\di\phi .
\end{align*}

\subsection{Time-like vectors}
\label{ss:TimelikeVectors}
In the exterior of the Schwarzschild space-time, there is a unique globally time-like
Killing vector, $\pt$, which is orthogonal to hypersurfaces of constant
$t$. In the Kerr space-time, there is no such structure, so we are
forced to consider a variety of vectors. It is useful to consider
\begin{align*}
\vecTBlend&=\pt +\omegaBlend\pp, 
&\omegaBlend&=\fnBlend \frac{a}{\rp^2+a^2} , \\
\vecTPerp&=\pt +\omegaPerp\pp, 
&\omegaPerp&=\frac{2aMr}{\KPi}, \\
\vecTPNV&=\pt +\omegaPNV\pp, 
&\omegaPNV&=\frac{a}{r^2+a^2} ,
\end{align*}
where $\fnBlend$ is a smooth, decreasing function of $r$, that is identically
$1$ for $r<10M$, is $0$ for $r>11M$, and satisfies $\forall k\in\Naturals: \pr^k\fnBlend \lesssim M^{-k}$. 

Each of these vectors has a length that approaches $-1$ as
$r\rightarrow\infty$ and which vanishes at a rate of
$-(\KDelta/r^2)^{1/2}$ as $r\rightarrow\rp$. The difference between
any two of them is bounded by $|a|\KDelta r^{-4}$, in particular, this
difference vanishes relative to the length of any of these three
vectors as $r\rightarrow\rp$, $r\rightarrow\infty$, and uniformly in
$r$ as $|a|\rightarrow 0$.

\subsection{Regularity of the null decomposition}
\label{ss:RegularityOfTheNullDecomposition}
Recall that we introduced
\begin{align*}
\phip&=2\MaxF[\vecLPNV,\vecMPNV] ,\\
\phiz&=\sqrt{2}\left(\MaxF[\vecLPNV,\vecNPNV]+ \MaxF[\Hnormalised,\Pnormalised],\right)\\
\phim&=2\MaxF[\vecNPNV,\vecMbPNV] .
\end{align*}

The vector fields $\vecLPNV$ and $\vecNPNV$ are globally smooth, but $\vecMPNV$ and $\vecMbPNV$ fail to be globally smooth, since they are constructed from 
\begin{align*}
\Hnormalised&=\frac{1}{\sqrt{\KSigma}}\ph,&
\Pnormalised&=\frac{1}{\sqrt{\KSigma}}\left(a\sin\theta\pt+\frac{1}{\sin\theta}\pp\right),  
\end{align*}
which fail to be smooth at $\theta\in\{0,\pi\}$. However
\begin{align*}
\KSigma\vecMbPNV\wedge\vecMPNV&=\ph\wedge\frac{1}{\sin\theta}\pp +a\sin\theta\ph\wedge\pt ,
\end{align*}
is the sum of the volume form on the round sphere plus the wedge product of the smooth vector fields $\sin\theta\ph$ and $\pt$, so $\vecMbPNV\wedge\vecMPNV$ is smooth. Hence $\phiz$ is a smooth function if $\MaxF$ is smooth. Similarly, $\solu=(r-ia\cos\theta)\phiz$ is also smooth. 

In contrast, $\phipm$ are not typically smooth functions, even if $\MaxF$ is smooth, since $\vecMPNV$ is not smooth. To overcome this, we must reinterpret them as sections of a vector bundle. At each point, consider the vector space spanned by $\Hnormalised$ and $\Pnormalised$. Since the wedge product of these vectors has already been shown to extend to a globally smooth $2$-form, their span can be extended to a globally defined $2$-dimensional vector bundle, $\CKBundlepm$. Since it is a sub-bundle of the tangent space, we can use the same indices in both bundles. Since $\Hnormalised$ and $\Pnormalised$ are each orthogonal unit vectors, the vector bundle $\CKBundlepm$ has a Riemannian metric. For $(\theta,\phi)\in(0,\pi)\times(0,2\pi)$, the metric on this vector bundle is 
\begin{align*}
\Hnormalised\otimes\Hnormalised+\Pnormalised\otimes\Pnormalised ,
%&=\KSigma\CarterQ .
\end{align*}
which we will denote by $\KSigma^{-1}\CarterQ$. This metric has a unique smooth extension, for each $(t,r)$ to $\omega\in\Stwo$. We will also denote this extension my $\KSigma^{-1}\CarterQ$. 

This $\KSigma^{-1}\CarterQ$ can be seen as a projection operator from the tangent space of the space-time to $\CKBundlepm$. Inspired by \cite{ChristodoulouKlainerman:DecayOfLinearFields}, we define
\begin{align*}
\CKalpha^\alpha &= \MaxF_{\gamma\delta}(\vecLPNV)^\gamma (\KSigma^{-1}\CarterQ)^{\alpha\delta} ,\\
\CKalphab^\alpha &= \MaxF_{\gamma\delta}(\vecNPNV)^\gamma (\KSigma^{-1}\CarterQ)^{\alpha\delta} . 
\end{align*}
The sections $\CKalpha$ and $\CKalphab$ are smooth if $\MaxF$ is. 

Finally, we return to the correct interpretation of $\phipm$. The components of $\CKalpha$ with respect to the $\Hnormalised,\Pnormalised$ basis are the real and imaginary components of $\phip$. The components of $\CKalphab$ with respect to the $\Hnormalised,-\Pnormalised$ basis are the real and imaginary components of $\phim$. Hence, $\phip$ and $\phim$ are isomorphic to $\CKalpha$ and $\CKalphab$. Furthermore, the modulus of the former complex numbers is equal to the length of the latter sections of the bundle $\CKBundlepm$. Since $\CKalpha$ and $\CKalphab$ are smooth, this lets us properly understand the appropriate notion of regularity for $\phipm$. From this perspective, it is clear that we should really be working with $\CKalpha$ and $\CKalphab$, instead of $\phipm$. However, for compactness of notation, we will use $\phipm$, and, if there is ever a concern about regularity at the poles, we will interpret $\phip$ and $\phim$ as $\CKalpha$ and $\CKalphab$ respectively.

%%%%%%%%%%%%%%%%%%%%%%%%%%%%%%%%%%%%%%%%%%%%%%%%%%%%%%%%%%%%%%%%%%%%%%%%%%
\subsection{Standard smallness assumptions}
\label{ss:Smallness}
In this subsection, we introduce several standard smallness assumptions on $|a|/M$ and other parameters. In particular, we will want to have quantifications on the choice of $\rfar$ and the parameter $\epsilondtsquared$, which is introduced in section \ref{s:FourierSpectral}. 

\begin{definition}
\label{def:asmall}
We say the estimates $A\lesssim B$, $A\leq C B$, or $A\leq C_1B_1
+C_2B_2$ holds for $|a|/M$ sufficiently small if $\exists C,C_1, C_2,
\epsilonSlowRotationIntro>0:$ $\forall M>0,
a\in[-\epsilonSlowRotationIntro M,\epsilonSlowRotationIntro M]:$ the estimate $A\leq CB$ or $A\leq C_1B_1 +C_2B_2$ holds respectively. 
\end{definition}

\begin{definition}
\label{def:epsilondtsquaredalsosmall}
We say the estimates $A\lesssim B$, $A\leq C B$, or $A\leq C_1B_1
+C_2B_2$ holds for $\max(|a|/\epsilondtsquared,\epsilondtsquared/M)$
sufficiently small if $\exists C,C_1, C_2,
\epsilonSlowRotationIntro>0:$ $\forall M>0,
a\in[-\epsilonSlowRotationIntro M,\epsilonSlowRotationIntro M],
\epsilondtsquared\in[-\epsilonSlowRotationIntro
  M,\epsilonSlowRotationIntro M]:$ the estimate $A\leq CB$ or $A\leq C_1B_1 +C_2B_2$ holds respectively. 
\end{definition}

\begin{definition}
\label{def:anyprojectionaway}
We say the estimates $A\lesssim B$, $A\leq C B$, or $A\leq C_1B_1 +C_2B_2$ holds for 
any projection away from the orbiting null geodesics and $|a|/M$ sufficiently small
if $\forall \rfarratio\in(0,1]:$ $\exists C,C_1, C_2,
  \epsilonSlowRotationIntro>0:$ $\forall M>0,
  a\in[-\epsilonSlowRotationIntro M,\epsilonSlowRotationIntro M], \rfar=\rfarratio M:$ the estimate $A\leq CB$ or $A\leq C_1B_1 +C_2B_2$ holds respectively. 
\end{definition}

\begin{definition}
\label{def:chooseprojection}
We say there is a projection away from the orbiting null geodesic such that for $\max(|a|/\epsilondtsquared,\epsilondtsquared/M)$ one has the estimates $A\lesssim B$, $A\leq C B$, or $A\leq C_1B_1 +C_2B_2$ 
if  $\exists  \rfarratio\in(0,1]: \exists C,C_1, C_2,
  \epsilonSlowRotationIntro>0, \epsilondtsquaredMax>0:$ $\forall M>0,
  a\in[-\epsilonSlowRotationIntro M,\epsilonSlowRotationIntro M], \epsilondtsquared\in(0,\epsilondtsquaredMax M], \rfar=\rfarratio M:$ the estimate $A\leq CB$ or $A\leq C_1B_1 +C_2B_2$ holds respectively. 
\end{definition}

\begin{definition}
\label{def:anyprojectionawayalso}
We say the estimates $A\lesssim B$, $A\leq C B$, or $A\leq C_1B_1 +C_2B_2$ holds for 
any projection away from the orbiting null geodesics and $\max(|a|/\epsilondtsquared,\epsilondtsquared/M)$ sufficiently small
if $\forall \rfarratio\in(0,1]:$ $\exists C,C_1, C_2,
  \epsilonSlowRotationIntro>0, \epsilondtsquaredMax>0:$ $\forall M>0,
  a\in[-\epsilonSlowRotationIntro M,\epsilonSlowRotationIntro M], \epsilondtsquared\in(0,\epsilondtsquaredMax M], \rfar=\rfarratio M:$ the estimate $A\leq CB$ or $A\leq C_1B_1 +C_2B_2$ holds respectively. 
\end{definition}

%%%%%%%%%%%%%%%%%%%%%%%%%%%%%%%%%%%%%%%%%%%%%%%%%%%%%%%%%%%%%%%%%%%%%%%%%
\section{Charge conservation}
\label{s:Cohomology}
\subsection{Charge, cohomology, and determination of the bound state}
The electric and magnetic charges evaluated on any two surface
$\TwoSfc$ are defined to be, respectively,
\begin{align*}
\ChargeE[\TwoSfc]&=\frac{1}{4\pi}\int_{\TwoSfc} \Hodge\MaxF ,&
\ChargeB[\TwoSfc]&=\frac{1}{4\pi}\int_{\TwoSfc} \MaxF . 
\end{align*}
Note that there is no need to specify a volume form or measure, since
these are each the integrals of a two form over a two surface. We are
particularly interested in evaluating these on closed two surfaces of
constant $(t,r)$ denoted by $\TwoSfc(t,r)$. For $r>\rp$,
$\TwoSfc(t,r)$ is a topological sphere. (We will only consider the
exterior region, since some of these arguments fail in the interior,
particularly at $r=0$.)

Since any $\TwoSfc(t_1,r_1)$ can be continuously deformed to any other
$\TwoSfc(t_2,r_2)$, from Stokes's theorem and the Maxwell equations,
which state that both $\MaxF$ and $\Hodge\MaxF$ are closed two forms,
it follows that the electric and magnetic charges are equal on all
$\TwoSfc(t,r)$. Thus, the charges are equal when evaluated on any closed
two surface that encloses the black hole. Note that, in contrast with
the situation in Minkowski space, $\Reals^{1+3}$, since the second
homology class of the exterior region,
$H^2(\Reals\times(\rp,\infty)\times\Stwo)$ is nontrivial, and, in
particular, the $\TwoSfc(t,r)$ are in a nontrivial equivalence class,
the charge need not be zero.

For any $\ChargeConstant\in\Complex$, we define the Coulomb solutions
to be given by
\begin{align}
\phiz&=\frac{\ChargeConstant}{\Kp^2}, &
\phipm&=0 .
\label{eq:ChargedSolutions}
\end{align}
which is a solution of the Maxwell system. By taking the limit at fixed $t$ as $r\rightarrow\infty$, one can easily verify that for the Coulomb solutions
\begin{align*}
\ChargeConstant&=\ChargeE+i\ChargeB . 
\end{align*}
Since $H^2(\Reals\times(\rp,\infty)\times\Stwo)=H^2(\Stwo)=\Complex$,
it follows that the cohomology class of two forms is
$H_2(\Reals\times(\rp,\infty)\times\Stwo)=H_2(\Stwo)=\Complex$, so the
Coulomb solutions provide a representative for every equivalence
class.

Since each charge on $\TwoSfc(t,r)$ is constant as a function of
$(t,r)$, the charges can be determined from the initial data. Since the
charges and the Maxwell equations are linear, given a solution, we can
determine the charges $\ChargeE$ and $\ChargeB$, construct the
corresponding Coulomb solutions, and generate a new solution by
subtracting this Coulomb solution from the original solution. This
proves the first part of proposition \ref{prop:ChargeDecomposition}.

\subsection{Regularity conditions}
\label{ss:RegularityConditions}
To avoid repetition, we now introduce regularity conditions which will
be assumed through out this paper. 

\begin{definition}
\label{def:FRegular}
$\MaxF\in\wedge^2$ is a regular, charge-free solution of the Maxwell
equation \eqref{eq:Maxwell} if the charge is zero on every sphere, $\MaxF$ is $C^2$ in the exterior of the Kerr space-time, $\MaxF$
has a $C^2$ extension to the closure of the exterior in the maximal
analytic extension, $\MaxF$ vanishes on $\hst{0}$ for sufficiently
large $r$, and $\MaxF$ is a
solution of the Maxwell equation \eqref{eq:Maxwell}.

$\MaxF\in\wedge^2$ is a regular solution of the Maxwell
equation \eqref{eq:Maxwell} if it is the sum of a Coulomb solution and a regular solution. 
\end{definition}

\begin{definition}
\label{def:UpsilonRegular}
$\solu$ is regular in the sense of solutions of the Fackerell-Ipser
equation if it is $C^2$ in the exterior of the Kerr space-time, has a
$C^2$ extension to the closure of the exterior in the maximal analytic
extension, and vanishes on $\hst{0}$ for sufficiently large $r$. 
\end{definition}

\begin{remark}
If $\MaxF$ is regular, then $\solu=\solu[\MaxF]$ is regular. It is for this reason that we require $\MaxF\in C^2$ instead of $\MaxF\in C^1$.

If $\MaxF$ or $\solu$ is regular, then they can be understood as
classical solutions of the Maxwell equation or the Fackerell-Ipser
equation respectively. For simplicity, this is the only type of
solution we will consider. We expect that the set of regular solutions is dense in the energy space, as is the case in Minkowski space. We will ignore this issue, since the proof would require elliptic theory. 

\end{remark}

\begin{definition}
\label{def:VFRegular}
A vector field $X\in\Gamma$ is regular if it is $C^1$ in the exterior of the Kerr space-time and has a continuous extension to the closure of the exterior in the maximal analytic extension. 

A coefficient $f$ is regular if it is $C^3$ in the exterior of the Kerr space-time and has a continuous extension to the closure of the exterior in the maximal analytic extension. 
\end{definition}

\begin{remark}
Regularity for a vector field is considerably weaker than the level of regularity that we assume for solutions of the Maxwell or Fackerell-Ipser equations. It is sufficient to apply the divergence theorem. We wish coefficients to be more regular so we can apply the d'Alembertian to the derivative of the coefficient. 
\end{remark}

%%%%%%%%%%%%%%%%%%%%%%%%%%%%%%%%%%%%%%%%%%%%%%%%%%%%%%%%%%%%%%%%%%%%%%%%%%
\subsection{An $L^2$ estimate on the spherical mean of the spin zero component}
In the Schwarzschild space-time, the Coulomb solutions are the only
spherically symmetric solutions. Furthermore, the spherically
symmetric functions are the only functions on the sphere with mean
zero. That fact allows one to prove a lower bound on the spherical
$L^2$ norm of charge-free solutions in terms of the spherical $L^2$
norm of their angular derivatives. In particular, the spherical
Laplacian can be bounded below by $\ell(\ell+1)=2$. The purpose of
this subsection is to provide a similar estimate in the Kerr
space-time. This estimate will be a crucial part of the Morawetz
estimate in subsection \ref{ss:FIMorawetzBasic}.

\begin{lemma}[Lower bound on angular derivatives of charge-free solution]
\label{lemma:LowerBoundForAngularDerivatives}
\hypothesisA, 
\hypothesisFMax, 
then for $t\in\Reals$ and $r>\rp$, 
\begin{align*}
2&\int_{\TwoSfc(t,r)} |\phiz|^2 \diomega\\
&\leq \int_{\TwoSfc(t,r)} |\pAng \phiz|^2 \diomega 
-a^2\potlL \int_{\TwoSfc(t,r)}
  |\phip+\phim|^2 \diomega .
\end{align*}
\end{lemma}

\begin{proof}
The charge are given by
\begin{align*}
\ChargeE+i\ChargeB
&=\int_{\TwoSfc(t,r)}\Hodge\MaxF+i\MaxF \\
&=\int_{\TwoSfc(t,r)}(\Hodge\MaxF+i\MaxF)(\ph,\pp)\di\theta\di\phi \\
&=\int_{\TwoSfc(t,r)}\left(\sqrt{\frac{\KgMetric(\ph,\ph)\KgMetric(\pp,\pp)}{\KgMetric(\vecTperp,\vecTperp)\KgMetric(\pr,\pr)}}\MaxF(\vecTperp,\pr)+i\MaxF(\ph,\pp)\right)\di\theta\di\phi
\\
&=\int_{\TwoSfc(t,r)}\left(\frac{\KPi\sin\theta}{\KSigma}\MaxF(\vecTperp,\pr)+i\MaxF(\ph,\pp)\right)\di\theta\di\phi .
\end{align*}
By introducing the transition matrix between the coordinate basis and
the tetrad \eqref{eq:TetradForCharge}, inverting the transition
matrix, expanding the basis $\{\vecTperp,\pr,\ph,\pp\}$ in terms
of the Carter tetrad, we find
\begin{align*}
\ChargeE+i\ChargeB
&= \frac12 \int_{\TwoSfc(t,r)} \left( -(r^2+a^2)\phiz +ia\sin\theta\sqrt{\KDelta}(\phip+\phim) \right)\diomega .
\end{align*}
If the two charges vanish, then 
\begin{align}
\int_{\TwoSfc(t,r)} (r^2+a^2)\phiz \diomega 
&=ia\sqrt{\KDelta} \int_{\TwoSfc(t,r)} (\sin\theta(\phip+\phim) )\diomega .
\label{eq:ChargeFreeConditionIntermediate}
\end{align}

Given $\phiz$ on a sphere, we can decompose it into the
spherically symmetric component and the remainder, $u_s$ and
$u_r$. These components are orthogonal in $L^2$, so
\begin{align*}
\int_{\TwoSfc(t,r)} |u|^2 \diomega
&=\int_{\TwoSfc(t,r)} |u_s|^2 \diomega
+\int_{\TwoSfc(t,r)} |u_r|^2 \diomega .
\end{align*}
Since $2$ is the first nonzero eigenvalue of the spherical Laplacian,
the non-spherically-symmetric component satisfies
\begin{align*}
2\int_{\TwoSfc(t,r)} |u_r|^2 \diomega
&\leq \int_{\TwoSfc(t,r)} |\pAng u_r|^2 \diomega 
= \int_{\TwoSfc(t,r)} |\pAng \phiz|^2 \diomega 
.
\end{align*}
From equation \eqref{eq:ChargeFreeConditionIntermediate}, the
spherically-symmetric component satisfies
\begin{align*}
u_s(t,r)&=\frac{ia\sqrt{\KDelta}}{r^2+a^2}\frac{1}{4\pi} \int_{\TwoSfc(t,r)}
(\phip+\phim) \diomega . 
\end{align*}
From the Cauchy-Schwarz inequality, 
\begin{align*}
|u_s(t,r)|^2
&\leq \frac{a^2\KDelta}{(r^2+a^2)^2} \frac{1}{(4\pi)^2} \int_{\TwoSfc(t,r)}
  |\phip+\phim|^2 \diomega (4\pi) \\
&= \frac{a^2\KDelta}{(r^2+a^2)^2} \frac{1}{4\pi} \int_{\TwoSfc(t,r)}
  |\phip+\phim|^2 \diomega .
\end{align*}
Integrating this constant function over the sphere $\TwoSfc(t,r)$ with
respect to $\diomega$ eliminates the factor of $(4\pi)^{-1}$ on the
right. 
\end{proof}

\subsection{Almost orthogonality of the charged and uncharged parts}
\label{ss:AlmostOrthogonal}
The purpose of this subsection is to complete the proof of proposition \ref{prop:ChargeDecomposition}. Given two possibly charged solutions, $\MaxF$ and $\MaxH$ with components $\phii$ and $\psii$ respectively, let
\begin{align*}
\langle \MaxH,\MaxF\rangle_{\hst{t}}
=\inthst{t} \Big(\sum_i \bar{\psii}\phii\Big) \diThreeNaive& \\
+\inthst{t} \Big(\frac{r^2+a^2}{\KDelta}(\pt(\overline{\Kp\psiz}))&(\pt(\Kp\psiz)) +\frac{\KDelta}{r^2+a^2}(\pr(\overline{\Kp\psiz}))(\pr(\Kp\phiz))\\ 
&+\frac{(\pAng(\overline{\Kp\psiz}))\cdot(\pAng(\Kp\phiz))}{r^2} +\frac{\KSigma}{r^2}\bar\psiz\phiz\Big) \diThreeNaive .
\end{align*}
This defines a nondegenerate, nonnegative quadratic form, so $\langle\MaxF,\MaxF\rangle^{1/2}$ is a norm. 

The sum of the energies for a stationary solution with $\MaxFStatic$, for which the middle component is $\ChargeConstant/\Kp^2$ and the extreme components vanish, satisfies
\begin{align*}
\EnergyCrudeFArg{\MaxFStatic}+\EnergyCrudeFIArg{\soluStatic}
\sim \frac{|\ChargeConstant|^2}{M} .
\end{align*}

The sum of the energies $\EnergyCrudeF+\EnergyCrudeFI$ is $\langle\MaxF,\MaxF\rangle$. Thus, if $\MaxFTotal$ is decomposed into stationary and charge-free components, $\MaxFStatic$ and $\MaxF$, then
\begin{align*}
\EnergyCrudeFArg{\MaxFTotal}+\EnergyCrudeFIArg{\MaxFTotal}
&=\EnergyCrudeFArg{\MaxFStatic}+\EnergyCrudeFIArg{\soluStatic}
+\EnergyCrudeFArg{\MaxF}+\EnergyCrudeFIArg{\MaxF}\\
&\quad+2\Re\langle\MaxFStatic,\MaxF\rangle. 
\end{align*}
The inner product is
\begin{align*}
&\inthst{t} \frac{\bar{\ChargeConstant}}{\bar{\Kp}^2}\phiz \diThreeNaive \\
&+\inthst{t}  \left(\frac{\KDelta}{r^2+a^2}\left(\pr\frac{\bar{\ChargeConstant}}{\bar{\Kp}}\right)(\pr(\Kp\phiz))
+\frac{\left(\ph\frac{\bar{\ChargeConstant}}{\bar{\Kp}}\right)\ph(\Kp\psiz)}{r^2} +\frac{\KSigma}{r^2}\frac{\bar{\ChargeConstant}}{\bar{\Kp}^2}\phiz\right) \diThreeNaive \\
&=\bar{q}\left(\inthst{t} 2\phiz \diThreeNaiveFI +\inthst{t}\left(\frac{r^2+\KSigma}{\bar{\Kp}^2} -2\right) \phiz \diThreeNaiveFI \right)\\
&\quad+\bar{q}\left(\inthst{t}  -\Kp\left(\pr \frac{r^2\KDelta}{r^2+a^2}\pr\frac{1}{\bar{\Kp}}\right) \phiz \diThreeNaiveFI +\inthst{t}\ph\left(\frac{1}{\bar{\Kp}}\right)\ph(\Kp\phiz) \diThreeNaiveFI\right) .
\end{align*}
The first of these integrals can be estimated using the charge-free condition. 
\begin{align*}
\left|\inthst{t} 2\phiz \diThreeNaiveFI\right|
&\lesssim \inthst{t} a|\phipm|\sqrt{\potlL} \diThreeNaiveFI \\
&\lesssim |a|\left(\inthst{t} |\phipm|^2\diThreeNaive\right)^{1/2} \left(\inthst{t}\potlL r^{-2} \diThreeNaiveFI\right)^{1/2} \\
&\lesssim a M^{-3/2} \EnergyCrudeF^{1/2} .
\end{align*}
Since $(r^2+\KSigma)\bar{\Kp}^{-2}-2$ is a homogeneous rational function of degree zero in $r$, $M$, $a$, and $a\cos\theta$, and it vanishes at $a=0$, this rational function is bounded by $a r^{-2}$. Thus,
\begin{align*}
\left|\inthst{t}\left(\frac{r^2+\KSigma}{\bar{\Kp}^2} -2\right) \phiz \diThreeNaiveFI\right|
&\lesssim |a| \left(\inthst{t} |\phiz|^2 \diThreeNaive\right)^{1/2}\left(\inthst{t} r^{-4}\diThreeNaiveFI\right)^{1/2} \\
&\lesssim a M^{-3/2} \EnergyCrudeF .
\end{align*}
The integral of $-\Kp(\pr (r^2\KDelta(r^2+a^2)^{-1})\pr\bar{\Kp}^{-1})\phiz$ can be treated similarly by breaking the coefficient into a spherically symmetric part and a part which vanishes linearly in $a$. The first part can be bounded using the charge-free condition, and the second can be bounded by using the additional decay in $r$ in the remainder. The integral of $(\ph\bar{\Kp}^{-1})(\ph(\Kp\phiz))$ can be bounded directly using the Cauchy-Schwarz inequality, since $|\ph\bar{\Kp}^{-1}|\lesssim |a| r^{-2}$. Combining all these results, one finds
\begin{align*}
\left|\langle\MaxFStatic,\MaxF\rangle\right|
&\lesssim \frac{|a|}{M} \frac{|\ChargeConstant|}{M^{1/2}}(\EnergyCrudeF+\EnergyCrudeFI)^{1/2} \\
&\lesssim \frac{|a|}{M} \left(\EnergyCrudeFArg{\MaxFStatic}+\EnergyCrudeFIArg{\soluStatic}
\right)^{1/2}(\EnergyCrudeF+\EnergyCrudeFI)^{1/2} .
\end{align*}
Thus, 
\begin{align*}
\EnergyCrudeFArg{\MaxFTotal}+\EnergyCrudeFIArg{\MaxFTotal}
&\sim \EnergyCrudeFArg{\MaxFStatic}+\EnergyCrudeFIArg{\soluStatic}
+\EnergyCrudeFArg{\MaxF}+\EnergyCrudeFIArg{\MaxF} ,
\end{align*}
and, in fact, the two sides are $a$ almost equivalent. 

This completes the proof of proposition \ref{prop:ChargeDecomposition}.

%%%%%%%%%%%%%%%%%%%%%%%%%%%%%%%%%%%%%%%%%%%%%%%%%%%%%%%%%%%%%%%%%%%%%%%%%%%%
\section{Estimates for the Maxwell field}
\label{s:Maxwell}
The purpose of this section is to prove core estimates \eqref{eq:CoreI}-\eqref{eq:CoreII}, which involve the Maxwell Field, $\MaxF$, and which can be proved without reference to the Fackerell-Ipser equation. Subsection \ref{ss:MaxwellVF} provides a brief review of the vector-field method; subsection \ref{ss:MaxwellACE} provides the proof of the almost conservation of energy result, core estimate \eqref{eq:CoreI}; and subsection \ref{ss:MaxwellMorawetz} provides the Morawetz estimate for the Maxwell field, core estimate \eqref{eq:CoreII}. 
 
\subsection{The stress-energy tensor for the Maxwell field}
\label{ss:MaxwellVF}
The material in this subsection is well known, \cf \eg \cite{ChristodoulouKlainerman:DecayOfLinearFields}. 

\begin{definition}
Given $F\in\wedge^2$, the Maxwell Lagrangian and energy-momentum tensor are
\begin{align*}
\Lagrangian&=\frac14 \MaxF_{\alpha\beta}\MaxF^{\alpha\beta}, &
\EMS[\MaxF]_{\alpha\beta}&=\MaxF_{\alpha\gamma}\MaxF_{\beta}{}^\gamma -\frac14\gMetric_{\alpha\beta}\MaxF_{\gamma\delta}\MaxF^{\gamma\delta} .
\end{align*}
\end{definition}

\begin{definition}
Given a regular vector field $\vecX$, $t\in\Reals$, $t_1<t_2$, and a regular $\MaxF\in\wedge^2$, the $4$-momentum generated by $\vecX$ of $\MaxF$, the energy  generated by $\vecX$ of $\MaxF$ and evaluated on $\hst{t}$, and the associated bulk term are
\begin{align*}
\GenMomentum{\vecX}[\MaxF]&=\EMS[\MaxF]_{\alpha\beta}\vecX^\beta ,\\
\GenEnergy{\vecX}[\MaxF](t)&=\inthst{t} \GenMomentum{\vecX}[\MaxF]_\alpha\diNormal{\alpha}{\hst{t}} ,\\
\GenBulk{\vecX}[\MaxF](t_1,t_2)&=\int_{\BulkZone{t_1}{t_2}} \EMS[\MaxF]_{\alpha\beta}\nabla^{(\alpha}\vecX^{\beta)} \diFourNatural. 
\end{align*}
\end{definition}

\begin{theorem}[Properties of the Maxwell energy-momentum]

\begin{enumerate}
\item{} [Dominant energy condition] If $\vecX, \vecY$ are time-like, future-pointing vectors and $\MaxF\in\wedge^2$, then $\EMS[\MaxF]_{\alpha\beta}\vecX^\alpha\vecY^\beta\geq0$. 
\item{} [Symmetry] If $\MaxF\in\wedge^2$, then $\EMS[\MaxF]_{\alpha\beta}=\EMS[\MaxF]_{\beta\alpha}$. 
\item{} [Trace free] If $\MaxF\in\wedge^2$, then $\EMS[\MaxF]_{\alpha}{}^\alpha=0$.
\item{} [Divergence free] If $\MaxF\in\wedge^2$ is a regular solution of the Maxwell equation \eqref{eq:Maxwell}, then $\nabla^\alpha\EMS[\MaxF]_{\alpha\beta}=0$. 
\end{enumerate}
\end{theorem}

\begin{corollary}[Energy generation properties]
Let $\MaxF\in\wedge^2$ be a regular solution of the Maxwell equation \eqref{eq:Maxwell} and $\vecX$ be a regular vector field. 
\begin{enumerate}
\item{} [Energy Generation 1] If $\vecX$ is time-like and future oriented, then $\forall t\in\Reals:$ 
\begin{align*}
\GenEnergy{\vecX}[\MaxF](t)\geq0 .
\end{align*}
\item{} [Energy Generation 2] If $t_2>t_1$, then 
\begin{align*}
\GenEnergy{\vecX}[\MaxF](t_2)-\GenEnergy{\vecX}[\MaxF](t_1)
&=\GenBulk{\vecX}[\MaxF](t_1,t_2) .
\end{align*}
\end{enumerate}
\end{corollary}
\begin{proof}
The first property follows from the normal to $\hst{t}$ being time-like and future oriented and from $\EMS$ satisfying the dominant energy condition. The second follows from the divergence theorem, the divergence free property of $\EMS$, and the vanishing of the fluxes through $\bif$ and $\spacelikeinfty$, which follows from the regularity of $\MaxF$, $\vecX$, and, hence, $\GenMomentum{\vecX}[\MaxF]$. 
\end{proof}

Because of the trace-free property of the energy-momentum tensor for the Maxwell field, there is the following formula for the bulk term, which frequently simplifies calculations. 

\begin{theorem}
If $\MaxF\in\wedge^2$ is regular, $\vecX$ is a regular vector field, $t_1<t_2$, and $\Omega$ is a regular function that is positive in the exterior of the Kerr space-time, then
\begin{align*}
\GenBulk{\vecX}(t_1,t_2)=-\frac12\int_{\BulkZone{t_1}{t_2}} \Omega^{-2}\EMS[\MaxF]_{\alpha\beta}\Lie_{\vecX}(\Omega^2\gMetric^{\alpha\beta}) \diFourNatural. 
\end{align*}
\end{theorem}
\begin{proof}
It is well known that 
\begin{align*}
\nabla^{(\alpha}\vecX^{\beta)}=-(1/2)\Lie_{\vecX}\gMetric^{\alpha\beta}. 
\end{align*}
Since $\Omega^2$ is positive in the exterior of the Kerr space-time, one can write $\Lie_{\vecX}g^{\alpha\beta}=(\vecX\Omega^{-2})\Omega^2\gMetric_{\alpha\beta}+\Omega^{-2}\Lie_{\vecX}(\Omega^2\gMetric^{\alpha\beta})$. Since $\EMS$ is trace free, one finds the contraction with $\EMS$ of the first of these terms is $\EMS_{\alpha\beta}\Lie_{\vecX}\gMetric^{\alpha\beta}=\EMS_{\alpha\beta}\Omega^{-2}\Lie_{\vecX}(\Omega^2\gMetric^{\alpha\beta})$. 
%This follows from the previous theorem, the product rule, and the trace-free property of the $\EMS$. 
\end{proof}

\begin{theorem}[Components of Maxwell $\EMS$]
\label{lemma:MaxFEMSComponents}
If $\MaxF\in\wedge^2$, then
\begin{align*}
\EMS(\vecLPNV,\vecLPNV)&=2|\phip|^2, \\
\EMS(\vecLPNV,\vecNPNV)&=|\phiz|^2, \\
\EMS(\vecNPNV,\vecNPNV)&=2|\phim|^2, \\
\EMS(\Tnormalised,\Tnormalised)&=\sum_i|\phii|^2, \\
\EMS(\Rnormalised,\Rnormalised)&=\left(|\phipm|^2-|\phiz|^2\right) \\
\EMS_{\alpha\beta}\left(\Hnormalised^\alpha\Hnormalised^\beta+\Pnormalised^\alpha\Pnormalised^\beta\right)&=2|\phiz|^2.
\end{align*}
\end{theorem}
%(In this last theorem, our choice of normalising factors in the definition of $\vecLPNV$, $\vecNPNV$, and the $\phii$, is different from other works in the field, such as \cite{CK:LinearFields}). \mnote{Check this parenthetical comment.}
\begin{proof}Direct computation.\end{proof}

\subsection{Almost energy conservation for the Maxwell equation}
\label{ss:MaxwellACE}
Recall the crude Maxwell energy from the introduction
\begin{align*}
\EnergyCrudeF(t)&=\inthst{t} \sum_i|\phii|^2 \diThreeNaive
\end{align*}
and the blended time-like vector field $\vecTBlend$ from subsection \ref{ss:TimelikeVectors}. 

\begin{lemma}[Energy equivalence for the Maxwell field]
\label{lemma:MaxwellEnergyEquivalence}
There are positive constants $\epsilonSlowRotationIntro,C>0$ such that $\forall M>0,a\in[-\epsilonSlowRotationIntro M,\epsilonSlowRotationIntro M]$ one has that if $\MaxF\in\wedge^2$ is regular and $t\in\Reals$, then
\begin{align*}
\left(1-C\frac{|a|}{M}\right)\GenEnergy{\vecTBlend}[\MaxF](t)
\leq \EnergyCrudeF(t)
\leq\left(1+C\frac{|a|}{M}\right)\GenEnergy{\vecTBlend}[\MaxF](t) .
\end{align*}
That is, $\EnergyCrudeF$ and $\GenEnergy{\vecTBlend}$ are $a$ almost equivalent. 
\end{lemma}
\begin{proof}
First, observe that for $\vecX,\vecY\in\{\Tnormalised,\Rnormalised,\Hnormalised,\Pnormalised\}$, one has
\begin{align*}
|\EMS(\vecX,\vecY)| \lesssim \sum_i|\phii|^2 .
\end{align*}
Now, consider the $\vecTBlend$ energy and how the vectors in its definition differ from $\vecTPNV$:
\begin{align*}
\GenEnergy{\vecTBlend}[\MaxF](t)
&=\inthst{t} \EMS[\MaxF]_{\alpha\beta}\vecTBlend^\beta\diNormal{\alpha}{\hst{t}} \\
&=\inthst{t} \EMS_{\alpha\beta}\vecTBlend^\beta\vecTperp^\alpha\frac{\KPi}{\KDelta} \di r\diomega\\
&=\inthst{t} \EMS_{\alpha\beta}\vecTPNV^\beta\vecTPNV^\alpha\frac{\KPi}{\KDelta} \di r\diomega\\
&\quad+\inthst{t} \EMS_{\alpha\beta}(\omegaBlend-\omegaPNV)\pp^\beta\vecTPNV^\alpha\frac{\KPi}{\KDelta} \di r\diomega\\
&\quad+\inthst{t} \EMS_{\alpha\beta}(\omegaPerp-\omegaPNV)\pp^\beta\vecTPNV^\alpha\frac{\KPi}{\KDelta} \di r\diomega\\
&\quad+\inthst{t} \EMS_{\alpha\beta}(\omegaBlend-\omegaPNV)(\omegaPerp-\omegaPNV)\pp^\alpha\pp^\beta \frac{\KPi}{\KDelta} \di r\diomega.
%&=\inthst{t} \EMS_{\alpha\beta}\Tnormalised^\beta\Tnormalised^\alpha\gMetric(\vecTPNV,\vecTPNV)\frac{\KPi}{\KDelta} \di r\diomega\\
%&\quad+\inthst{t} \EMS_{\alpha\beta}(\omegaBlend-\omegaPNV)\pp^\beta\vecTPNV^\alpha\frac{\KPi}{\KDelta} \di r\diomega\\
%&\quad+\inthst{t} \EMS_{\alpha\beta}(\omegaPerp-\omegaPNV)\pp^\beta\vecTPNV^\alpha\frac{\KPi}{\KDelta} \di r\diomega .
\end{align*}
The difference between the factor $g(\vecTPNV,\vecTPNV)\KPi\KDelta^{-1}r^{-2}$ and $1$ is $a$ small. Thus, since $r^{-1}\pp$ and $\vecTPNV$ can be expressed as bounded linear combinations of $\Tnormalised$, $\Rnormalised$, $\Hnormalised$, and $\Pnormalised$ and since $\omegaBlend-\omegaPNV$ and $\omegaPerp-\omegaPNV$ vanish linearly in $\KDelta$ and quadratically in $r^{-1}$ and are $a$ small, one finds
\begin{align*}
|\GenEnergy{\vecTBlend}[\MaxF](t)-\EnergyCrudeF(t)|
&\lesssim \frac{|a|}{M}\inthst{t}\sum_i |\phii|^2 \diThreeNaive \\
&\lesssim \frac{|a|}{M}\EnergyCrudeF(t). 
\end{align*}
\end{proof}

\begin{remark}
\label{rem:EnergyDegeneracy}
While the $\phii$ are defined with respect to a nondegenerate basis, the energies $\EnergyCrudeF(t)$ and $\GenEnergy{\vecTBlend}[\MaxF](t)$ are both degenerate as $r\rightarrow\rp$. In $\EnergyCrudeF$, the volume form $\di r$ is degenerate. In $\GenEnergy{\vecTBlend}[\MaxF](t)$, the vector field $\vecTBlend$ is degenerate. Both $\di r$ and $\vecTBlend$ degenerate as $(\KDelta/r^2)^{1/2}$ as $r\rightarrow\rp$, which allows for the equivalence of the energies. 
\end{remark}

\begin{prop}[Almost energy conservation for the Maxwell field, core estimate \eqref{eq:CoreI}]
\label{prop:MaxwellACE}
\hypothesisA, 
\hypothesisFMax, 
then $\forall t_1<t_2$
\begin{align*}
\EnergyCrudeF(t_2)
\lesssim \frac{|a|}{M}\left(\BulkMorawetzF+\BulkFILowerOrder\right) +\EnergyCrudeF(t_1) .
\end{align*}
\end{prop}
\begin{proof}
From the second energy generation property,
\begin{align*}
\GenEnergy{\vecTBlend}(t_2)-\GenEnergy{\vecTBlend}(t_1)
&=\int_{\BulkZone{t_1}{t_2}} \EMS_{\alpha\beta}\nabla^{(\alpha}\vecTBlend^{\beta)} \diFourNatural .
\end{align*}
Since $\pt$ and $\pp$ are Killing, the symmetric derivative of $\vecTBlend$ is $(\pr\fnBlend)a(\rp^2+a^2)^{-1}\pr^{(\alpha}\pp^{\beta)}$. Since $\pr\fnBlend$ is supported in $[10M,11M]$ and has a bounded derivative, it can be bounded by a multiple of any positive function. Since $\pr$ and $r^{-1}\pp$ can be expressed as bounded combinations of $\Tnormalised$, $\Rnormalised$, $\Hnormalised$, and $\Pnormalised$ for $r\in[10M,11M]$, one finds
\begin{align*}
\left| \EMS_{\alpha\beta} \nabla^{(\alpha}\vecTBlend^{\beta)}\right| 
&\lesssim \frac{|a|}{M} \frac{\KDelta}{(r^2+a^2)^2} \sum_i|\phii|^2 .
\end{align*}
Integrating this and the equivalence of $\EnergyCrudeF$ and $\GenEnergy{\vecTBlend}[\MaxF]$ provide the desired result. 
\end{proof}

\subsection{The Morawetz estimate for the Maxwell field}
\label{ss:MaxwellMorawetz}
\begin{definition}
Given $M>0$, $a\in[-M,M]$, the oversimplified Morawetz vector field is defined to be
\begin{align*}
\vecMorawetzF&=\fnMorawetzF\pr ,\\
\fnMorawetzF&=-\frac{\KDelta}{r^2+a^2} \left(1-\frac{3M}{r}\right)
\end{align*}
\end{definition}
The purpose of this section is to prove core estimate \eqref{eq:CoreII}. To the best of our knowledge, it is not possible to choose a vector field $\vecMorawetzF$ for which $\GenBulk{\vecMorawetzF}>0$ when $\MaxF\not=0$ and for which $\GenEnergy{\vecMorawetzF}\leq\GenEnergy{\vecTBlend}$. Instead, we choose one for which $\GenBulk{\vecMorawetzF}+\BulkFILowerOrder\geq\BulkMorawetzF$. Because we are willing to accept $\BulkFILowerOrder$ on the left, we can be quite naive in choosing $\vecMorawetzF$. 

\begin{lemma}[Bound for the Maxwell bulk term]
\hypothesisB, 
\hypothesisFMax,
then $\forall\FinalTime>0:$
\begin{align*}
C_1\GenBulk{\vecMorawetzF}+C_2\BulkFILowerOrder
&\geq\BulkMorawetzF .
\end{align*}
\end{lemma}
\begin{proof}
Let
\begin{align*}
\Omega^{-2}&=\frac{\KSigma\KDelta}{(r^2+a^2)^2}, 
\end{align*}
so that
\begin{align*}
&\Omega^{-2}\KgMetric^{\alpha\beta}\\
&=\left(\frac{\KDelta}{r^2+a^2}\right)^2\pr^\alpha\pr^\beta
-\pt^\alpha\pt^\beta-\frac{4aMr}{(r^2+a^2)^2}\pt^{(\alpha}\pp^{\beta)}+\frac{\KDelta-a^2}{(r^2+a^2)^2}\pp^\alpha\pp^\beta +\potlL\OpQ^{\alpha\beta} \\
&=\left(\frac{\KDelta}{r^2+a^2}\right)^2\pr^\alpha\pr^\beta
-\pt^\alpha\pt^\beta\\
&\quad-\frac{2a}{r^2+a^2}\pt^{(\alpha}\pp^{\beta)}-\frac{a^2}{(r^2+a^2)^2}\pp^\alpha\pp^\beta +\potlL\left(\vecH^\alpha\vecH^\beta+\vecPPNV^\alpha\vecPPNV^\beta\right) 
\end{align*}
and
\begin{align*}
&\Lie_{\vecMorawetzF}(\Omega^{-2}\KgMetric)^{\alpha\beta}\\
&=\left(\fnMorawetzF\pr\left(\frac{\KDelta^2}{(r^2+a^2)^2}\right)-2\frac{\KDelta^2}{(r^2+a^2)^2}\pr\fnMorawetzF\right)\pr^\alpha\pr^\beta
+\fnMorawetzF\pr\potlL\left(\vecH^\alpha\vecH^\beta+\vecPPNV^\alpha\vecPPNV^\beta\right) \\
&\quad-\fnMorawetzF\pr\frac{2a}{r^2+a^2}\pt^{(\alpha}\pp^{\beta)}-\fnMorawetzF\pr\frac{a^2}{(r^2+a^2)^2}\pp^\alpha\pp^\beta\\
&=-2\left(\frac{\KDelta}{r^2+a^2}\right)^3\left(\pr\frac{r^2+a^2}{\KDelta}\fnMorawetzF\right)\pr^\alpha\pr^\beta 
+\fnMorawetzF\pr\potlL\left(\vecH^\alpha\vecH^\beta+\vecPPNV^\alpha\vecPPNV^\beta\right) \\
&\quad+\frac{4ar}{(r^2+a^2)^2}\left(\pt+\frac{a}{r^2+a^2}\pp\right)^{(\alpha}\pp^{\beta)} .
\end{align*}
The contraction with each of these terms with $\EMS$ can now be computed. 

Since
\begin{align*}
-2\left(\frac{\KDelta}{r^2+a^2}\right)^3\left(\pr\frac{r^2+a^2}{\KDelta}\fnMorawetzF\right)
&=\frac{6M}{r^2}\left(\frac{\KDelta}{r^2+a^2}\right)^3 ,
\end{align*}
and $\pr$ is $a$ almost equal to $(\KDelta/(r^2+a^2))^{1/2}\Rnormalised$, one finds
\begin{align*}
&\left|\left(-2\left(\frac{\KDelta}{r^2+a^2}\right)^3\left(\pr\frac{r^2+a^2}{\KDelta}\fnMorawetzF\right)\pr^\alpha\pr^\beta
-\frac{6M}{r^2}\left(\frac{\KDelta}{r^2+a^2}\right)^2\Rnormalised^\alpha\Rnormalised^\beta\right)\EMS_{\alpha\beta}\right|\\
&\qquad \lesssim \frac{|a|}{M} \frac{M}{r^2} \left(\frac{\KDelta}{r^2+a^2}\right)^2\sum_i|\phii|^2 .
\end{align*}
Thus, there are positive constants $C$, $D$, such that
\begin{align*}
&\left(-2\left(\frac{\KDelta}{r^2+a^2}\right)^3\left(\pr\frac{r^2+a^2}{\KDelta}\fnMorawetzF\right)\right)\pr^\alpha\pr^\beta\EMS_{\alpha\beta}\\
&\qquad\geq C\frac{M}{r^2}\left(\frac{\KDelta}{r^2+a^2}\right)^2|\phipm|^2 
-D\frac{M}{r^2}\left(\frac{\KDelta}{r^2+a^2}\right)^2|\phiz|^2 
\end{align*}

The potential $r^{-2}(1-2M/r)$ has a unique maximum at $r=3M$ where the second derivative is strictly negative. Thus, for $a$ sufficiently small, the potential $\KDelta(r^2+a^2)^{-2}=\potlL$ also has a unique maximum not more than a constant multiple of $a$ (in fact of $a^2/M$) from $r=3M$. Thus, except between this root and $3M$, the coefficient $\fnMorawetzF\pr\potlL$ is nonnegative and between the root and $3M$, this coefficient is bounded below by $|a|^2M^{-5}$. In particular, there are constants such that
\begin{align}
-\fnMorawetzF\pr\potlL
\geq C\frac{\KDelta}{r^2+a^2}\frac{1}{r^3}\left(1-\frac{3M}{r}\right)^2 -D\frac{|a|}{M}\left(\frac{\KDelta}{r^2+a^2}\right)\frac{1}{r^3} \left(\frac{M}{r}\right)^{100}. 
\label{eq:FMorawetzIntA}
\end{align}
(The exponent $100$ here has been chosen here for convenience. Away from $r=3M$, the first term on the right dominates, so that the behaviour of the second term on the right can be freely adjusted by making small adjustments in the constant $C$.) Thus, since $\EMS_{\alpha\beta}\left(\vecH^\alpha\vecH^\beta+\vecPPNV^\alpha\vecPPNV^\beta\right)=2\KSigma|\phiz|^2$, 
\begin{align}
\fnMorawetzF\pr\potlL&\left(\vecH^\alpha\vecH^\beta+\vecPPNV^\alpha\vecPPNV^\beta\right)\EMS_{\alpha\beta} \nonumber\\
&\geq \left(C\frac{\KDelta}{r^2+a^2}\frac{1}{r}\left(1-\frac{3M}{r}\right)^2 -D\frac{|a|}{M}\left(\frac{\KDelta}{r^2+a^2}\right)\frac{1}{r} \left(\frac{M}{r}\right)^{100}\right)|\phiz|^2 . 
\label{eq:FMorawetzIntB}
\end{align}

The remaining term in $\Lie_{\vecMorawetzF}(\Omega^2\KgMetric)^{\alpha\beta}$ is
\begin{align*}
\fnMorawetzF\frac{4ar}{(r^2+a^2)^2}\vecTPNV^{(\alpha}\pp^{\beta)} .
\end{align*}
The vector $\pp$ is a bounded linear combination of $\vecPPNV$ and $a\sin\theta\vecTPNV$. Because of the degeneracy of $\vecTPNV$, 
\begin{align*}
\EMS(\vecTPNV,\vecTPNV)&\lesssim \frac{\KDelta}{r^2+a^2}\sum|\phii|^2 .
\end{align*}
In the contraction, $\EMS_{\alpha\beta}\vecTPNV^\alpha\vecPPNV^\beta=\MaxF_{\alpha\gamma}\MaxF_\beta{}^\gamma\vecTPNV^\alpha\vecPPNV^\beta$, either $\gamma=\Rnormalised$, in which case $|\MaxF_{\vecTPNV\Rnormalised}|\lesssim\KDelta^{1/2}r^{-1}|\phiz|$ and $|\MaxF_{\vecPPNV\Rnormalised}|\lesssim r|\phipm|$, or $\gamma=\Hnormalised$, in which case  $|\MaxF_{\vecTPNV\Hnormalised}|\lesssim \KDelta^{1/2}r^{-1}|\phipm|$ and $|\MaxF_{\vecPPNV\Hnormalised}|\lesssim r|\phiz|$. Thus, 
\begin{align*}
|\EMS(\vecTPNV,\vecPPNV)|
&\lesssim \left(\frac{\KDelta}{r^2+a^2}\right)^{1/2} r|\EMS(\Tnormalised,\Hnormalised)| \\
&\lesssim \left(\frac{\KDelta}{r^2+a^2}\right)^{1/2} r|\phiz||\phipm| \\
&\lesssim r\left(|\phiz|^2 +\frac{\KDelta}{r^2+a^2}|\phipm|^2\right) , 
\end{align*}
so that
\begin{align*}
\left|\fnMorawetzF\frac{4ar}{(r^2+a^2)^2}\vecTPNV^{\alpha}\pp^\beta\EMS_{\alpha\beta}\right|
\lesssim \frac{|a|}{M}
\left(\frac{\KDelta}{r^2+a^2}\frac{M}{r^2}|\phiz|^2 +\left(\frac{\KDelta}{r^2+a^2}\right)^2\frac{M}{r^2}|\phipm|^2\right) .
\end{align*}
Combining this estimate with estimates \eqref{eq:FMorawetzIntA}-\eqref{eq:FMorawetzIntB} and applying the factor of $\Omega^{-2}$, one finds there are constants $C_1$ and $C_2$ such that
\begin{align*}
\GenBulk{\vecMorawetzF}
\geq C_1\int_{\BulkZone{t_1}{t_2}}\frac{M}{r^2}\frac{\KDelta}{r^2+a^2}|\phipm|^2\diFourNatural
-C_2\int_{\BulkZone{t_1}{t_2}}\frac{M}{r^2}\frac{\KDelta}{r^2+a^2}|\phiz|^2\diFourNatural. 
\end{align*}
\end{proof}

\begin{lemma}(Bound the Morawetz energy for the Maxwell field)
\hypothesisA,
\hypothesisFMax,
then $\forall t\in\Reals:$
\begin{align*}
\GenEnergy{\vecMorawetzF}\lesssim \GenEnergy{\vecTBlend} ,
\end{align*}
\end{lemma}
\begin{proof}
For $a=0$, the vector fields $(\KDelta/(r^2+a^2))\pr$ and $\pt$ degenerate relative to $\Rnormalised$ and $\Tnormalised$ at a rate of $(\KDelta/(r^2+a^2))^{1/2}$. For $a\not=0$, the additional difference between these can be written as $(|a|/M)(\KDelta/(r^2+a^2))$ times bounded linear combinations of the normalised basis vectors. Since $\fnMorawetzF=-(\KDelta/(r^2+a^2))(1-3M/r)\pr$ and the factor $-(1-3M/r)$ is bounded, one finds
one finds
\begin{align*}
\left|\GenEnergy{\vecMorawetzF}\right|
&\lesssim \inthst{t} \left|\EMS[\MaxF]_{\alpha\beta}\vecMorawetzF^\alpha\vecTperp^\beta\right| \frac{\KPi}{\KDelta} \diThreeNaiveFI \\
&\lesssim \inthst{t} \frac{\KDelta}{r^2+a^2} \sum_i|\phii|^2 \frac{\KPi}{\KDelta}\diThreeNaiveFI \\
&\lesssim \inthst{t}  \sum_i|\phii|^2 \frac{\KPi}{r^2+a^2}\diThreeNaiveFI \\
&\lesssim \GenEnergy{\vecTBlend} .
\end{align*}
\end{proof}

\begin{prop}(Morawetz estimate for the Maxwell field, core estimate \eqref{eq:CoreII}
\hypothesisA,
\hypothesisFMax,
then $\forall\FinalTime>0$,
\begin{align*}
\BulkMorawetzF\lesssim \EnergyCrudeF(\FinalTime)+\EnergyCrudeF(0) +\BulkFILowerOrder .
\end{align*}
\end{prop}
\begin{proof}
This follows from the previous two lemmas and the energy generation formula for the Maxwell field. 
\end{proof}

%%%%%%%%%%%%%%%%%%%%%%%%%%%%%%%%%%%%%%%%%%%%%%%%%%%%%%%%%%%%%%%%%%%%%%%%%
\section{The energy estimate for the Fackerell-Ipser equation}
\label{s:FI}
%%%%%%%%%%%%%%%%%%%%%%%%%%%%%%%%%%%%%%%%%%%%%%%%%%%%%%%%%%%%%%%%%%%%%%%%%
\subsection{The Fackerell-Ipser equation}
The rescaled spin-weight zero component
\begin{align*}
\solu&= \Kp \phiz
\end{align*}
satisfies the Fackerell-Ipser equation \cite{FackerellIpser}
\begin{align}
0&=\nabla^\alpha\nabla_\alpha\solu -\potlFI\solu 
\label{eq:FI}
\end{align}
where
\begin{align*}
\potlFI&=-\frac{2M}{\Kp^3} .
\end{align*}
This is equivalent to 
\begin{align}
0&=\pr\KDelta\pr\solu +\frac{\CurlyR}{\KDelta}\solu -\KSigma\potlFI\solu ,
\label{eq:RedFI}
\end{align}
where
\begin{align}
\CurlyR
&=\CurlyR(r;M,a;\pt,\pp,\OpQ)\nonumber\\
&=-(r^2+a^2)\pt^2-4aMr\pt\pp+(\KDelta-a^2)\pp^2+\KDelta\OpQ ,\\
\OpQ&=\frac{1}{\sin\theta}\ph\sin\theta\ph +\frac{\cos^2\theta}{\sin^2\theta}\pp^2+a^2\sin^2\theta\pt^2 .
\label{eq:DefR}
\end{align}

Because equation \eqref{eq:RedFI} is $\KSigma$ times equation \eqref{eq:FI}, and $\KSigma\sin\theta$ is the term $\sqrt{-\det\KgMetric}$ appearing in the volume form when expressed in Boyer-Lindquist coordinates, it is convenient to work with the simplified volume form
\begin{align*}
\diFourFI&= \di r\diomega\di t. 
\end{align*}

\subsection{Vector-field analysis and almost energy conservation}
The purpose of this subsection is to prove estimate \eqref{eq:CoreIII}. The Fackerell-Ipser equation has a complex potential, which invalidates the standard arguments from calculus of variations. Nonetheless, we can treat the Fackerell-Ipser equation as a wave equation with potential plus an $a$ small perturbation. 

\begin{definition}
The partial derivative matrix is 
\begin{align*}
\PDmat_{\alpha\beta}&= \Re(\partial_\alpha\bar{\solu}\partial_\beta\solu) .
\end{align*}
The pseudo-Lagrangian and pseudo-energy-momentum tensor for the Fackerell-Ipser equation are
\begin{align*}
\Lagrangian&=\frac12\left(\PDmat^\alpha{}_\alpha+(\Re\potlFI)\solub\solu\right), \\
\EMS_{\alpha\beta}&=\PDmat_{\alpha\beta}-\gMetric_{\alpha\beta}\Lagrangian .
\end{align*}

Given a regular vector field $\vecX$ and $t\in\Reals$, the pseudo-momentum and pseudo-energy of a regular solution of the Fackerell-Ipser equation $\solu$ are
\begin{align*}
\GenMomentum{\vecX}[\solu]_\alpha&=\EMS[\solu]_{\alpha\beta}\vecX^\beta , \\
\GenEnergy{\vecX}[\solu](t)&=\int_{\hst{t}}\GenMomentum{\vecX}[\solu]_\alpha \diNormal{\alpha}{\hst{t}} , \\
\GenBulk{\vecX}[\solu](t_1,t_2)&=-\int_{\BulkZone{t_1}{t_2}} \left(\EMS_{\alpha\beta}\nabla^{(\alpha}\vecX^{\beta)} +\vecX^\beta\nabla^\alpha\EMS_{\alpha\beta}\right) \diFourNatural .
\end{align*}
\end{definition}

\begin{theorem}[Energy generation theorem for the Fackerell-Ipser equation]
\label{thm:EnergyGenerationFI}
Let $X$ be a regular vector field. If $\solu$ is a regular solution of the Fackerell-Ipser equation, \eqref{eq:FI}, then $\forall t_1<t_2$
\begin{align*}
\nabla^\alpha\EMS_{\alpha\beta}
&= (\Im\potlFI)\Im(\solub\nabla_\beta\solu)-\Re(\nabla_\beta\potlFI)|\solu|^2/2 ,\\
\GenEnergy{\vecX}[\solu](t_2)-\GenEnergy{\vecX}[\solu](t_1)
&= \GenBulk{\vecX}[\solu](t_1,t_2).
\end{align*}
\end{theorem}
\begin{proof}
Consider first,
\begin{align*}
\nabla^\alpha(\PDmat_{\alpha\beta}-\gMetric_{\alpha\beta}\Lagrangian)
&=\Re\left((\nabla^\alpha\nabla_\alpha\solub) (\nabla_\beta\solu) +(\nabla_\alpha\solub)(\nabla^\alpha\nabla_\beta\solu) \right)\\
&\qquad-\Re((\nabla_\beta\nabla^\gamma\solub)(\nabla_\gamma\solu))\\
&\qquad -\frac12\left(\Re(\potlFI\nabla_\beta\solub\solu)+\Re(\potlFI\solub\nabla_\beta\solu)\right)\\
&\qquad-\Re((\nabla_\beta\potlFI) |\solu|^2/2) \\
&=\frac12\Re\left(2\overline{\potlFI}\solub(\nabla_\beta\solu)
-\potlFI(\nabla_\beta\solub)\solu -\potlFI\solub\nabla_\beta\solu\right)\\
&\qquad -\Re(\nabla_\beta\potlFI)|\solu|^2/2 \\
&=(\Im\potlFI)\Im(\solub\nabla_\beta\solu)
 -\Re(\nabla_\beta\potlFI)|\solu|^2/2 .
\end{align*}
This gives the first of the desired results. 

The second part follows from the divergence theorem. From our definition of $\solu$ being regular, the energy is a convergent integral. In addition, from our definition of regular for a vector field, the fluxes through the bifurcation sphere and space-like infinity are zero. 
\end{proof}

Recall $\vecTperp$ and $\vecTBlend$ from subsection \ref{ss:TimelikeVectors} and the artificial energy from the introduction
\begin{align*}
\EnergyCrudeFI(t)
&=\frac12\inthst{t} \left(\frac{r^2+a^2}{\KDelta}|\vecTperp\solu|^2 +\frac{\KDelta}{r^2+a^2}|\pr\solu|^2 +\frac{|\pAng\solu|^2}{r^2} +\frac{|\solu|^2}{r^2}\right) \diThreeNaive .
\end{align*}

\begin{lemma}[Fackerell-Ipser energy equivalence]
\label{lemma:FIEnergyEquivalence}
\hypothesisA, 
\hypothesisFsolu, 
then $\forall t\in\Reals:$
\begin{align*}
\GenEnergy{\vecTBlend}[\solu](t)
\lesssim \EnergyCrudeFI(t)
\lesssim \GenEnergy{\vecTBlend}[\solu](t) +\GenEnergy{\vecTBlend}[\MaxF](t) .
\end{align*}
\end{lemma}
\begin{proof}
Recall $\diNormal{\alpha}{\hst{t}}=\vecTperp^\alpha (\KPi/\KDelta) \diThreeNaiveFI$. Thus, since $\vecTBlend=\vecTperp+(\omegaBlend-\omegaPerp)\pp$ and $g(\vecTperp,\pp)=0$, one finds
\begin{align*}
\GenEnergy{\vecTBlend}
&=\inthst{t} \left(\PDmat_{\alpha\beta}-\frac12\gMetric_{\alpha\beta}\PDmat^\gamma{}_\gamma-\frac12\gMetric_{\alpha\beta}\Re\potlFI |\solu|^2\right)\vecTBlend^\beta\vecTperp^\alpha\frac{\KPi}{\KDelta}\diThreeNaiveFI \\
&=\inthst{t}\left(\PDmat_{\alpha\beta}-\frac12\gMetric_{\alpha\beta}\PDmat^\gamma{}_\gamma\right)\vecTperp^\alpha\vecTperp^\beta \frac{\KPi}{\KDelta}\diThreeNaiveFI \\
&\qquad +\inthst{t}\PDmat_{\alpha\beta}(\omegaBlend-\omegaPerp)\pp^\beta\vecTPerp^\alpha\frac{\KPi}{\KDelta}\diThreeNaiveFI \\
&\qquad+\inthst{t} -\frac12\gMetric_{\alpha\beta}\vecTperp^\beta\vecTperp^\alpha\Re\potlFI|\solu|^2 \frac{\KPi}{\KDelta} \diThreeNaiveFI .
\end{align*}
Since $\gMetric(\vecTperp,\vecTperp)=-\KDelta\KSigma/\KPi$, $\KSigma\sim r^2+a^2$, $\gMetric^{rr}\sim\KDelta(r^2+a^2)^{-1}$, $\gMetric^{\theta\theta}\sim r^{-2}$, and $\gMetric^{\phi\phi}\sim (r^2\sin^2\theta)^{-1}$, one finds 
\begin{align*}
\inthst{t}&\left(\PDmat_{\alpha\beta}-\frac12\gMetric_{\alpha\beta}\PDmat^\gamma{}_\gamma\right)\vecTperp^\alpha\vecTperp^\beta \frac{\KPi}{\KDelta}\diThreeNaiveFI\\
&\sim\frac12\inthst{t} \left( \frac{r^2+a^2}{\KDelta}|\vecTperp\solu|^2+\frac{\KDelta}{r^2+a^2}|\pr\solu|^2 +\frac{|\pAng\solu|^2}{r^2}\right) \diThreeNaive .
\end{align*}

Since $\omegaBlend-\omegaPerp$ vanishes linearly in $\KDelta$, cubicly in $r^{-1}$, and is $a$ small,
\begin{align*}
\left|\inthst{t} \PDmat_{\alpha\beta} (\omegaBlend-\omegaPerp)\pp^\alpha\vecTperp^\beta \frac{\KPi}{\KDelta}\diThreeNaiveFI\right| 
&\lesssim |a| \inthst{t}\left(|\vecTperp\solu|^2 +\frac{1}{r^2}|\pp\solu|^2\right)\diThreeNaive .
\end{align*}

The potential
\begin{align*}
\Re\potlFI&=-\Re\frac{2M}{\Kp^3}
=2M\frac{\Re\Kp^3}{|\Kp|^6}
=-2M\frac{r^3+3ra^2\cos^2\theta}{|\Kp|^6} 
\end{align*}
is negative, decays like $Mr^{-3}$, and is bounded by $4|p|^{-2}$. Thus,
\begin{align*}
\frac{\Re\potlFI|\solu|^2}{2} \leq 0 \leq |\phiz|^2 .
\end{align*}
On the other hand, $|\phiz|^2\sim\frac{|\solu|^2}{r^2}$, so
\begin{align*}
\frac{|\solu|^2}{r^2}
\sim |\phiz|^2
\leq \Re\potlFI|\solu|^2 +4|\phiz|^2 .
\end{align*}
Thus, the desired result holds. 
\end{proof}

\begin{lemma}[Fackerell-Ipser energy growth]
\label{lemma:FIEnergy}
\hypothesisC,
\hypothesisFsolu,
then for $\FinalTime>0$, 
\begin{align*}
\GenEnergy{\vecTBlend}[\solu](\FinalTime)
-\GenEnergy{\vecTBlend}[\solu](0)
\lesssim \frac{|a|}{M}\left(\BulkFILowerOrder +\BulkFIOne +\BulkFITwo\right). 
\end{align*}
\end{lemma}
\begin{proof}
From the Fackerell-Ipser energy generation theorem \ref{thm:EnergyGenerationFI},
\begin{align*}
\GenEnergy{\vecTBlend}[\solu](t_2)
-\GenEnergy{\vecTBlend}[\solu](t_1)
=&-\int_{\BulkZone{t_1}{t_2}} \EMS_{\alpha\beta}\nabla^{(\alpha}\vecTBlend^{\beta)} \diFourNatural \\
&\quad-\int_{\BulkZone{t_1}{t_2}} \vecTBlend^\beta(\Im\potlFI)\Im(\solub\partial_\beta\solu) \diFourNatural\\
&\quad+\int_{\BulkZone{t_1}{t_2}} \frac12\vecTBlend^\beta(\Re\nabla_\beta\potlFI)|\solu|^2\diFourNatural.
\end{align*}
Since $\nabla^{(\alpha}\vecX^{\beta)}=-\frac12\Lie_{\vecX}\gMetric^{\alpha\beta}$ and since $\pt$ and $\pp$ are Killing vectors, the Lie derivative
\begin{align*}
\Lie_{\vecTBlend}\gMetric^{\alpha\beta}
&=\frac{a}{\rp^2+a^2} \Lie_{\fnBlend\pp}\gMetric^{\alpha\beta}
\end{align*}
is $a$ small, smooth, and supported in the compact set $\supp\nabla\fnBlend\subset{r\in[10M,11M]}$, which is far from the orbiting null geodesics. Since
\begin{align*}
\left|\EMS_{\alpha\beta}\nabla^{(\alpha}\vecTBlend^{\beta)}\right|
\lesssim \frac{|a|}{M^2} |\pr\fnBlend|\left(|\vecTBlend\solu|^2+|\pr\solu|^2+M^{-2}|\pAng\solu|^2+M^{-2}|\solu|^2\right) ,
\end{align*}
one finds
\begin{align*}
\left| \int_{\BulkZone{t_1}{t_2}}\EMS_{\alpha\beta}\nabla^{(\alpha}\vecTBlend^{\beta)}\diFourNatural \right|
&\lesssim \frac{|a|}{M} \left(\BulkFILowerOrder+\BulkFITwo\right) .
\end{align*}

Since $|\Im\potlFI|\lesssim r^{-4}$, by dividing into the regions where $|r-3M|$ is smaller or larger than $\rfar$, using the Cauchy-Schwarz inequality when it is larger and Fubini's theorem one it is smaller, one finds
\begin{align*}
\left|\int_{\BulkZone{t_1}{t_2}} \vecTBlend^\beta(\Im\potlFI)\Im(\solub\partial_\beta\solu) \diFourNatural \right|
&\leq\left|\int_{\BulkZone{t_1}{t_2}\cap\{|r-3M|\geq\rfar\}} \vecTBlend^\beta(\Im\potlFI)\Im(\solub\partial_\beta\solu) \diFourNatural \right|\\
&\quad+\left|\int_{\BulkZone{t_1}{t_2}\cap\{|r-3M|<\rfar\}} \vecTBlend^\beta(\Im\potlFI)\Im(\solub\partial_\beta\solu) \diFourNatural \right|\\
&\lesssim  \frac{|a|}{M}(\BulkFILowerOrder+\BulkFITwo) + \frac{|a|}{M}\BulkFIOne .
\end{align*}

Since $\potlFI$ is independent of $t$ and $\phi$, 
\begin{align*}
\int_{\BulkZone{t_1}{t_2}} \vecTBlend^\beta(\Re\nabla_\beta\potlFI)|\solu|^2\diFourNatural
=0 .
\end{align*}
\end{proof}

\begin{prop}[Core estimate \eqref{eq:CoreIII}]
\hypothesisC,
\hypothesisFsolu,
then $\forall\FinalTime>0$,
\begin{align*}
\EnergyCrudeFI(\FinalTime)
\lesssim \EnergyCrudeFI(0)+\EnergyCrudeF(0)
+\frac{|a|}{M}\left(\BulkMorawetzF+\BulkFILowerOrder+\BulkFIOne+\BulkFITwo\right) .
\end{align*}
\end{prop}
\begin{proof}
From the Fackerell-Ipser energy equivalence lemma \ref{thm:EnergyGenerationFI}, the crude energy $\EnergyCrudeFI(\FinalTime)$ is estimated by $\GenEnergy{\vecTBlend}[\solu](\FinalTime)+\GenEnergy{\vecTBlend}[\MaxF](\FinalTime)$. From the Maxwell and Fackerell-Ipser energy growth proposition \ref{prop:MaxwellACE} and lemma \ref{lemma:FIEnergy}, these energies are bounded by their initial values, $\GenEnergy{\vecTBlend}[\MaxF](0)+\GenEnergy{\vecTBlend}[\solu](0)$, plus $|a|/M$ times the bulk terms, $\BulkMorawetzF+\BulkFILowerOrder+\BulkFIOne+\BulkFITwo$. Finally, from Maxwell and Fackerell-Ipser energy equivalence lemmas \ref{lemma:MaxwellEnergyEquivalence} and \ref{lemma:FIEnergyEquivalence}, the initial $\vecTBlend$ energies can be estimated by the crude ones, $\EnergyCrudeFI(0)+\EnergyCrudeF(0)$. 
\end{proof}

\begin{corollary}
\label{cor:TrivialEnergyBound}
\hypothesisC,
\hypothesisFsolu, 
then $\forall t_1<t_2$, 
\begin{align*}
\EnergyCrudeF(t_2)+\EnergyCrudeFI(t_2)
\leq C_1 e^{C_2\frac{|a|}{M^2}(t_2-t_1)}\left(\EnergyCrudeF(t_1)+\EnergyCrudeFI(t_1)\right).
\end{align*}
\end{corollary}
\begin{proof}
Core estimates \eqref{eq:CoreI} and \eqref{eq:CoreIII} remain valid with the initial and final times changed from $0$ and $\FinalTime$ to $t_1$ and $t_2$ respectively. $\BulkMorawetzF[t_1,t_2]$ etc.\ will denote the bulk terms with these limits for the integration in $t$. There is the trivial bound
\begin{align*}
\BulkMorawetzF[t_1,t_2]+\BulkFILowerOrder[t_1,t_2]+\BulkFIOne[t_1,t_2]+\BulkFITwo[t_1,t_2]
\lesssim \frac{1}{M} \int_{t_1}^{t_2} \EnergyCrudeF(s)+\EnergyCrudeFI(s) \di s.
\end{align*}
This estimate, the analogues of core estimates \eqref{eq:CoreI} and \eqref{eq:CoreIII}, and Gronwall's inequality combine to prove the main result of this corollary. 
\end{proof}

\section{Fourier-spectral analysis of the Fackerell-Ipser equation}
\label{s:FourierSpectral}
\subsection{Fourier-spectral analysis and the Morawetz estimate: an overview}
\label{ss:FIMorawetzOutline}
The goal of the remainder of this section is to prove core estimates \eqref{eq:CoreIV}-\eqref{eq:CoreV}, which we refer to as Morawetz estimates. The main idea is to construct a vector field that points away from the orbiting null geodesics. 

The location of the orbiting null geodesics are determined by the double roots of the potential appearing in an ODE for the radial components. The potential is $\CurlyR(r;M,a;\geodesice,\geodesiclz,\geodesicQ)$, where $\geodesice$, $\geodesiclz$, and $\geodesicQ$ are constant of the geodesic motion, but $\CurlyR$ remains defined by the function given in equation \eqref{eq:DefR}. With respect to a convenient, non-affine parameterisation, $\lambda$, of the null geodesics, the ODE is $(\di r/\di\lambda)^2=\CurlyR$. Standard ODE analysis, as in a Newtonian potential problem, dictates that there are orbiting null geodesics when $\CurlyR=0$ and $\pr\CurlyR=0$. For any nonnegative, differentiable weight $\fnMorawetzSA$, these conditions are equivalent to
\begin{align*}
\frac{\fnMorawetzSA}{\KDelta}\CurlyR&=0, &
\pr \left(\frac{\fnMorawetzSA}{\KDelta}\CurlyR\right)&=0 .
\end{align*}
With our sign conventions, the instability of the orbiting null geodesics is $\pr^2\CurlyR<0$. This instability condition can also be rewritten with rescaling functions when convenient. To discuss the location of the orbiting null geodesics and their stability, we find it useful to introduce positive weights $\fnMorawetzSA$ and $\fnMorawetzSB$ and to use the quantities we introduced in \cite{AnderssonBlue}
\begin{align*}
\DCurlyRT&=\pr \left(\frac{\fnMorawetzSA}{\KDelta}\CurlyR\right), &
\DDCurlyRTT&= \pr\left(\frac{\fnMorawetzSA^{1/2}}{\KDelta^{1/2}}\fnMorawetzSB\DCurlyRT\right) .
\end{align*}

In this paper, we use a Fourier-spectral multiplier to prove the Morawetz estimate. Given a self-adjoint operator or collection of commuting self-adjoint operators, the spectral theorem defines a spectral transform on $L^2$. The transformed function is a function on the spectral space, which is the spectrum of the operators. A Fourier-spectral variable is a variable that takes values in this spectral space. In this paper, we use the phrase Fourier-spectral multiplier to refer to a function of the Fourier-spectral variables. For $i\pt$, the spectral transform is just the standard Fourier transform. While an analogue of the basic Morawetz, core estimate \eqref{eq:CoreIV}, can be proved using only differential operators \cite{AnderssonBlue}, the refined Morawetz estimate, core estimate \eqref{eq:CoreV} appears to require Fourier-spectral refinements. Thus, the construction in this paper is closer in spirit to \cite{DafermosRodnianski:KerrBoundedness,TataruTohaneanu}, although the details follow those in \cite{AnderssonBlue}. 

In subsection \ref{ss:FourierSpectralTransform}, we justify taking a
spectral transform. We denote the spectral variables by
$\spectralcollection=(\spectralt,\spectralp,\spectralQ)$. The
functions $\CurlyR$, $\DCurlyRT$, and $\DDCurlyRTT$ can be interpreted
as functions of $r$ and the spectral variables, by replacing the
conserved quantities for null geodesics by associated differential
operators acting on functions.

The central idea is to construct a spectral analogue of a vector field and an auxiliary function
\begin{subequations}
\label{eq:DefGenMorawetz}
\begin{align*}
\vecMorawetzBasic&=\fnMorawetzS\pr, &
\fnqS&=\fnMorawetzSA\pr(\fnMorawetzSB\fnMorawetzSC) \\
\fnMorawetzS&=\fnMorawetzSA(r)\fnMorawetzSB(r)\fnMorawetzSC(r;\spectralt,\spectralp,\spectralQ) .
\end{align*}
\end{subequations}
For the basic Morawetz estimate in subsection \ref{ss:FIMorawetzBasic}, we introduce a parameter $\epsilondtsquared\geq0$  and an associated norm on the spectral parameters, $\modspectralcollection$, and choose
\begin{subequations}
\label{eq:DefBasicMorawetz}
\begin{align}
\fnMorawetzSA&=\fnMorawetzSAA\fnMorawetzSAB,&
\fnMorawetzSAA&=\frac{\KDelta}{(r^2+a^2)^2},&
\fnMorawetzSAB&=1-\epsilondtsquared^2\frac{\KDelta}{(r^2+a^2)^2},\\
\fnMorawetzSB&=\fnMorawetzSBA\fnMorawetzSBB,&
\fnMorawetzSBA&=\frac{(r^2+a^2)^4}{3r^2+a^2},&
\fnMorawetzSBB&=\frac{1}{2r} ,
\end{align}
\begin{align}
\fnMorawetzSC&=\chimid\fnMorawetzSCA+(1-\chimid)\fnMorawetzSCB, \\
\fnMorawetzSCA(r;\spectralt,\spectralp,\spectralQ)
&=-\pr\left(\frac{\fnMorawetzSA}{\KDelta}\CurlyR(r;M,a;\spectralt,\spectralp,\spectralQ)\right)\modspectralcollection^{-2} ,\\
\fnMorawetzSCB&=-\pr\potlL ,
\end{align}
\end{subequations}and $\chimid$ to be a smooth function that is identically $1$ for $r\in[2.7M,5M]$, identically $0$ for $<2.4M$ or $r>6M$, monotone on each interval in between, and such that $\forall k\in\Naturals:$ $\pr^k\chimid\lesssim M^{-1}$. 

We have chosen these functions so that the following properties hold:
\begin{itemize}
\item $\fnMorawetzSCA$ is a measure of distance from the orbiting null geodesics but remains bounded as $\modspectralcollection\rightarrow\infty$. In particular, $\fnMorawetzSC$ is a rescaling of $\DCurlyRT$. This gives a perfect square in the $\mathcal{U}$ term defined below.  
\item $\fnMorawetzSC$ is independent of the spectral parameters near $r=\rp$ and $r=\infty$. This helps us control the interaction with the error terms arising from the cut-off $\ChiT$ in lemma \ref{lemma:BasicMorawetzEnergy}. 
\item $\epsilondtsquared^2$ is the coefficient in $\fnMorawetzSC$, $\DCurlyRT$, and $\DDCurlyRTT$ of $\spectralt^2$, associated with $-\pt^2$. 
\item $\fnMorawetzSAA$ is such that, if $\fnMorawetzSAB$ had been $1$, which corresponds to $\epsilondtsquared=0$, then the coefficient of $\spectralt^2$ in $\fnMorawetzSC$ would be zero.
\item $\fnMorawetzSAB$ is such that, if $\epsilondtsquared>0$, then the coefficient of $\epsilondtsquared \spectralt^2$ in $\fnMorawetzSC$ is nonnegative and equal to the coefficient of $Q$.
\item $\fnMorawetzSBA$ is such that, if $\fnMorawetzSAB$ and $\fnMorawetzSBB$ had been $1$, then the coefficient of $\spectralt\spectralp$ in $\DDCurlyRTT$ %(defined in equation \eqref{Broken}) 
would vanish. 
\item $\fnMorawetzSBB$ is such that (i) $\DDCurlyRTT$ is positive everywhere and (ii) $\GenEnergy{\vecMorawetzBasic,\fnqS}[\solu]\lesssim \GenEnergy{\vecTBlend}[\solu]$. Once the form $Cr^{-1}$ was chosen, the factor of $C=1/2$ was chosen so that the coefficient in $\mathcal{A}$, defined below, is $1$.
\end{itemize}
The factors $\fnMorawetzSAA$, $\fnMorawetzSAB$, $\fnMorawetzSBA$, and $\fnMorawetzSC$ are uniquely defined by the above properties. In contrast, the factor $\fnMorawetzSBB$ is both over determined, since we have chosen it to satisfy two conditions that are not \apriori{} obviously compatible, and under determined, since it so happens that there are many functions that allow these two conditions to be satisfied.

For the refined estimate, we take $\fnMorawetzSA$ as above so that $\DCurlyRT$ is unchanged. This has a unique root, which we denote $\rorbit$. We take $\fnMorawetzSB=1$, since we did not find a more useful choice. We take
\begin{align}
\fnMorawetzSC=
%\epsilonRef
M^2\arctan(\rescaledr) \chiTworfar, 
\end{align}
so that it vanishes linearly at the root of $\DCurlyRT$ and each successive derivative introduces an increasing power of $\modspectralcollection$. The factor $\chiTworfar$ is chosen to be smooth, identically one for $M^{-1}(r-\rorbit)<2\rfar$, zero for $M^{-1}(r-\rorbit)>3\rfar$, monotone on the intervals between, and satisfying for all $k\in\Naturals$ $\pr^k\chiThreerfar\lesssim \rfar^{-k}$; it localises to the region near $\rorbit$, where we wish to prove the refined Morawetz estimate.

%%%%%%%%%%%%%%%%%%%%%%%%%%%%%%%%%%%%%%%%%%%%%%%%%%%%%%%%%%%%%%%%%%%%%%%%%%%%%
\subsection{The Fourier-spectral transform}
\label{ss:FourierSpectralTransform}
In this subsection, we introduce a spectral transform of $\solu$, derive energy-generation formulae for it, and prove a lower bound for the spectral parameters for charge-free solutions. 

The spectral transform will define a function $\soluS$ in terms of the spectral parameters $\spectralt$, $\spectralp$, $\spectralQ$ corresponding to the operators $i\pt$, $i\pp$, $\OpQ$. Without a bounded energy estimate, it is not obvious that $\solu$ is even a tempered distribution on $\Exterior$. Thus, we introduce a cut-off in time, to obtain a function that is in $L^2(\di t)$. Let $\FinalTime>0$. This will be the time at which we want to estimate the energy and the upper end point of the interval on which we wish to prove the Morawetz estimate. Let $\ChiT$ be a smooth cut-off function that is identically $1$ on $[0,\FinalTime]$, identically $0$ for $t<-M$ and for $t>\FinalTime+M$, monotone on the intervals in between, and such that for each $k\in\Naturals$, there is a bound on the $k$th derivative by $|\pt^k\ChiT|\lesssim M^{-k}$. 

\newcommand{\spectralQS}{Q}
\newcommand{\spectralsLap}{\ell(\ell+1)}
We let
\begin{align*}
\soluCut&= \solu\ChiT .
\end{align*}
Since $\solu$ is regular and $\ChiT$ is smooth as a function of $t$, and since $\ChiT$ is compactly supported in $t$, it follows that $\soluCut$ is $L^2$ in $t$. Hence it has a Fourier transform in $t$. Let $\spectralt$ by the Fourier variable conjugate to $t$ and $\soluPartialFourier$ be the Fourier transform of $\soluCut$ in the $t$ variable alone, i.e.
\begin{align*}
\soluPartialFourier(\spectralt,r,\theta,\phi)&=\int_{\Reals} \solu(t,r,\theta,\phi) e^{-i\spectralt t}\di t . 
\end{align*}
Since $i\pp$ is self-adjoint, we can perform a standard Fourier transform in $\phi$. We use $\spectralp$ to denote the harmonic parameter associated with $\pp$. Finally, we note that for fixed $\spectralt$ and $\spectralp$, the operator $\OpQ=\sin^{-1}\ph\sin\theta\ph-\cot^2\theta\spectralp^2-a^2\sin^2\theta\spectralt^2$ is symmetric and a bounded perturbation (as an operator on $L^2$) of the standard spherical Laplacian, which is self-adjoint. Since bounded perturbations of self-adjoint operators are themselves self-adjoint \cite{ReedSimon}, the operator $\OpQ$ is self-adjoint and admits a spectral decomposition. Since for fixed $\spectralt$ and $\spectralp$, the operator $\OpQ$ is strictly negative, the spectral decomposition is supported on the nonpositive real line. For convenience, we use $\spectralQ$ to denote the spectral parameter associated with $-\OpQ$, so that $\spectralQ$ is always nonnegative. We use 
\begin{align*}
\soluS(r,\spectralt,\spectralp,\spectralQS)
\end{align*}
to denote the spectral transform in $i\pt$, $i\pp$, and $-\OpQ$. Let $\spectralcollection=(\spectralt,\spectralp,\spectralQ)$, $\dispectral$ denote the spectral measure, and $\spectralspace$ denote the support of the spectral measure. Note that because $i\pp$ has discrete spectrum, $\dispectral$ is a discrete measure with respect to $\spectralp$. The same occurs for $\spectralQ$. The spectral theorem states that the spectral transform is an isomorphism with respect to $L^2$ norms. Thus, with $\diFourSpectral=\dispectral\di r$, 
\begin{align*}
\int_{\mathcal{M}} |\soluCut|^2 \diFourFI
&=\int_{\fullspectralspace} |\soluS|^2 \diFourSpectral .
\end{align*}

The transform satisfies
\begin{align}
\left(\pr\KDelta\pr-\frac{1}{\KDelta}\CurlyR(r;M,a;\spectralt,\spectralp,\spectralQ) -\potlFIz \right)\soluS&=\ErrorSCut +\ErrorSIm,
\label{eq:SpectralFI}
\end{align}
where $\potlFIz=-2M/r$ is an approximation of $\KSigma\potlFI$ at $a=0$, $\ErrorSCut$ is the transform of $\KSigma(\solu(\nabla^\alpha\nabla_\alpha\ChiT)+2(\nabla_\alpha\solu)(\nabla^\alpha\ChiT))$ and $\ErrorSIm$ is the transform of $(\potlFIz-\KSigma\potlFI)\soluCut$. The sign in front of $\CurlyR$ is opposite that appearing in the Fackerell-Ipser equation, since $\spectralt$, $\spectralp$, and $\spectralQS$ correspond to $i\pt$, $i\pp$, and $-\OpQ$, instead of $\pt$, $\pp$, and $\OpQ$. 

\begin{theorem}[Energy generation for the spectral transform]
\label{thm:EnergyGenerationSpectral}
Let $M>0$, $a\in(-M,M)$. Let $(\fnMorawetzS,\fnqS)$ be a pair of of smooth, real-valued functions on $\fullspectralspace$ such that they and all their partial derivatives have bounded limits on $\{\rp\}\times\spectralspace$ and $\{\infty\}\times\spectralspace$.

If $\solu$ is a solution of the Fackerell-Ipser equation \eqref{eq:FI} with spectral transform $\soluS$, then
\begin{align*}
\EnergyS &=\BulkSMain +\BulkSIm ,
\end{align*}
where
%\begin{subequations}
\begin{align*}
\EnergyS&=\Re\int_{\fullspectralspace} \Re((\fnMorawetzS(\pr\soluSb)+\fnqS\soluSb)\ErrorSCut) \diFourSpectral, \\
\BulkSMain&=\int_{\fullspectralspace} \left(
\mathcal{A}|\pr\soluS|^2
+\soluSb\mathcal{U}\soluS
+\mathcal{V}|\soluS|^2
\right)\diFourSpectral ,\\
\BulkSIm
&=-\int_{\fullspectralspace} \Re(\fnMorawetzS(\pr\soluSb)\ErrorSIm) \diFourSpectral, 
\end{align*}
%\end{subequations}
and
\begin{subequations}
\label{eq:mathcaldefs}
\begin{align}
\mathcal{A}&=\left(-\frac12(\pr\KDelta)\fnMorawetzS+\frac12\KDelta(\pr\fnMorawetzS)+\fnqS\KDelta\right)
\\
%\frac12\left((-\pr\KDelta)\fnMorawetzS+2\KDelta(\pr\fnMorawetzS)+2\fnqSRed\KDelta\right) ,\\
\mathcal{U}&=\left(-\frac12(\pr(\KDelta^{-1}\CurlyR))\fnMorawetzS-\frac12\KDelta^{-1}\CurlyR(\pr\fnMorawetzS)+\KDelta^{-1}\CurlyR\fnqS\right)\\
%-\frac12\left(\pr\frac{\CurlyR}{\KDelta}\right)\fnMorawetzS +\fnqSRed\frac{\CurlyR}{\KDelta},\\
\mathcal{V}&=\left(-\frac12(\pr\potlFIz)\fnMorawetzS-\frac12\potlFIz(\pr\fnMorawetzS)+\fnqS\potlFIz\right)
%\\
%-(\pr\potlFIz)\fnMorawetzS +2\fnqSRed\potlFIz-\frac12(\pr\KDelta\pr\fnqS) ,
%\\
%\CurlyR&=\CurlyR(r;M,a;\spectralt,\spectralp,\spectralQ) .
\end{align}
\end{subequations}
\end{theorem}
\begin{proof}
For a given $\soluS$, let
\begin{align*}
\EMSSpec&=\frac12(\pr\soluSb)(\pr\soluS) -\frac12\soluSb\KDelta^{-2}\CurlyR\soluS -\frac12\KDelta^{-1}\potlFIz|\soluS|^2 ,\\
\GenMomentumS&=\EMSSpec\fnMorawetzS +\fnqS\Re(\soluSb\pr\soluS)-\frac12(\pr\fnqS)|\soluS|^2 .
\end{align*}
By direct computation
\begin{align*}
\pr(\KDelta\EMSSpec)
&=\Re((\pr\soluSb)(\pr\KDelta\pr\soluS)-\frac12(\pr\KDelta)|\pr\soluS|^2)\\
&\quad+\Re(-(\pr\soluSb)\KDelta^{-1}\CurlyR\soluS-\frac12\soluSb(\pr(\KDelta^{-1}\CurlyR)\soluSb)\\
&\quad+\Re(-(\pr\soluSb)\potlFIz\soluS-\frac12(\pr\potlFIz)|\soluS|^2) \\
&=\Re((\pr\soluSb)(\ErrorSCut+\ErrorSIm))\\
&\quad+\frac12\Re\left(-(\pr\KDelta)|\pr\soluS|^2-\soluSb(\pr(\KDelta^{-1}\CurlyR))\soluS-(\pr\potlFIz)|\soluS|^2\right) .
\end{align*}
Thus,
\begin{align*}
\pr(\KDelta\GenMomentumS)
&=(\pr(\KDelta\EMSSpec))\fnMorawetzS +\KDelta\EMSSpec(\pr\fnMorawetzS)\\
&\quad+\fnqS\KDelta|\pr\soluS|^2+\fnqS\Re(\soluS(\pr\KDelta\pr\soluS))\\
&\quad-\frac12(\pr\KDelta\pr\fnqS)|\soluS|^2 .
\end{align*}
By substituting for $\EMSSpec$, $\pr(\KDelta\EMSSpec)$, and $\pr\KDelta\pr\soluS$, one finds
\begin{align*}
\pr(\KDelta\GenMomentumS)
&=\left(-\frac12(\pr\KDelta)\fnMorawetzS+\frac12\KDelta(\pr\fnMorawetzS)+\fnqS\KDelta\right)|\pr\soluS|^2\\
&\quad+\soluSb\left(-\frac12(\pr(\KDelta^{-1}\CurlyR))\fnMorawetzS-\frac12\KDelta^{-1}\CurlyR(\pr\fnMorawetzS)+\KDelta^{-1}\CurlyR\fnqS\right)\soluS \\
&\quad+\left(-\frac12(\pr\potlFIz)\fnMorawetzS-\frac12\potlFIz(\pr\fnMorawetzS)+\fnqS\potlFIz\right)|\soluS|^2\\
&\quad-\frac12(\pr\KDelta\pr\fnqS)|\soluS|^2\\
&\quad+\Re\left((\fnMorawetzS\pr\soluSb+\fnqS\soluSb)(\ErrorSCut+\ErrorSIm)\right).
\end{align*}
From the definition of $\fnqS$, $\mathcal{A}$, $\mathcal{U}$, $\mathcal{V}$, one finds
\begin{align*}
\pr(\KDelta\GenMomentumS)
&=\mathcal{A}|\pr\soluS|^2+\soluSb\mathcal{U}\soluS+\mathcal{V}|\soluS|^2 \\
&\quad+\Re\left((\fnMorawetzS\pr\soluSb+\fnqS\soluSb)(\ErrorSCut+\ErrorSIm)\right) .
\end{align*}
Integrating this over $\fullspectralspace$, one obtains the desired result, since $\KDelta\GenMomentumS\rightarrow0$ as $r\rightarrow\rp$ and $r\rightarrow\infty$. 
\end{proof}

\begin{lemma}[Simplified energy generation coefficients for factored $\fnMorawetzS$ and $\fnqS$]
\label{lemma:SimplifiedAUV}
If 
\begin{align*}
\fnMorawetzS&=\fnMorawetzSA\fnMorawetzSB\fnMorawetzSC, &
\fnqS&=\frac12\fnMorawetzSA\pr(\fnMorawetzSB\fnMorawetzSC) ,
\end{align*}
then the quantities $\mathcal{A}$, $\mathcal{U}$, and $\mathcal{V}$ from the previous theorem reduce to
\begin{align*}
\mathcal{A}&=\frac12 \KDelta^{3/2}\fnMorawetzSA^{1/2} \pr\left(\frac{\fnMorawetzSA^{1/2}}{\KDelta^{1/2}}\fnMorawetzSB\fnMorawetzSC\right) ,\\
\mathcal{U}&=-\frac12\left(\pr \frac{\fnMorawetzSA\CurlyR}{\KDelta}\right)\fnMorawetzSB\fnMorawetzSC ,\\
\mathcal{V}&=-(\pr\fnMorawetzSA\potlFIz)\fnMorawetzSB\fnMorawetzSC 
-\frac12(\pr\KDelta\pr\fnqS) .
\end{align*}
\end{lemma}
\begin{proof}
Direct computation. 
\end{proof}

The lower bound by the angular derivatives for charge-free solutions, lemma \ref{lemma:LowerBoundForAngularDerivatives}, will be applied in lemma \ref{lemma:BulkMainMorawetzFI} to the spectral transform $\soluS$, instead of the original solution $\solu$. The following lemma provides the spectral version of this lower bound. 

\begin{lemma}[Spectral lower bound for $\spectralQ$]
\label{lemma:LowerBoundForSpectralQ}
\hypothesisA, $\FinalTime>0$, 
%There exists a positive constant $\epsilonSlowRotationIntro>0$ such that for all $M>0$, $a\in[-\epsilonSlowRotationIntro M,\epsilonSlowRotationIntro M]$, $\FinalTime>0$, 
if $\MaxF$ is a solution of the Maxwell equation \eqref{eq:Maxwell}, $\solu=\solu[\MaxF]$, and $\soluS$ is the $\FinalTime$-spectral transform of $\solu$, then for all $r\in(\rp,\infty)$
\begin{align*}
\left(2-C\frac{|a|}{M}\right)\int_{\spectralspace} |\soluS|^2 \dispectral
&\leq \int_{\spectralspace} (\spectralQ+\spectralp^2)|\soluS|^2 \dispectral \\
&\quad+C a^2\int_\Reals\int_{\Stwo} \frac{\KDelta}{r^2+a^2}\ChiT^2|\phipm|^2 \diomega \di t
\end{align*}
\end{lemma}
\begin{proof}
To begin, multiply the result of lemma \ref{lemma:LowerBoundForAngularDerivatives} by $(r\ChiT)^2$, observe that this factor commutes with angular derivatives, observe that the difference between $\int |(r-ia\cos\theta)\phiz|^2 \diomega$ and $\int |r\phiz|^2\diomega$ is bounded by $|a|M^{-1}$ times either of these, and that the difference between $\int |\pAng ((r-ia\cos\theta)\phiz)|^2\diomega$ and $\int |r\pAng \phiz|^2\diomega$ is bounded by $|a|M^{-1} \int (|r\phiz|^2+|r\pAng\phiz|)\diomega$. From these observations, one concludes
\begin{align*}
\left(2-\frac{|a|}{M}\right) \int |\soluCut|^2\diomega 
&\leq \int |\pAng\soluCut|^2\diomega \\
&\qquad+Ca^2\frac{\KDelta}{r^2+a^2}\ChiT^2\int |\phipm|^2 \diomega .
\end{align*}
The inequality remains valid if, to the right hand side, one adds the positive term $\int a^2\sin^2\theta|\pt\soluCut|^2\diomega$. Since $\soluCut$ is compactly supported in the time and angular variables, one can integrate in $t$ and then apply integration by parts in the time and angular variables to obtain
\begin{align*}
\left(2-\frac{|a|}{M}\right) \int |\soluCut|^2\diomega \di t
&\leq \int \soluCutb (-\OpQ-\pp^2)\soluCut\diomega\di t \\
&\qquad+Ca^2\frac{\KDelta}{r^2+a^2}\int |\phipm|^2 \ChiT^2\diomega\di t .
\end{align*}
Applying the spectral transform to the term on the left and the first term on the right gives the desired result. 
\end{proof}

%%%%%%%%%%%%%%%%%%%%%%%%%%%%%%%%%%%%%%%%%%%%%%%%%%%%%%%%%%%%%%%%%%%%%%%%%%%%%
\subsection{The basic Morawetz estimate for the Fackerel-Ipser equation}
\label{ss:FIMorawetzBasic}
In this section, take $\vecMorawetzBasic$, $\fnMorawetzS$, $\fnMorawetzSA$, $\fnMorawetzSAA$, $\fnMorawetzSAB$, $\fnMorawetzSAB$, $\fnMorawetzSB$, $\fnMorawetzSBA$, $\fnMorawetzSBB$, $\fnMorawetzSC$, $\chimid$, $\fnMorawetzSA$, $\fnMorawetzSCB$, and $\fnqS$ as defined in subsection \ref{ss:FIMorawetzOutline}. For $\epsilondtsquared\geq0$, we define
\begin{align*}
\modspectralcollection^2 
&=\epsilondtsquared^2\spectralt^2 +\spectralp^2+\spectralQ .
\end{align*}

\begin{lemma}[Properties of $\DCurlyRT$ and $\DDCurlyRTT$]
\label{lemma:PropertiesOfDCurlyRT}
\hypothesisD, 
for each set of the spectral parameters $\spectralcollection=(\spectralt,\spectralp,\spectralQ)$, the function $\DCurlyRT$ has a unique root, $\rorbit$, in the exterior, and
\begin{align*}
|\rorbit-3M|&\lesssim \smallparameter M, \\
\left|\DCurlyRT\right|
&\geq 2\left(1-C\smallparameter\right) \frac{|r-\rorbit|}{r^4} ,\\
-\DDCurlyRTT
&\geq \left(1-C\smallparameter\right) \frac{M}{r^2}\modspectralcollection^2 .
\end{align*}
\end{lemma}
\begin{proof}
Recall $\DCurlyRT=\pr(\fnMorawetzSA\KDelta^{-1}\CurlyR)$ from subsection \ref{ss:FIMorawetzOutline}. The coefficient of $\spectralt^2$ in $\fnMorawetzSA\KDelta^{-1}\CurlyR$ is $-1+\epsilondtsquared^2\potlL$. Because the derivative of a constant vanishes, 
\begin{align*}
\DCurlyRT
&=(\pr\potlL)(\epsilondtsquared^2\spectralt^2+\spectralp^2+\spectralQ)\\
&\quad+\pr\left(-\frac{4aMr}{(r^2+a^2)^2}\spectralt\spectralp-\frac{a^2}{(r^2+a^2)^2}(1-\epsilondtsquared^2\potlL)\spectralp^2+\epsilondtsquared^2\potlL^2(\spectralp^2+\spectralQ)\right) .
\end{align*}
It is convenient to rewrite this as a quadratic expression in $\epsilondtsquared\spectralt$, $\spectralp$, and $\spectralQ^{1/2}$, namely
\begin{align*}
\DCurlyRT
&=(\pr\potlL)\modspectralcollection^2\\
&\quad+\left(-\frac{a}{\epsilondtsquared}\frac{Mr}{(r^2+a^2)^2}(\epsilondtsquared\spectralt)\spectralp-\frac{a^2}{(r^2+a^2)^2}(1-\epsilondtsquared^2\potlL)\spectralp^2+\epsilondtsquared^2\potlL^2(\spectralp^2+\spectralQ)\right) .
\end{align*} 
At $a=0$, one has $\pr\potlL=-2r^{-4}(r-3M)$. Since $\potlL$ is of maximal degree and not of order $0$, the derivative is of maximal degree. Hence
\begin{align}
\left|\DCurlyRT-\frac{-2(r-3M)}{r^4}\modspectralcollection^2\right|
\lesssim \smallparameter\frac{M}{r^4}\modspectralcollection^2 .
\label{eq:DCurlyRTEstimateFirst}
\end{align}

Now consider 
\begin{align*}
\DDCurlyRTT&=\pr\left(\frac{\fnMorawetzSA^{1/2}}{\KDelta^{1/2}}\fnMorawetzSB\DCurlyRT\right) .
\end{align*}
At $a=0$, the term to compare this with is
\begin{align*}
\pr\left(\frac{-1}{3}\left(1-\frac{3M}{r}\right)\right)
&=-\frac{M}{r^2} .
\end{align*}
Here one is differentiating a homogeneous rational function of maximal degree, but of degree $0$, so that the resulting function is no longer of maximal. Nonetheless, one can still estimate the difference between the true value and the approximation, by
\begin{align*}
\left|\DDCurlyRTT -\left(-\frac{M}{r^2}\right)\modspectralcollection^2\right|
&\lesssim \smallparameter \frac{M}{r^2}\modspectralcollection^2 .
\end{align*}
Thus,
\begin{align*}
\DDCurlyRTT
\geq& -\left(1-C\smallparameter\right) \frac{M}{r^2}\modspectralcollection^2 .
\end{align*}

Since $\DCurlyRT$ at $a=0$ is $-2r^{-4}(r-3M)(\epsilondtsquared^2\spectralt^2+(1-2\epsilondtsquared\potlL)(\spectralp^2+\spectralQ))$ has a simple root at $r=3M$, by continuity, the function $\DCurlyRT$ continues to have a root near $r=3M$ when $a$ and $\epsilondtsquared$ are small. For sufficiently small $a$ and $\epsilondtsquared$, because $\DDCurlyRTT$ is strictly positive, the root of $\DCurlyRT$ remains simple and unique. Let $\rorbit$ denote this root. From estimate \eqref{eq:DCurlyRTEstimateFirst}, one finds
\begin{align*}
|\rorbit-3M|&\lesssim \smallparameter M.
\end{align*}
Since $\DCurlyRT$ vanishes at $r=\rorbit$, its derivative there is only a small deviation from the value at $a=0$, and from estimate \eqref{eq:DCurlyRTEstimateFirst}, one finds
\begin{align*}
\left|\DCurlyRT\right|
&\geq 2\left(1-C\smallparameter\right) \frac{|r-\rorbit|}{r^4} .
\end{align*}
\end{proof}

\begin{lemma}[Bound for the main bulk term in the basic Morawetz estimate]
\label{lemma:BulkMainMorawetzFI}
Let $\fnMorawetzS$ and $\fnqS$ be defined as in equations \eqref{eq:DefGenMorawetz}-\eqref{eq:DefBasicMorawetz}. 

\hypothesisD,
\hypothesisFsoluSNoCharge, 
then
\begin{align*}
\BulkSMain 
\geq& C_1\int_{\fullspectralspace} 
\frac{M\KDelta^2}{r^2(r^2+a^2)}|\pr\soluS|^2 
+\frac{(r-\rorbit)^2}{r^3}\modspectralcollection^2|\soluS|^2
+\frac{M}{r^2}|\soluS|^2 \diFourSpectral \\
&\quad- C_2 a^2 \int_{\Exterior} \frac{M\KDelta}{(r^2+a^2)^2r^2} |\phipm|^2\ChiT^2 \diFourFI .
\end{align*}
\end{lemma}
\begin{proof}
At first, arbitrary, small values of $\rfar$ will be permitted. For each such $\rfar$, there will be $\epsilonSlowRotationIntro$ and $\epsilondtsquaredMax$ for which estimates will hold uniformly in $a$ and $\epsilondtsquared$. Towards the end of the proof, a particular value of $\rfar$ will be chosen. Since the estimates prior to that point in the proof will have been proven uniformly in $a$ and $\epsilondtsquared$, the values of $\epsilonSlowRotationIntro$ and $\epsilondtsquaredMax$ can be shrunk further, if necessary. 

\StartStep
\Step{The $\mathcal{A}$, $\mathcal{U}$, and $\mathcal{V}$ terms}
First, observe that for $a=0=\epsilondtsquared$,
\begin{align*}
\fnMorawetzSCA
&=-\frac{\DCurlyRT}{\modspectralcollection^2}
=-\pr\potlL 
=\fnMorawetzSCB
=\fnMorawetzSC ,
\end{align*}
and that, uniformly in $a/\epsilondtsquared$ and $\epsilondtsquared/M$, in the support of $\pr\chimid$, 
\begin{align*}
\left|\fnMorawetzSCA-\fnMorawetzSCB\right|
&\lesssim \smallparameter \frac{1}{r^3} .
\end{align*}
Thus, $\fnMorawetzSC$ satisfies an estimate like $\DCurlyRT$, namely, 
\begin{align*}
|\fnMorawetzSC|
\geq \left(2-C\smallparameter\right)\frac{|r-\rorbit|}{r^4} .
\end{align*}
Similarly, the dominant parts of $\DDCurlyRTT\modspectralcollection^{-2}$ and $\pr(\fnMorawetzSA^{1/2}\KDelta^{-1/2}\fnMorawetzSB\fnMorawetzSC)$ coincide, so
\begin{align*}
-\pr\left(\frac{\fnMorawetzSA^{1/2}}{\KDelta^{1/2}}\fnMorawetzSB\fnMorawetzSC\right)
&\geq \left(1-C\smallparameter\right) \frac{M}{r^2} .
\end{align*}

From the formula for $\mathcal{A}$ in lemma \ref{lemma:SimplifiedAUV}, one finds
\begin{align}
\mathcal{A}
%&=\frac12\left(
%(-\pr\KDelta)\fnMorawetzSA\fnMorawetzSB\fnMorawetzSC 
%+2\KDelta\pr(\fnMorawetzSA\fnMorawetzSB\fnMorawetzSC )
%-(\pr\fnMorawetzSA)\fnMorawetzSB\fnMorawetzSC \right)\\
&=\frac12 \KDelta^{3/2}\fnMorawetzSA^{1/2} \pr\left(\frac{\fnMorawetzSA^{1/2}}{\KDelta^{1/2}}\fnMorawetzSB\fnMorawetzSC\right) \nonumber\\
&\geq  \frac12\frac{\KDelta^2}{r^2+a^2} \frac{M}{r^2} \left(1-C\smallparameter\right) .
\label{eq:HardyATerm}
\end{align}

Similarly, 
\begin{align*}
\mathcal{U}
&=-\frac12 \DCurlyRT \fnMorawetzSB \fnMorawetzSC \\
&\geq \left(1-C\smallparameter\right) \frac{1}{3} \frac{(r-\rorbit)^2}{r^3} \modspectralcollection^2 .
%&=-\frac12\left(\pr\frac{\CurlyR}{\KDelta}\right)\fnMorawetzSA\fnMorawetzSB\fnMorawetzSC -\frac{\CurlyR}{\KDelta}(\pr\fnMorawetzSA)\fnMorawetzSB\fnMorawetzSC \\
%&=-\frac12\left(\pr \frac{\fnMorawetzSA\CurlyR}{\KDelta}\right)\fnMorawetzSB\fnMorawetzSC\\
%&=\frac12\fnMorawetzSB(\DCurlyRT)^2\modspectralcollection^{-2} ,
\end{align*}
A lower bound for the expectation value of this can be obtained by using the fact that $2M/r$ is bounded above by $1$-minus a small constant and using lemma \ref{lemma:LowerBoundForSpectralQ}. The estimate is 
\begin{align*}
&\int_{\fullspectralspace} \soluSb\mathcal{U}\soluS\diFourSpectral\\
&\geq \left(1-C\smallparameter\right) M\int_{\fullspectralspace} \frac23 \frac{(r-\rorbit)^2}{r^4} \modspectralcollection^2 |\soluS|^2 \diFourSpectral\\
&\geq \left(1-C\smallparameter\right) M\int_{\fullspectralspace} \frac43 \frac{(r-\rorbit)^2}{r^4}  |\soluS|^2 \diFourSpectral \\
&\quad -Ca^2 \int_{\Exterior} \frac{M\KDelta}{(r^2+a^2)^2r^2}|\phipm|^2 \ChiT^2\diFourFI  .
\end{align*}
Since $|\rorbit-3M|\lesssim\smallparameter M$, it follows that $(r-\rorbit)^2 >(r-3M)^2-C\smallparameter^2M^2$,. Hence, 
\begin{align}
\label{eq:HardyUTerm}
&\left(1-C\smallparameter\right) M\int_{\fullspectralspace} \frac43 \frac{(r-\rorbit)^2}{r^4}  |\soluS|^2 \diFourSpectral \\
&\geq \left(1-C\smallparameter\right) \int_{\fullspectralspace} \frac{8Mr^2-48M^2r+72M^3}{6r^4}  |\soluS|^2 \diFourSpectral\nonumber\\
&\quad-C\smallparameter^2 \int_{\fullspectralspace} \frac{M}{r^2}|\soluS|^2 \diFourSpectral  \nonumber.
\end{align}

From the formula for $\mathcal{V}$ in lemma \ref{lemma:SimplifiedAUV}, one finds
\begin{align*}
\mathcal{V}
&=-(\pr\fnMorawetzSA\potlFIz)\fnMorawetzSB\fnMorawetzSC 
-\frac12(\pr\KDelta\pr(\fnMorawetzSA\pr(\fnMorawetzSB\fnMorawetzSC))) .
\end{align*}
Each of these two terms can be approximated by its value at $a=0$. First, the $-\frac12(\pr\KDelta\pr\fnqS)$ term was already computed in \cite{AnderssonBlue} when $a=0=\epsilondtsquared$, so we know
\begin{align*}
\left|-\frac12(\pr\KDelta\pr\fnqS)-\frac{9Mr^2-46M^2r+54M^3}{6r^4}\right|
\lesssim \smallparameter \frac{M}{r^2} .
\end{align*}
Note that there are no spectral parameters on the right, since $\fnqS$ is, when viewed as a function of the spectral parameters, a ratio with a quadratic function of $\epsilondtsquared\spectralt$, $\spectralp$, and $\spectralQ^{1/2}$ in the numerator and $\modspectralcollection^2$ in the denominator. Thus, in terms of the spectral parameters, there is a uniform bound by a constant depending on $r$, $M$, $\epsilondtsquared$, and $a$. 

The treatment of $(\pr \fnMorawetzSA\potlFIz)\fnMorawetzSB\fnMorawetzSC$ is similar. When $a=0$, one finds this expression is
\begin{align*}
\left(-\pr \frac{1}{r^2}\left(1-\frac{2M}{r}\right)\frac{2M}{r}\right)\frac{r^3}{6}\frac{(-2)(r-3M)}{r^4}
&=\frac{2M(3-8Mr^{-1})(1-3Mr^{-1})}{6r^2}\\
&=-\frac{6Mr^2-34M^2r+48M^3}{6r^4} .
\end{align*}
Thus, combining these estimates, one finds
\begin{align}
\label{eq:HardyVTerm}
\left|\mathcal{V}-\frac{3Mr^2-12M^2r+6M^3}{6r^4}\right|
&\lesssim \smallparameter \frac{M}{r^2} .
\end{align}

\Step{Review of Hardy estimate and ODE techniques}
In \cite{AnderssonBlue}, we extended the method from \cite{BlueSoffer:Errata} to prove nonnegativity of expressions of the form
\begin{align}
\int_0^\infty \HardyprCoeff|\partialHardy\soluS|^2 + \HardyPotl|\soluS|^2 \di \HardyvarRed 
\label{eq:HardyIntegral}
\end{align}
by relating this to the existence of positive solutions of an associated ODE. We refer to nonnegativity of the integral \eqref{eq:HardyIntegral} as a Hardy estimate, since we allow $\HardyPotl$ to be negative in some regions. This is related to the problems in this paper and \cite{AnderssonBlue} by taking 
\begin{align*}
\HardyvarRed&=r-\rp ,&
\Hardyd&=2\sqrt{M^2-a^2} ,\\
\HardyprCoeff&= \frac{\KDelta^2}{(r^2+a^2)r^2} ,
\end{align*}
and $\HardyPotl$ is $r^{-4}$ times a quadratic expression in $r$, $M$, and $a$ with coefficients specified below. 

The substitution $\HardysoluRed=\HardyprCoeff^{1/2} \soluS$ transforms the integral \eqref{eq:HardyIntegral} to
\begin{align*}
\int_0^\infty |\partialHardy\HardysoluRed|^2 +\HardyPotlRed|\HardysoluRed|^2 \di\HardyvarRed 
\end{align*}
where $\HardyPotlRed$ and the coefficients $\HardyX$, $\HardyY$, and $\HardyZ$ are defined by
\begin{align*}
\HardyPotlRed
&=\frac{\HardyX \HardyvarRed^2 +\HardyY \HardyvarRed +\HardyZ}{6\HardyvarRed^2 (\HardyvarRed+\Hardyd)^2} 
=\frac{\HardyPotl}{\HardyprCoeff} +\frac12\frac{\partialHardy^2\HardyprCoeff}{\HardyprCoeff} -\frac14\frac{(\partialHardy\HardyprCoeff)^2}{\HardyprCoeff^2} . 
\end{align*}
To prove that the integral  $\int |\partialHardy\HardysoluRed|^2+\HardyPotlRed|\HardysoluRed|^2\di\HardyvarRed$ is nonnegative, it is sufficient to show that the following ODE has a nonnegative solution
\begin{align*}
0&=-\partialHardy^2\HardyODEsolu+\HardyPotlRed\HardyODEsolu .
\end{align*}
As explained in \cite{AnderssonBlue}, a solution of this equation is given by 
\begin{align*}
\HardyODEsolu&= \HardyvarRed^\Hardyalpha(\HardyvarRed+\Hardyd)^{\Hardybeta}F(\Hardya,\Hardyb,\Hardyc;-(r-\Hardyd)/r) ,
\end{align*}
where $F={}_2F_1$ is the Gauss hypergeometric function, and $\Hardyalpha$, $\Hardybeta$, $\Hardya$, $\Hardyb$, and $\Hardyc$ are parameters\footnote{It is unfortunate that $a$ is is almost universally used to denote
both the first parameter of the hypergeometric and the rotation
parameter for a Kerr black hole. We have attempted to reduce the
confusion by using different fonts. It should
be clear from context, which is which.} satisfying
\begin{subequations}
\label{eq:HypergeometricParameterConditions}
\begin{align}
\Hardyalpha(\Hardyalpha-1)\Hardyd^2-\HardyZ/6&=0 ,\\
\Hardyd\left((\HardyX/6)-\Hardyalpha(\Hardyalpha-1)-\Hardybeta(\Hardybeta-1)\right)&=-2\Hardyd\Hardyalpha(\Hardyalpha-1)-\HardyY/6,\\
\Hardyc&=2\Hardyalpha,\\
-\Hardya-\Hardyb-1&=-2(\Hardyalpha+\Hardybeta),\\
-\Hardya\Hardyb&=-\Hardyalpha(\Hardyalpha-1)-2\Hardyalpha\Hardybeta-\Hardybeta(\Hardybeta-1)+\HardyX/6 . 
\end{align}
\end{subequations}
To show that $\HardyODEsolu$ is nonnegative, it is sufficient to
choose the parameters so that $\Hardyalpha$ is non-integer and
\begin{align}
\label{eq:HardyConditions}
\Hardya&<0<\Hardyb<\Hardyc .
\end{align}

\Step{Apply Hardy estimate to the main bulk term using the spectral lower bound on charge-free solutions}
In the analysis of the wave equation in \cite{AnderssonBlue}, we were able to show there are $\mathcal{A}$, $\mathcal{U}$, $\mathcal{V}$ terms, corresponding to coefficients of $(\pr\soluS)^2$, of $|\pt\solu|^2+|\pAng\solu|^2$, and of $|\solu|^2$. For both the wave equation and the Fackerell-Ipser equation, there are regions where the potential $\mathcal{V}$ is negative. For the wave equation, we showed that $\int \mathcal{A}|\pr\solu|^2 +\mathcal{V}|\solu|^2\di r\geq 0$ using the Hardy estimate reviewed in the previous step. 

Unfortunately, for the Fackerel-Ipser equation, $\mathcal{V}$ is so negative that the integral \eqref{eq:HardyIntegral} can fail to be positive with $\HardyprCoeff=M^{-1}\mathcal{A}$ and $\HardyPotl=M^{-1}\mathcal{V}$. This can be seen by substituting the transform of the charged solution for $\soluS$. In this step, we achieve positivity by taking advantage of the $\mathcal{U}$ term and the spectral lower bound on $\modspectralcollection^2$ for charge-free solutions. 

By summing the contributions from $\mathcal{A}$, $\mathcal{U}$, and $\mathcal{V}$ in equations \eqref{eq:HardyATerm}, \eqref{eq:HardyUTerm}, and \eqref{eq:HardyVTerm}, one finds
\begin{align*}
\int_{\fullspectralspace}& \mathcal{A}|\pr\soluS|^2+\soluSb\mathcal{U}\soluS+\mathcal{V}|\soluS|^2 \diFourSpectral\\
\geq& M\int_{\fullspectralspace} \frac{\KDelta^2}{(r^2+a^2)r^2}|\pr\soluS|^2 +\frac{11 r^2 -60M r+78M^2}{6r^4}|\soluS|^2 \diFourSpectral \\
&-CM\smallparameter \int_{\fullspectralspace} \frac{1}{r^2} |\soluS|^2 \diFourSpectral \\
&-Ca^2 \int_{\Exterior} \frac{M\KDelta}{(r^2+a^2)^2r^2}|\phipm|^2 \ChiT^2\diFourFI  .
\end{align*}
Taking 
\begin{align*}
A&=\KDelta^2(r^2+a^2)^{-1}r^{-2},&
\HardyPotl&=(11r^2-60Mr+78M^2)/(6r^4),
\end{align*} 
one can apply the analysis from the previous step. For $a=0$, one finds the following values for the coefficients in the transformed potential $\HardyPotlRed$, 
\begin{align*}
\HardyX&=11, &
\HardyY&=-60M+2\HardyX\rp=-16M, &
\HardyZ&= 78M^2+\HardyY(\rp)+\HardyX\rp^2=2M^2 .
\end{align*}
Taking convenient choices of roots in the equation $\Hardyalpha$ and $\Hardybeta$, one finds the hypergeometric parameters 
\begin{align*}
\Hardyalpha&= \frac12 +\frac13\sqrt{3},\\
\Hardybeta&= \frac12-\frac12\sqrt{22},\\
\Hardya&= -\frac12\sqrt{22}+\frac12-\frac{1}{2}\sqrt{3} \simeq -2.7 ,\\
\Hardyb&= -\frac12\sqrt{22}+\frac12+\frac{7}{6}\sqrt{3} \simeq .18 ,\\
\Hardyc&= 1+\frac23\sqrt{3} \simeq 2.2. 
\end{align*}
These satisfy conditions \eqref{eq:HardyConditions}. The conditions \eqref{eq:HardyConditions} are open conditions, and the parameters $\Hardyalpha$, $\Hardybeta$, $\Hardya$, $\Hardyb$, $\Hardyc$ depend continuously on the rotation parameter $a$ and the coefficients $\HardyX$, $\HardyY$, $\HardyZ$, which depend continuously on the original coefficients in $\HardyPotl$. From this freedom to slightly adjust the coefficient, one can obtain strict positivity instead of mere nonnegativity. Similarly, one can perturb to absorb the contributions involving $\smallparameterwithrfar$. Thus, there is are constants such that
\begin{align*}
\int_{\fullspectralspace}& \mathcal{A}|\pr\soluS|^2+\soluSb\mathcal{U}\soluS+\mathcal{V}|\soluS|^2 \diFourSpectral\\
\geq& C_1\int_{\fullspectralspace} 
\frac{M\KDelta^2}{r^2(r^2+a^2)}|\pr\soluS|^2 
+\frac{(r-\rorbit)^2}{r^3}\modspectralcollection^2|\soluS|^2
+\frac{M}{r^2}|\soluS|^2 \diFourSpectral \\
&\quad- C_2a^2 \int_{\Exterior} \frac{M\KDelta}{(r^2+a^2)^2r^2} |\phipm|^2\ChiT^2 \diFourFI .
\end{align*} 
\end{proof}

\begin{lemma}[Bound on the contribution from the remainder of the Fackerell-Ipser potential in the basic Morawetz estimate]
\label{lemma:BasicMorawetzEnergy}
\hypothesisD,
\hypothesisFsoluS,
\begin{align*}
|\BulkSIm|
&\lesssim \frac{|a|}{M}  \int_{\fullspectralspace} \left(\frac{M\KDelta^2}{(r^2+a^2)r^2}|\pr\soluS|^2 +\frac{M}{r^2}|\soluS|^2\right) \diFourSpectral .
%(\BulkFILowerOrder +\BulkFITwo) .
\end{align*}
\end{lemma}
\begin{proof}
The functions $\potlFI$ and $\potlFIz$ can be viewed as homogeneous polynomials in $r$, $M$, $a$, $a\cos\theta$. Because of the inclusion of $a\cos\theta$ terms, the previous analysis of rational functions is not entirely valid, although the same ideas apply. Since $\potlFI$ and $\potlFIz$ are rational functions of maximal degree in $r$, coincide when $a=0$, and have order $r^{-1}$, it follows that
\begin{align*}
|\potlFI-\potlFIz|\lesssim \frac{|a|}{M}\frac{M}{r^2} .
\end{align*}
Thus, by inverting the spectral transform, 
\begin{align*}
\int_{\fullspectralspace}|\ErrorSIm|^2\diFourSpectral
&\lesssim \int_{\Exterior} |\potlFI-\potlFIz|^2|\soluCut|^2 \diFourFI\\
&\lesssim \frac{a^2}{M^2} \int_{\Exterior} \frac{M^2}{r^4} |\soluCut|^2 \diFourFI\\
&\lesssim \frac{a^2}{M^2} \int_{\fullspectralspace} \frac{M^2}{r^4} |\soluS|^2 \diFourSpectral.
\end{align*}

From the asymptotics of $\fnMorawetzSA$, $\fnMorawetzSB$, and $\fnMorawetzSC$, one finds
\begin{align*}
|\fnMorawetzS|
&\lesssim |\fnMorawetzSAA||\fnMorawetzSAB||\fnMorawetzSBA||\fnMorawetzSBB||\fnMorawetzSC|\\
&\lesssim \frac{\KDelta}{(r^2+a^2)^2} (1) (r^6)\frac{1}{r}\frac{1}{r^3} \\
&\lesssim \frac{\KDelta}{r^2+a^2} .
\end{align*}
Similarly
\begin{align*}
|\fnqS|
&\lesssim \frac{1}{r} . 
\end{align*}
Thus,
\begin{align*}
\BulkSIm
&=\int_{\fullspectralspace} (\fnMorawetzS(\pr\soluS)+\fnqS\soluS)\ErrorSIm \diFourSpectral\\
|\BulkSIm|
&\lesssim \frac{|a|}{M} \int_{\fullspectralspace}\left(M|\fnMorawetzS|^2|\pr\soluS|^2 +M|\fnqS|^2|\soluS|^2 +\frac{M}{a^2}|\ErrorSIm|^2\right)\diFourSpectral\\
&\lesssim \frac{|a|}{M} \int_{\fullspectralspace} \left(\frac{M\KDelta^2}{(r^2+a^2)^2}|\pr\soluS|^2 +\left(\frac{M}{r^2}+\frac{M^3}{r^4}\right)|\soluS|^2\right) \diFourSpectral 
%\\
%&\lesssim \frac{|a|}{M} (\BulkFILowerOrder+\BulkFITwo) 
.
\end{align*}
Since $M\lesssim r$ and $r^2\sim r^2+a^2$, the result follows. 
\end{proof}

\begin{lemma}[Energy bound for the basic spectral estimate]
\hypothesisD,
\hypothesisFsoluS,
then
\begin{align*}
|\EnergyS|
&\lesssim \EnergyCrudeFI(\FinalTime)+\EnergyCrudeFI(0) \\
&\qquad+\smallparameter \int_{\fullspectralspace} \left(\frac{M\KDelta^2}{(r^2+a^2)r^2}|\pr\soluS|^2 +\frac{M}{r^2}|\soluS|^2\right) \diFourSpectral .
\end{align*}
\end{lemma}
\begin{proof}
Because $\ErrorSCut$ is supported in $t\in[-M,0]\cup[\FinalTime,\FinalTime+M]$, one might expect that $|\EnergyS| \lesssim \sum_{t\in[-M,0]\cup[\FinalTime,\FinalTime+M]} \EnergyCrudeFI(t)$. However, because $\vecMorawetzBasic$ depends on $\spectralcollection$, one cannot localise in $t$. For this reason, we approximate $\vecMorawetzBasic$ by a $\spectralcollection$-independent vector field. This allows us to localise one part in $t$ and to gain a factor of $|a|/M$ in the remainder. 

Later in this argument, it will be useful to have the estimate
\begin{align*}
\int_{\fullspectralspace} &\frac{M^3}{r^4}\left(\frac{\KDelta}{r^2+a^2}\right)^2 |\ErrorSCut|^2 \diFourSpectral \\
&=\int_{\Exterior} \frac{M^3}{r^4}\left(\frac{\KDelta}{r^2+a^2}\right)^2 |2(\pt\solu)(\pt\ChiT)+\solu(\pt^2\ChiT)|^2 \left(\frac{\KPi}{\KDelta}\right)^2\diFourFI\\
&\lesssim \frac{1}{M} \int_{\supp \pt\ChiT} \left(|\pt\solu|^2 +\frac{1}{r^2}|\solu|^2\right) r^2 \diFourFI \\
&\lesssim \sup_{t\in[-M,0]\cup[\FinalTime,\FinalTime+M]} \EnergyCrudeFI(t) \\
&\lesssim \EnergyCrudeFI(\FinalTime)+\EnergyCrudeFI(0).
\end{align*} 
The last estimate follows from the trivial energy estimate corollary \ref{cor:TrivialEnergyBound}.

Let
\begin{align*}
\fnMorawetzSApprox&=\frac{\KDelta}{(r^2+a^2)^2} \frac{r^5}{6} (\pr\potlL) ,\\
\fnqSApprox&=\frac12\frac{\KDelta}{(r^2+a^2)^2}\pr\left(\frac{r^5}{6}(\pr\potlL)\right) .
\end{align*}
Outside the support of $\chimid$, the approximators $(\fnMorawetzSApprox,\fnqSApprox)$ are exactly equal to $(\fnMorawetzS,\fnqS)$. Inside the support of $\chimid$, the differences between the coefficients of $\epsilondtsquared^2\spectralt^2\modspectralcollection^{-2}$, $\spectralt\spectralp\modspectralcollection^{-2}$, $\spectralp^2\modspectralcollection^{-2}$, and $\spectralQ\modspectralcollection^{-2}$ all have a coefficient of $\smallparameter$. Thus, 
\begin{align*}
|\fnMorawetzS-\fnMorawetzSApprox|
&\lesssim \smallparameter \left(\frac{\KDelta}{r^2+a^2}\right)^2\frac{M^2}{r^2} ,\\
|\fnqS-\fnqSApprox|
&\lesssim \smallparameter \frac{\KDelta}{r^2+a^2}\frac{M^2}{r^3} .
\end{align*}
From this, one can estimate the error from using $(\fnMorawetzSApprox,\fnqSApprox)$ to approximate $(\fnMorawetzS,\fnqS)$ in $\ErrorSCut$. The estimate is
\begin{align*}
&\left|\EnergyS -\int_{\fullspectralspace}(\fnMorawetzSApprox(\pr\soluSb)+\fnqSApprox\soluSb)\ErrorSCut \diFourSpectral\right|\\
&\lesssim \smallparameter \int_{\fullspectralspace} \left(\frac{M\KDelta^2}{(r^2+a^2)^2}|\pr\soluS|^2 +\frac{M}{r^2}|\soluS|^2 +\frac{M^3|\ErrorSCut|^2}{r^4}\left(\frac{\KDelta}{r^2+a^2}\right)^2\right) \diFourSpectral 
%\\
%&\lesssim \smallparameter (\BulkFILowerOrder+\BulkFITwo+\EnergyCrudeFI(\FinalTime)+\EnergyCrudeFI(0)) 
.
\end{align*}
The term involving $|\ErrorSCut|^2$ is bounded by $\EnergyCrudeFI(\FinalTime)+\EnergyCrudeFI(0)$.

One can estimate the integral of the approximation by inverting the spectral transform
\begin{align*}
&\left|\int_{\fullspectralspace}(\fnMorawetzSApprox(\pr\soluSb)+\fnqSApprox\soluS)\ErrorSCut \diFourSpectral\right|\\
&\lesssim \left|\int_{\Exterior} (\fnMorawetzSApprox(\pr\soluCutb)+\fnqSApprox\soluCutb)\left(2(\vecTperp\solu)\frac{\KPi}{\KDelta}(\pt\ChiT)+\solu\frac{\KPi}{\KDelta}(\pt^2\ChiT)\right) \diFourFI\right| .
\end{align*}
All of these have a factor where at least one derivative is applied to $\ChiT$, so that they are supported near $t=0$ and $t=\FinalTime$. The integrand is the product of two factors, each of which is the sum of two terms. The terms arising from the first term in the second factor are
\begin{align*}
\left|\fnMorawetzSApprox(\pr\soluCutb)2(\vecTperp\solu)\frac{\KPi}{\KDelta}(\pt\ChiT)\right|
&\lesssim |\pr\solub||\vecTperp\solu|\frac{\KPi}{r^2+a^2}(\pt\ChiT)\\
&\lesssim \left(\frac{r^2+a^2}{\KDelta}|\vecTperp\solu|^2+\frac{\KDelta}{r^2+a^2}|\pr\solu|^2\right)r^2 (\pt\ChiT) ,\\
\left|\fnqSApprox(\soluCutb)2(\vecTperp\solu)\frac{\KPi}{\KDelta}(\pt\ChiT)\right|
&\lesssim \left( |\vecTperp\solu|^2 +\frac{1}{r^2}|\solu|^2\right)r^2(\pt\ChiT) .
\end{align*}
Thus,
\begin{align*}
\int_{\Exterior}&\left|(\fnMorawetzSApprox\pr\soluSb+\fnqSApprox\soluSb)2((\vecTperp\solu)\frac{\KPi}{\KDelta}(\pt\ChiT)\right|\diFourFI\\
&\lesssim 2M\sup_{t\in[-M,0]\cup[\FinalTime,\FinalTime+M]} \frac{\EnergyCrudeFI(t)}{M} \\
&\lesssim \EnergyCrudeFI(\FinalTime)+\EnergyCrudeFI(0). 
\end{align*}
The remaining two terms must be treated together
\begin{align*}
&\Re\left((\fnMorawetzSApprox(\pr\soluCutb)+\fnqSApprox\soluCutb)\frac{\KPi}{\KDelta}\solu(\pt^2\ChiT)\right)\\
&=\frac{\KPi}{\KDelta}\frac{\KDelta}{(r^2+a^2)^2}\left( \frac{r^5}{6}(\pr\potlL)\Re((\pr\solub)\solu)+\frac12\pr\left(\frac{r^5}{6}(\pr\potlL)\right)|\solu|^2\right)(\pt^2\ChiT) \\
&=\frac{\KPi}{(r^2+a^2)^2}\pr\left(\frac{r^5}{12}(\pr\potlL)|\solu|^2 \right)(\pt^2\ChiT) .
\end{align*}
Since $\KPi(r^2+a^2)^{-2}$ is a homogeneous rational function of order $0$ in $r$, its derivative is a homogeneous rational function of order at most $-1$, which decays at worst like $Mr^{-2}$. We now integrate this, in a purely coordinate based formalism, over the region $(t,r,\theta,\phi)\in\Reals\times(\rp,\infty)\times\Stwo$ and apply integration by parts in $r$ to find
\begin{align*}
\int_{\Exterior} &\Re\left((\fnMorawetzSApprox(\pr\soluCutb)+\fnqSApprox\soluCutb)\frac{\KPi}{\KDelta}\solu(\pt^2\ChiT)\right)\diFourFI\\
&=\int_{\Exterior} \pr\left(\frac{\KPi}{(r^2+a^2)^2}\right)\left(\frac{r^5}{12}(\pr\potlL)|\solu|^2 \right)(\pt^2\ChiT) \diFourFI \\
&\qquad +\int_{\bif} \left(\frac{\KPi}{(r^2+a^2)^2}\right)\left(\frac{r^5}{12}(\pr\potlL)|\solu|^2 \right)(\pt^2\ChiT) \di t \diomega .
\end{align*}
The first of these integrals is bounded in absolute value by $\EnergyCrudeFI(\FinalTime)+\EnergyCrudeFI(0)$ because $\pt^2\ChiT$ is supported in $[-M,0]\cup[\FinalTime,\FinalTime+M]$. One should expect that the integral over $\bif$ is identically zero, since, in the closure of the Kerr spacetime, $\bif$ is a $2$-surface, not a $3$-hypersurface. However, the function $\ChiT$ does not extend even continuously to the closure of the Kerr space-time. In particular, it has no limit on $\bif$. Thus, some care must be taken in evaluating this integral. Since $\solu$ is, by assumption, continuous on the closure of the exterior of the Kerr space-time, it follows that $\lim_{r\rightarrow\rp}\solu(t,r,\theta,\phi)$ is independent of $t$. Thus, the only $t$ dependence in the integral of $\bif$ is through the function $\pt^2\ChiT$. Since $\int_\Reals \pt^2\ChiT \di t =0$, it follows that the integral over $\bif$ is indeed zero as the naive intuition would suggest. This completes the lemma. 
\end{proof}

\begin{lemma}[Basic spectral Morawetz estimate]
\label{lemma:BasicSpectralMorawetzFI}
\hypothesisD,
\hypothesisFsoluSNoCharge,
then 
\begin{align*}
\int_{\fullspectralspace}
&\left( \frac{M\KDelta^2}{(r^2+a^2)^2}|\pr\soluS|^2
+\frac{(r-\rorbit)^2}{r^3}\modspectralcollection^2|\soluS|^2
+\frac{M}{r^2}|\soluS|^2
\right)\diFourSpectral\\
&\lesssim \EnergyCrudeFI(\FinalTime)+\EnergyCrudeFI(0)
+Ca^2\int_{\Exterior} \frac{M\KDelta}{(r^2+a^2)^2r^2}|\phipm|^2\ChiT^2 \diFourFI .
\end{align*}
\end{lemma}
\begin{proof}
This follows from combining the estimates in this subsection with the energy generation formula for the spectral transform and taking $\smallparameter$ sufficiently small. 
\end{proof}

\begin{prop}[Basic Morawetz estimate, core estimate \eqref{eq:CoreIV}]
\label{prop:BasicMorawetzFI}
\hypothesisE,
\hypothesisFsoluSNoCharge,
then
\begin{align*}
\BulkFILowerOrder+\BulkFITwo
&\lesssim \EnergyCrudeFI(\FinalTime)+\EnergyCrudeFI(0)
+C\frac{|a|}{M}\BulkMorawetzF .
\end{align*}
\end{prop}
\begin{proof}
The essential strategy is to invert the spectral transform in the previous lemma. To start, observe that since $|\rorbit-3M|\lesssim\smallparameter M$, for any choice of $\rfar$, there is a constant $C$ such that for sufficiently small $\smallparameter$, 
\begin{align*}
\left|\frac{r-\rorbit}{r}\right|\geq C\chirfar .
\end{align*}
From the previous lemma and the lower bound involving $\chirfar$, one has
\begin{align*}
\int_{\fullspectralspace}
&\left( \frac{M\KDelta^2}{(r^2+a^2)^2}|\pr\soluS|^2
+\frac{1}{r}\chirfar\modspectralcollection^2|\soluS|^2
+\frac{M}{r^2}|\soluS|^2
\right)\diFourSpectral\\
&\lesssim \EnergyCrudeFI(\FinalTime)+\EnergyCrudeFI(0)
+Ca^2\int_{\Exterior} \frac{\KDelta}{(r^2+a^2)r^2}|\phipm|^2\ChiT^2 \diFourFI .
\end{align*}
Inverting the spectral transform, one finds
\begin{align*}
\int_{\Exterior}
&\left( \frac{M\KDelta^2}{(r^2+a^2)^2}|\pr\soluCut|^2
+\frac{1}{r}\chirfar\soluCutb(-\epsilondtsquared^2\pt^2-\pp^2-\OpQ)\soluCut
+\frac{M}{r^2}|\soluS|^2
\right)\diFourSpectral\\
&\lesssim \EnergyCrudeFI(\FinalTime)+\EnergyCrudeFI(0)
+Ca^2\int_{\Exterior} \frac{\KDelta}{(r^2+a^2)r^2}|\phipm|^2\ChiT^2 \diFourFI .
\end{align*}
Applying integration by parts in the time and angular variables and using the compact support of $\soluCut$ in these variables, one finds
\begin{align*}
\int_{\Exterior}
&\left( \frac{M\KDelta^2}{(r^2+a^2)^2}|\pr\soluCut|^2
+\frac{1}{r}\chirfar(\epsilondtsquared^2|\pt\soluCut|+|\pAng\soluCut|^2)
+\frac{M}{r^2}|\soluS|^2
\right)\diFourSpectral\\
&\lesssim \EnergyCrudeFI(\FinalTime)+\EnergyCrudeFI(0)
+Ca^2\int_{\Exterior} \frac{M\KDelta}{(r^2+a^2)^2r^2}|\phipm|^2\ChiT^2 \diFourFI .
\end{align*}
Since the integrand on the left is nonnegative, it is only reduced by restricting to the interval $[0,\FinalTime]$. The values of $\epsilondtsquared M^{-1}$ and $\rfar M^{-1}$ can now be treated as fixed constants. With these fixed, $|a|\epsilondtsquared^{-1}$ can be given an upper bound in terms of $|a|M^{-1}$. 

Finally, the integral involving $|\phipm|^2$ can be estimated by $|a|M^{-1}\BulkMorawetzF$ for $t\in[0,\FinalTime]$, and by $|a|M^{-1}\EnergyCrudeF(0)$ and $|a|M^{-1}\EnergyCrudeF(\FinalTime)$ for $t\in[-M,0]$ and $[\FinalTime,\FinalTime+M]$ respectively. 
\end{proof}

\subsection{The refined Morawetz estimate for the Fackerell-Ipser equation}
The purpose of this subsection is to prove core estimate \eqref{eq:CoreV} following the ideas in \cite{BlueSoffer:LongPaper,AnderssonBlueNicolas}. Refined Morawetz estimates also appear in \cite{TataruTohaneanu}. In this section, it is convenient to imagine $\epsilondtsquared/M$ is as in proposition \ref{prop:BasicMorawetzFI}, but $\epsilondtsquared/M$ can be chosen independently. Furthermore, we take $\fnMorawetzSA=\potlL(1-\epsilondtsquared^2\potlL)$ as in the previous subsection, $\fnMorawetzSB=1$, and
\begin{align*}
\fnMorawetzSC&=M^2 \modspectralcollection^{\ExponentRefined} \arctan(\rescaledr) \chiTworfar 
\end{align*}
$\chiTworfar$ is a smooth function with support in $|r-3M|\leq 3\rfar$, identically one in $|r-3M|\leq2\rfar$, monotone in the intervals between, and for all $k\in\Naturals: \pr^k\chiTworfar\lesssim\rfar^{-k}$. The estimates in this subsection are proved independently of those in the previous subsection. 

\begin{lemma}[Bound the main bulk term for the refined Morawetz estimate]
\hypothesisF,
\hypothesisFsoluS,
then
\begin{align*}
\BulkSMain
&\geq \int_{\fullspectralspace} M^{-1}\modspectralcollection^{\ExponentRefinedThree} |\soluS|^2 \diFourSpectral\\
&\qquad-C\left(\EnergyCrudeFI(\FinalTime)+\EnergyCrudeFI(0) +\BulkFILowerOrder+\BulkFITwo\right) .
\end{align*}
\end{lemma}
\begin{proof}
The integral in $\BulkSMain$ consists of three terms, involving $\mathcal{A}$, $\mathcal{U}$, and $\mathcal{V}$. The $\mathcal{U}$ term is dominant for $\modrescaledr\geq 1$, and the $\mathcal{V}$ term is dominant for $\modrescaledr\leq 1$. The exponent $\ExponentRefined$ is chosen as in \cite{BlueSoffer:LongPaper}. Although it will not be useful in this paper, further refinements can estimate $\int M^{-1} \modspectralcollection^{2-\epsilon}|\soluS|^2 \diFourSpectral$ for any $\epsilon>0$ by using a technical argument that also uses a Fourier transform in the $r$ variable \cite{BlueSoffer:LongPaper}. 

Because of the localising factor $\chiTworfar$, it is not necessary to track asymptotic behaviour near $r=\rp$ or $r=\infty$. Similarly, since all the factors are bounded, and all the factors except $\arctan(\rescaledr)$ have bounded derivatives, the only terms that might fail to be bounded by $|\soluS|^2$ or $\chirfar\modspectralcollection^2|\soluS|^2$ arise in $\mathcal{A}$ when the arctangent is differentiated, in $\mathcal{U}$, and in $\mathcal{V}$ when $\fnqS$ is differentiated twice. By inverting the spectral transform, the integral of $\chiThreerfar M^{-1} |\soluS|^2$ can be decomposed into the intervals $[-M,0]$, $[0,\FinalTime]$, and $[\FinalTime,\FinalTime+M]$, where it is bounded by $\EnergyCrudeFI(0)$, $\BulkFILowerOrder$, and $\EnergyCrudeFI(\FinalTime)$ respectively. Similarly, the integral of $\chirfar\chiThreerfar\modspectralcollection^2|\soluS|^2$ can be bounded by $\EnergyCrudeFI(0)$, $\BulkFITwo$, and $\EnergyCrudeFI(\FinalTime)$. 

The relevant term from $\mathcal{A}$ is
\begin{align*}
\frac12M^2 \KDelta\fnMorawetzSA\chiTworfar(\pr\arctan(\rescaledr)) |\pr\soluS|^2 .
\end{align*}
This is nonnegative, since the arctangent function is increasing and all the other functions are nonnegative. 

The $\mathcal{U}$ term is
\begin{align*}
\mathcal{U}|\soluS|^2
&=-\frac12\chiTworfar\DCurlyRT\arctan(\rescaledr) |\soluS|^2.
\end{align*}
Both $-\DCurlyRT$ and $\arctan(\rescaledr)$ have a unique root at $r=\rorbit$ and go from negative to positive. Thus, from the approximation of $\DCurlyRT$ in lemma \ref{lemma:PropertiesOfDCurlyRT}, one finds
\begin{align*}
\mathcal{U}|\soluS|^2
&=\frac12M^2\chiTworfar \left|\DCurlyRT\right|\arctan(\rescaledr)| |\soluS|^2 \\
&\gtrsim M^{-1} \chiTworfar\modspectralcollection^{2-\ExponentRefined}\rescaledr\arctan(\rescaledr) . 
\end{align*}
Since $x\arctan x\geq 1$, for $|x|\geq 1$, this estimate is quite strong in that it dominates $\modspectralcollection^{2-\ExponentRefined}$ for $\modrescaledr\geq 1$. 

Because $\fnqS$ involves the derivative of $\fnMorawetzSC$, and the $\mathcal{V}$ term involves two derivatives of $\fnqS$, there are up to three derivatives of $\arctan(\rescaledr)$. (Regardless of whether the derivative is applied to $\arctan(\rescaledr)$ or some other term, all the terms in $\mathcal{V}$ have a factor of $M^{-1}$.) The $k$th derivative is of the form $M^{-1} \modspectralcollection^{\ExponentRefinedk} (1+(\rescaledr)^2)^{-(k+1)/2}$. In the remainder of this paragraph, we use $A$ dominates $B$ to mean that for any constant $C_1$ there is a $C_2$ such that $|A|\leq C_2 +B/C_1$. Using this terminology, for $\modrescaledr\leq  1$, the third derivative dominates the zeroth, first, and second derivatives, and for $\modrescaledr\geq 1$, all the terms are dominated by $\modspectralcollection^{\ExponentRefinedThree} = \modspectralcollection^{2-\ExponentRefined}$, which is the factor arising in $\mathcal{U}$. Thus, for $\rescaledr$ small, the third derivative term dominates, and for $\rescaledr$ large, all terms are dominated by the term arising from $\mathcal{U}$. Thus, 
\begin{align*}
\chiTworfar (\mathcal{U}+\mathcal{V}+CM^{-1})|\soluS|^2 &\geq \chiTworfar \modspectralcollection^{\ExponentRefinedThree}|\soluS|^2 M^{-1}.
\end{align*}
The $CM^{-1} \chiTworfar |\soluS|^2$ term can be estimated by $\EnergyCrudeFI(0)+\EnergyCrudeFI(\FinalTime)+\BulkFILowerOrder$, as explained earlier in this proof. 

Combining these, one obtains the desired result. 
\end{proof}

\begin{lemma}[Bound on the $\BulkSIm$ and $\EnergyS$ terms for the refined Morawetz estimate]
\hypothesisF,
\hypothesisFsoluS,
then
\begin{align*}
&|\BulkSIm|+|\EnergyS|\\
&\lesssim \EnergyCrudeFI(\FinalTime)+\EnergyCrudeFI(0)+\BulkFILowerOrder +\BulkFITwo \\
&\quad +\int_{\fullspectralspace} \chiTworfar M^{-1} \modspectralcollection |\soluS|^2\diFourSpectral  .
\end{align*}
\end{lemma}
\begin{proof}
Since $\fnMorawetzS$ and $\fnqS$ are both supported on $|r-\rorbit|<3\rfar$, we can introduce an additional cut-off $\chiThreerfar$ that is identically one for $|r-\rorbit|<3\rfar$, identically $0$ for $|r-\rorbit|>4\rfar$, monotone in between, and having for $k\in\Naturals:\pr^k\chiThreerfar<\rfar^{-k}$. Thus, 
\begin{align*}
|\BulkSIm|
&= \int_{\fullspectralspace} (\fnMorawetzS (\pr\soluSb)+\fnqS\soluSb)\chiThreerfar \ErrorSIm \diFourSpectral \\
&\lesssim \int_{\fullspectralspace} M^{-1}|\fnMorawetzS|^2|\pr\soluS|^2 +M^{-1}|\fnqS|^2|\soluS|^2 + M\chiThreerfar|\ErrorSIm|^2 \diFourSpectral .
\end{align*}
Similarly, 
\begin{align*}
|\EnergyS|
&= \int_{\fullspectralspace} (\fnMorawetzS (\pr\soluSb)+\fnqS\soluSb)\chiThreerfar \ErrorSCut \diFourSpectral \\
&\lesssim \int_{\fullspectralspace} M^{-1}|\fnMorawetzS|^2|\pr\soluS|^2 +M^{-1}|\fnqS|^2|\soluS|^2 +M\chiThreerfar|\ErrorSCut|^2 \diFourSpectral .
\end{align*}
The terms involving $|\fnMorawetzS|^2|\pr\soluS|^2$, $|\ErrorSIm|^2$, and $|\ErrorSCut|^2$ can be estimated using the compact support of $\chiThreerfar$ and the boundedness of $\fnMorawetzS$:
\begin{align*}
\int_{\fullspectralspace} M |\fnMorawetzS|^2|\pr\soluSb|^2 \diFourSpectral
&\lesssim \EnergyCrudeFI(\FinalTime)+\EnergyCrudeFI(0)+\BulkFITwo , \\
\int_{\fullspectralspace} \chiThreerfar|\ErrorSIm|^2\diFourSpectral 
&\lesssim \int_{\Exterior} M^{-1} \chiThreerfar |\soluCut|^2\diFourFI \\
&\lesssim \int_{\fullspectralspace} M^{-1} \chiThreerfar |\soluS|^2\diFourSpectral\\
&\lesssim \EnergyCrudeFI(\FinalTime)+\EnergyCrudeFI(0)+\BulkFILowerOrder ,\\
\int_{\fullspectralspace} M^{-1} \chiThreerfar|\ErrorSCut|^2\diFourSpectral
&\lesssim\int_{\Exterior}M^3\chiThreerfar (|\pt\ChiT|^2|\vecTperp\solu|^2 \diFourFI\\
&\quad+\int+M^3|\pt^2\ChiT|^2|\solu|^2)\diFourFI \\
&\lesssim \EnergyCrudeFI(\FinalTime)+\EnergyCrudeFI(0) .
\end{align*}
The term involving $|\fnqS|^2|\soluS|^2$ can be estimated using the bound $|\fnqS|\lesssim M^{-1}\modspectralcollection^{\ExponentRefined}\chiThreerfar$ and 
\begin{align*}
\int_{\fullspectralspace}& M|\fnqS|^2|\soluS|^2 \diFourSpectral\\
&\lesssim \int_{\fullspectralspace} M^{-1}\modspectralcollection \chiThreerfar |\soluS|^2\diFourSpectral \\
&\lesssim  \int_{\fullspectralspace} M^{-1} \modspectralcollection \chiTworfar \chiThreerfar |\soluS|^2\diFourSpectral \\
&\quad+\int_{\fullspectralspace} M^{-1} \modspectralcollection (1-\chiTworfar) \chiThreerfar |\soluS|^2\diFourSpectral
\end{align*}
Because $(1-\chiTworfar)\chiThreerfar \rfar^2\lesssim (r-\rorbit)^2$, from the Cauchy-Schwarz inequality, one finds
\begin{align*}
\int &M^{-1}(1-\chiTworfar)\chiThreerfar \modspectralcollection|\soluS|^2\diFourSpectral\\
&\leq\int M^{-1}(1-\chiTworfar)\chiThreerfar \left(|\soluS|^2 +\modspectralcollection^2|\soluS|^2\right)\diFourSpectral\\
&\lesssim \EnergyCrudeFI(\FinalTime)+\EnergyCrudeFI(0) +\BulkFILowerOrder+\BulkFITwo .
\end{align*}
\end{proof}

\begin{lemma}[Refined spectral Morawetz estimate]
\hypothesisC,
\hypothesisFsoluS,
then
\begin{align*}
\int_{\fullspectralspace} \chiTworfar \modspectralcollection |\soluS|^2\diFourSpectral 
&\lesssim \EnergyCrudeFI(\FinalTime)+\EnergyCrudeFI(0) +\BulkFILowerOrder+\BulkFITwo . 
\end{align*}
\end{lemma}
\begin{proof}
From the previous lemmas and the energy generation formula for the spectral transform, one finds
\begin{align*}
\int_{\fullspectralspace}& M^{-1} \chiTworfar \modspectralcollection^{\ExponentRefinedThree}|\soluS|^2\diFourSpectral \\
&\lesssim \EnergyCrudeFI(\FinalTime)+\EnergyCrudeFI(0)+\BulkFILowerOrder+\BulkFITwo \\
&\quad+\int_{\fullspectralspace}M^{-1} \chiTworfar\modspectralcollection|\soluS|^2\diFourSpectral  .
\end{align*}
Since there is a $C_2$ such that $\modspectralcollection \lesssim C_2 +\modspectralcollection^{\ExponentRefinedThree}/(2C)$, where $C$ is the implicit constant in the previous equation, one finds
\begin{align*}
\int_{\fullspectralspace}& M^{-1}\chiTworfar \modspectralcollection^{\ExponentRefinedThree}|\soluS|^2\diFourSpectral\\ 
&\leq C\left(\EnergyCrudeFI(\FinalTime)+\EnergyCrudeFI(0)+\BulkFILowerOrder+\BulkFITwo\right)\\
&\quad +2C_2\int_{\fullspectralspace} M^{-1}\chiTworfar|\soluS|^2\diFourSpectral  \\
&\quad +\frac12 \int_{\fullspectralspace} M^{-1} \chiTworfar\modspectralcollection^{\ExponentRefinedThree}|\soluS|^2\diFourSpectral ,
\end{align*}
from which the desired estimate follows. 
\end{proof}

\begin{prop}[Refined Morawetz estimate, core estimate \ref{eq:CoreV}]
\hypothesisC,
\hypothesisFsoluS,
then
\begin{align*}
\BulkFIOne
&\lesssim \EnergyCrudeFI(\FinalTime)+\EnergyCrudeFI(0) +\BulkFILowerOrder+\BulkFITwo .
\end{align*}
\end{prop}
\begin{proof}
Recall
\begin{align*}
\BulkFIOne
&=\int_{\rp}^\infty \chirfar\left|\int_0^\FinalTime \int_{\Stwo} \Im(\solub\pt\solu) \diomega\di t\right| \di r.
\end{align*}
In the time interval $[0,\FinalTime]$, the function $\solu$ can be replaced by $\soluCut$. The integrals of this integrand over the time intervals $[-M,0]$ and $[\FinalTime,\FinalTime+M]$ are bounded by $\EnergyCrudeFI(0)$ and $\EnergyCrudeFI(0)$ respectively, so that
\begin{align*}
\BulkFIOne
\lesssim \int_{\rp}^\infty \chirfar\left|\int_\Reals \int_{\Stwo} \Im(\soluCutb\pt\soluCut) \diomega\di t\right| \di r +\EnergyCrudeFI(\FinalTime)+\EnergyCrudeFI(0)  .
\end{align*}
Applying the spectral transform and observing that $\modspectralcollection\lesssim 1+\modspectralcollection^{\ExponentRefinedThree}$, one finds
\begin{align*}
\BulkFIOne
&\lesssim \int_{\rp}^\infty \chirfar\left|\int_{\spectralspace} M^{-1} \modspectralcollection|\soluS|^2 \dispectral\right| \di r +\EnergyCrudeFI(\FinalTime)+\EnergyCrudeFI(0) \\
& \lesssim \int_{\rp}^\infty \chirfar\left|\int_{\spectralspace} M^{-1}(1+\modspectralcollection^{\ExponentRefinedThree})|\soluS|^2 \dispectral\right| \di r\\
&\qquad +\EnergyCrudeFI(\FinalTime)+\EnergyCrudeFI(0) \\
&\lesssim \EnergyCrudeFI(\FinalTime)+\EnergyCrudeFI(0)+\BulkFILowerOrder+\BulkFITwo .
\end{align*}
\end{proof}

Now that core estimates \eqref{eq:CoreI}-\eqref{eq:CoreV} have been proved, this completes the proof of theorems \ref{thm:EnergyBound} and \ref{thm:Morawetz}. 

\subsection*{Acknowledgements}
We are grateful to Steffen Aksteiner, Dietrich H\"a{}fner, Lionel
Mason, and Jean-Philippe Nicolas for valuable discussions. We would
like to thank MSRI, where the paper was completed, for support and
hospitality. P.B. was funded by CANPDE, which was funded by an EPSRC
Science and Innovation grant and by EPSRC grant EP/J011142/1. LA was
supported in part by the Knut and Alice Wallenberg Foundation, under
a contract to the Royal Institute of Technology, Stockholm, Sweden.

\bibliographystyle{plain}

\bibliography{FourierKerrMaxwell}

\end{document}